\documentclass[aos,preprint]{imsart}

\RequirePackage{amsthm,amsmath,amsfonts,amssymb}
\RequirePackage[numbers,sort&compress]{natbib}
\RequirePackage[colorlinks,citecolor=blue,urlcolor=blue]{hyperref}
\RequirePackage{graphicx}
\usepackage{stmaryrd}

\usepackage{tikz}
\usepackage{pgfplots}
\pgfplotsset{compat=1.18}
\usetikzlibrary{matrix,arrows,external}
\startlocaldefs

\theoremstyle{plain}

\newtheorem{theorem}{Theorem}[section]
\newtheorem{lemma}[theorem]{Lemma}
\theoremstyle{definition}
\newtheorem{definition}[theorem]{Definition}

\theoremstyle{plain}
\newtheorem{proposition}[theorem]{Proposition}
\newtheorem{corollary}{Corollary}

\theoremstyle{definition}
\newtheorem{remark}[theorem]{Remark}

\newcommand{\op}{\mathrm{op}}

\DeclareMathOperator{\sign}{sign}

\newcommand{\1}{\boldsymbol{1}}

\newcommand{\ClusterRisk}{\mathcal{R}_n^{\mathrm{clust}}}

\newcommand{\ClusterRiskOn}{\mathcal{R}_n^{\mathrm{clust,on}}}
\newcommand{\ClassifRisk}{\mathcal{R}_n^{\mathrm{class}}}
\newcommand{\ClassifRiskOn}{\mathcal{R}_n^{\mathrm{class,on}}}

\newcommand{\rom}[1]{\lowercase\expandafter{\romannumeral #1\relax}}
\DeclareMathOperator*{\argmin}{arg\,min}

\newcommand{\EE}{\mathbb{E}}

\newcommand{\PP}{\mathbb{P}}

\usepackage{mathtools}

\DeclarePairedDelimiterX\Set[1]{\lbrace}{\rbrace}%
{  #1 }

\endlocaldefs

\begin{document}

\begin{frontmatter}
\title{Sharp and Adaptive Cluster Recovery in Slowly Mixing Gaussian Hidden Markov Models}
\runtitle{Cluster Recovery in Gaussian Hidden Markov Models}

\begin{aug}
\author[A]{\fnms{Ibrahim}~\snm{Kaddouri} 
\ead[label=e1]{Ibrahim.Kaddouri@warwick.ac.uk}}
\author[B]{\fnms{Mohamed}~\snm{Ndaoud}\ead[label=e2]{ndaoud@essec.edu}}
\address[A]{Department of Statistics, University of Warwick 
\printead[presep={,\ }]{e1}}

\address[B]{Department of Decisions and Operations,
ESSEC Business School \printead[presep={,\ }]{e2}}
\end{aug}

\begin{abstract}
We study cluster recovery in a two-state Gaussian hidden Markov model in the slowly mixing regime, where the transition probability $\delta$ may vanish with the sample size. Temporal persistence creates long homogeneous segments and fundamentally changes the recovery problem relative to the i.i.d. Gaussian mixture benchmark. Rather than being purely adverse, label dependence provides exploitable structure: the statistical difficulty becomes concentrated near the rare transition points, giving rise to recovery regimes governed explicitly by their frequency $\delta$.

We first characterize, up to universal constants, the oracle Bayes
risk for offline and online clustering. Both problems exhibit a
nonstandard polynomial risk regime, but the sequential constraint
induces an additional logarithmic localization cost, yielding a genuine
online/offline gap. We also study fixed lag prediction with delayed
label feedback and identify the effective memory horizon beyond which
past labels no longer improve the optimal rate.

When the signal direction is unknown and possibly high dimensional,
the problem is governed by the effective signal strength
\[
    r_n^2
    =
    \frac{\|\theta\|^4/\sigma^4}
         {\|\theta\|^2/\sigma^2+d/n}.
\]
We derive minimax lower bounds and construct fully adaptive
polynomial time procedures attaining the sharp recovery rates. In
particular, almost full recovery is possible precisely when
$r_n^2\gg\delta$. If $n\delta\gtrsim\log n$, adaptation to both the
signal $\theta$ and the transition probability $\delta$ incurs no first order loss, and
the sharp exact recovery threshold (under the Hamming loss) is
\[
    r_n^2
    =
    2\log(n\delta)\,(1+o(1)).
\]
Thus the effective complexity of the problem is governed by the number
of latent transitions, rather than by the sample size itself.
\end{abstract}

\begin{keyword}[class=MSC]
\kwdgroup[type=primary]{\kwd{62C20}
\kwd{62M05}}
\kwdgroup[type=secondary]{\kwd{62H30}}
\end{keyword}

\begin{keyword}
\kwd{hidden Markov models}
\kwd{minimax exact recovery}
\kwd{adaptive estimation}
\end{keyword}

\end{frontmatter}

\section{Introduction}

Clustering is a fundamental problem in unsupervised learning: given observations generated from a heterogeneous population, the objective is to recover the latent groups without observing their labels. A large body of work has developed algorithms and statistical guarantees for mixture models under an independence assumption; see, among many others, \cite{KSH05,SBY09,MVW17,GV19,Ndaoud18}, and the surveys in \cite{MR3889693,G21}. In many applications, however, observations arrive sequentially and the latent states exhibit substantial temporal persistence. Treating such data as i.i.d. can therefore miss an important source of statistical information.

We study this phenomenon in a two-state Gaussian hidden Markov model. The observations satisfy
\begin{equation}
\label{eq:model_hmm}
    X_i=\eta_i\theta+\sigma\xi_i,
    \qquad i=1,\dots,n,
\end{equation}
where $\theta\in\mathbb{R}^d$ is an unknown signal vector, $\sigma>0$ is the noise standard deviation, $\xi_i\stackrel{\mathrm{i.i.d.}}{\sim}\mathcal{N}(0,I_d)$, and $(\eta_i)_{i\in[n]}$ is a stationary symmetric Markov chain on $\{-1,+1\}$ with transition matrix
\[
    Q=
    \begin{pmatrix}
        1-\delta & \delta\\
        \delta & 1-\delta
    \end{pmatrix},
    \qquad 0<\delta\le \frac12.
\]
The parameter $\delta$ controls the persistence of the hidden labels. At $\delta=1/2$, the labels are independent and the model reduces to the balanced two-component Gaussian mixture. Our focus is on the slowly mixing regime, in which $\delta$ may vanish with $n$ and is not assumed to be bounded below by a positive constant.

\textbf{The regime $\delta\geq 1/n$.} Throughout the paper, 
we will often restrict attention to
$
    \delta\geq\frac1n.
$
This is a natural nondegeneracy condition. The expected number of
transitions is $(n-1)\delta$, whereas
\[
    \PP(\eta_1=\cdots=\eta_n)
    =
    (1-\delta)^{n-1}.
\]
Hence, if $\delta\lesssim1/n$, the latent path has no transition with
probability bounded away from zero. If $\delta=o(1/n)$, this
probability tends to one. On this event the sample contains a single
latent cluster and recovery, modulo the global label permutation, is
trivial. The lower bound $\delta\ge1/n$ excludes the asymptotically
transition-free regime $n\delta\to0$, but still permits paths
with only a bounded expected number of transitions.

\textbf{The noise level $\sigma$.}
For simplicity, several constructions below are stated for known
$\sigma$, and the oracle analysis is normalized to $\sigma=1$.
Estimating $\sigma$ is orthogonal to the slowly mixing structure:
the dependence of the hidden labels affects the low-rank signal
component, whereas the isotropic Gaussian noise level can be estimated
by standard moment or spectral methods. Thus a consistent estimator
of $\sigma$ may be plugged into the procedures below without changing
the rates or phase transitions. We therefore treat $\sigma$ as known
and focus on the difficulties created by the latent Markov structure.

Slow mixing changes the geometry of the recovery problem. A typical latent trajectory consists of long constant stretches separated by relatively few transition points. As a consequence, clustering is no longer governed solely by the overlap of two Gaussian components. Away from a transition, many neighboring observations carry information about the same hidden state, whereas observations near transition points are the most difficult to assign correctly. This suggests the heuristic decomposition
\[
    \text{Clustering risk}
    \approx
    \text{Transition frequency}
    \times
    \text{Localization error per transition}.
\]
This connection with change-point localization is one of the main structural themes of the paper. Related Gaussian change-point problems are studied, for example, in \cite{VFLR20}, but the present setting differs in that both the number and the locations of the changes are random and generated by a Markov chain, and the loss is a global clustering loss modulo label permutation.

We first study an oracle problem in which the model parameters are known. This isolates the intrinsic effect of temporal dependence from the difficulty of estimating the signal direction. We consider both offline clustering, in which the full trajectory $(X_1,\dots,X_n)$ is available before any decision is made, and online clustering, in which assignments are made sequentially and cannot be revised. The two settings exhibit the same weak-signal threshold, but they differ substantially in the intermediate signal regime. Offline access to observations on both sides of a potential transition removes the additional delay cost that appears in the online problem. We also introduce a \textit{fixed-lag online prediction} problem in which labels are revealed after a fixed lag. This auxiliary model provides a useful benchmark for understanding how much information is carried by recent historical labels and plays a role in the analysis of the online lower bound.

We then turn to the statistically more demanding setting in which $\theta$ is unknown and may be high-dimensional. For the minimax class
\[
    \Theta(\Delta)=\{\theta\in\mathbb{R}^d:\|\theta\|\ge \Delta\},
\]
the effective signal is
\begin{equation}\label{eq:rn}
    r_n^2 = \frac{\Delta^4/\sigma^4}{\Delta^2/\sigma^2+d/n},
\end{equation}

which is the same high-dimensional signal parameter that appears in the i.i.d. two-component Gaussian mixture \cite{Ndaoud18}. The Markov structure enters through $\delta$: our lower bounds identify $r_n^2\asymp\delta$ as the onset of nontrivial recovery, rather than the constant-order threshold of the i.i.d. model. We complement these lower bounds with adaptive procedures. In particular, we construct estimators of the signal magnitude and an initial label direction, show that $\delta$ can be estimated up to a constant factor in the recovery regime, and develop two polynomial-time recovery procedures: an iterative estimator that closely mirrors the oracle construction, and a one-shot sample-splitting estimator that attains the minimax rate.


The main contributions can be summarized as follows.
\begin{enumerate}
    \item We characterize the oracle Bayes risk for offline clustering in a slowly mixing Gaussian HMM and identify the weak, intermediate, and strong-signal regimes. In particular, temporal persistence lowers the threshold for almost full recovery from constant order in the i.i.d. model to the transition scale $\delta$.

    \item We characterize the corresponding online Bayes risk. The sequential constraint creates an additional logarithmic localization cost in the intermediate regime, revealing a genuine online/offline gap.

    \item We analyze the fixed-lag online prediction problem and identify a critical lag separating a memory-dominated regime, in which a recently revealed label is the most useful source of information, from an observation dominated regime, in which sufficiently old labels no longer improve the optimal rate.

    \item  In the unknown direction high-dimensional problem, we derive
    minimax lower bounds and fully adaptive polynomial-time procedures.
    Almost full recovery occurs exactly when $r_n^2\gg\delta$ and when $n\delta\ge C_0\log n$ for a
sufficiently large universal constant $C_0$, the sharp first-order exact-recovery
    threshold is $r_n^2=2\log(n\delta)(1+o(1))$. The adaptive ingredients include signal strength estimation, a spectral label initializer, and constant factor estimation of the transition probability $\delta$.
\end{enumerate}


\section{Related literature}
\label{sec:related-literature}

Our work lies at the intersection of clustering in mixture models, inference in hidden Markov models, and change-point detection and localization.

\textbf{Clustering in Gaussian mixtures.}
The balanced two-component Gaussian mixture is the natural independent
benchmark for label recovery. A broad literature studies computationally
efficient procedures and sharp statistical guarantees, including spectral
methods \cite{KSH05,SBY09}, Lloyd-type algorithms
\cite{LH16}, semidefinite relaxations
\cite{MVW17,GV19,Royer17,Perchet18}, and related
methods. In particular, \cite{Ndaoud18} gives a sharp minimax analysis
under Hamming loss modulo the global label permutation. For minimax classes
satisfying $\|\theta\|\geq\Delta$, the relevant effective signal is
$r_n^2$ defined in~\eqref{eq:rn}. The same quantity governs
our high-dimensional lower and upper bounds. The essential difference is
that the labels are no longer independent: their persistence introduces the
additional scale $\delta$ and converts much of the recovery problem into
localization around random transition times.

\textbf{Inference and decision problems in hidden Markov models.}
Hidden Markov models provide a canonical framework for inference with
dependent latent states; see, for example, \cite{CMT05}. Classical
statistical theory has largely focused on estimation of the model
parameters. Consistency and asymptotic normality of likelihood-based
estimators were established under suitable ergodicity and forgetting
conditions by, among others, \cite{BickelRitovRyden1998} and
\cite{DoucMoulinesRyden2004}. Nonparametric identification and maximum
likelihood estimation for finite-state hidden Markov models were studied
by \cite{AlexandrovichHolzmannLeister2016}.

A related decision-theoretic literature exploits latent Markov dependence
for structured inference. For instance, \cite{SunCai2009} considers
large-scale multiple testing under a two-state hidden Markov model and
constructs oracle and asymptotically optimal data-driven procedures that
exploit dependence between the latent states. The inferential target there
is multiple testing rather than recovery of the entire latent sequence
under clustering loss.

The closest work to ours is \cite{GKN23}, which studies clustering
in a general nonparametric HMM framework. Their results are tailored to
strongly mixing chains, with transition probabilities bounded away from
zero. The regime considered here is complementary: $\delta$ may vanish
with $n$, so the usual mixing and forgetting constants are no longer
uniform and the dependence on $\delta$ becomes part of the statistical
difficulty itself. This leads to polynomial-risk regimes and recovery
thresholds that disappear when the mixing parameters are treated as
fixed constants. Our focus is also different from classical parameter
estimation: the primary target is sharp recovery of the latent state
sequence.

\textbf{Change-point detection and localization.}
When $\delta\ll1$, a typical hidden path is piecewise constant with relatively few changes,
which creates a close connection with change-point problems. Minimax detection and
localization rates for Gaussian sequences have been studied extensively; see, for example,
\cite{VFLR20} and the review \cite{yu2020}. In high dimensions,
projection-based reductions and sparse change-point detection are particularly relevant to our
setting. In \cite{WangSamworth2018} an informative sparse projection direction is estimated before
reducing the problem to univariate change-point estimation, while
\cite{chao} characterize the minimax testing rates for detecting a single sparse
change in a high-dimensional Gaussian mean vector. Finite-sample guarantees for detection and
localization of changes in high-dimensional mean vectors were obtained by
\cite{YuChen2021}. Change-point segmentation under temporal dependence has also been
considered in more general time-series settings; see, for example,
\cite{ZhaoJiangShao2022}.

The analogy with change-point analysis is useful but not exact. In our setting, both the number
and the locations of the changes are random and generated by the latent Markov chain; the
signal direction may be unknown and high-dimensional; and the loss concerns global recovery
of the latent sequence modulo a label permutation rather than detection or localization of a
fixed collection of deterministic change-points. In particular, the transition frequency
$\delta$ plays simultaneously the role of a stochastic change-point intensity and a fundamental
scale in the recovery risk.

\textbf{Estimation of the emission parameter.}
A complementary line of work concerns estimation of $\theta$ up to sign.
In the Gaussian location model, corresponding to perfectly constant
labels, the minimax estimation rate is $\sqrt{d/n}$; see \cite{wu2017lecture}.
In the i.i.d.\ Gaussian mixture model, \cite{WZ21} establish minimax
rates for estimating $\theta$ and analyze spectral and EM-type procedures.
For Markov-dependent labels, \cite{ZW22} study estimation
over a range of transition probabilities and show that temporal dependence
can improve estimation. A spectral procedure from that work was revisited
in \cite{KN24}, where an adaptive bucketed spectral
estimator was developed. These estimation results provide useful
ingredients for our adaptive clustering procedures, but the statistical
target here is different: the oracle part of the paper isolates the
intrinsic label-recovery difficulty, while the minimax part quantifies the
additional cost of learning an unknown high-dimensional signal direction.

\section{Setting and definitions}
\label{sec:setting_def}
\subsection{Notation}
\label{sec:notation}

We write $\mathbb N=\{1,2,\ldots\}$ and $[n]=\{1,\ldots,n\}$.  For
integers $i\le j$, let
\[
    X_{i:j}=(X_i,\ldots,X_j),
    \qquad
    \eta_{i:j}=(\eta_i,\ldots,\eta_j), \qquad \llbracket i,j\rrbracket=\{i,i+1,\ldots,j\}.
\]

For a vector $x$, $\|x\|$ denotes its Euclidean norm and
$\langle x,y\rangle=x^\top y$ the Euclidean inner product.  For a matrix
$A$, we write $\|A\|_{\mathrm{op}}$ for its operator
norm,  and $A^\top$ for its transpose.  If $A$ is
symmetric, $\lambda_{\max}(A)$ denotes its largest eigenvalue and
$v_{\max}(A)$ denotes a corresponding unit eigenvector. The symbols $\lfloor x\rfloor$ and $\lceil x\rceil$ denote the floor and
ceiling of $x$.

Probabilities and expectations under the model with parameters
$(\theta,\delta,\sigma)$ are denoted by
$\mathbb P_{\theta,\delta,\sigma}$ and
$\mathbb E_{\theta,\delta,\sigma}$ when it is useful to emphasize the
underlying law.  When the parameters are clear from context, we simply write
$\mathbb P$ and $\mathbb E$.  When the noise level is normalized to $\sigma=1$,
we may write $\mathbb P_{\theta,\delta}$ and $\mathbb E_{\theta,\delta}$ if the
parameter dependence needs to be displayed.

We write $Z_n=o_{\PP}(1)$ for convergence to zero in probability along
the parameter sequence under consideration.
For nonnegative quantities $a$ and $b$, the notation $a\lesssim b$ means
that $a\le Cb$ for a universal constant $C>0$, and $a\asymp b$ when $a\lesssim b$ and $b\lesssim a$.  Unless stated
otherwise, these constants are independent of $n,d,\theta,\delta,\sigma$ and
of any lag or block length.  Constants denoted by $c,c',c_0,C,C_0$, possibly
with different values at different occurrences, are positive and universal.

We denote by $\mathcal P_2[n]$ the set of partitions of $[n]$ into at
most two nonempty blocks.
Finally, 
we denote by $\Phi$ the distribution function
of a standard Gaussian random variable and $\Phi^c(x)=1-\Phi(x)$.

\subsection{Offline clustering}

The latent sequence $\eta_{1:n}\in\{-1,+1\}^n$ induces the partition
\[
    \Pi_n
    \coloneqq
    \Big\{\{i\in[n]:\eta_i=+1\},\;\{i\in[n]:\eta_i=-1\}\Big\}
    \setminus\{\varnothing\}.
\]
Thus, two indices belong to the same block of $\Pi_n$ if and only if their latent labels agree. The offline clustering problem is to recover $\Pi_n$ from the full observation sequence $X_{1:n}$.

\begin{definition}[Clusterer]
An $n$-clusterer is a measurable map
\[
    g:(\mathbb{R}^d)^n\to \mathcal{P}_2[n].
\]
We denote by $\mathcal{G}_n$ the set of all $n$-clusterers.
\end{definition}

For $\Pi,\Pi'\in\mathcal{P}_2[n]$, define the maximum-overlap loss
\[
    \ell_{\mathrm{cl}}(\Pi,\Pi')
    \coloneqq
    1-\frac{1}{n}\max_M\sum_{\{C,C'\}\in M}|C\cap C'|,
\]
where the maximum is over all matchings $M$ between the blocks of $\Pi$ and $\Pi'$; each block of either partition may appear in at most one matched pair. This loss is invariant under permutations of the cluster labels; see \cite{meila2001experimental,meila2005comparing}.

The clustering risk of $g\in\mathcal{G}_n$ is
\begin{equation}
\label{eq:risk-clustering-offline}
    \ClusterRisk(\theta,\delta,\sigma,g)
    \coloneqq
    \EE_{\theta,\delta,\sigma}
    \bigl[\ell_{\mathrm{cl}}(g(X_{1:n}),\Pi_n)\bigr].
\end{equation}
When the model parameters are known, the corresponding oracle Bayes clustering risk is
\[
    \inf_{g\in\mathcal{G}_n}\ClusterRisk(\theta,\delta,\sigma,g).
\]

It is convenient to formulate the same problem in terms of binary label prediction. This representation makes the global sign ambiguity explicit.

\begin{definition}[Classifier]
An $n$-classifier is a measurable map
\[
    h:(\mathbb{R}^d)^n\to\{-1,+1\}^n.
\]
We denote by $\mathcal{H}_n$ the set of all $n$-classifiers.
\end{definition}

Define $\pi_n:\{-1,+1\}^n\to\mathcal{P}_2[n]$ by
\[
    \pi_n(z_{1:n})
    \coloneqq
    \Big\{\{i\in[n]:z_i=+1\},\;\{i\in[n]:z_i=-1\}\Big\}
    \setminus\{\varnothing\}.
\]
Every classifier $h\in\mathcal{H}_n$ induces the clusterer $\pi_n\circ h$. Conversely, because $\mathcal{P}_2[n]$ contains only partitions with at most two blocks, every $n$-clusterer admits such a representation after an arbitrary labeling of its blocks; the two possible labelings differ by a global sign.

The classification risk of $h\in\mathcal{H}_n$ is
\begin{equation}
\label{eq:risk-classif:slow}
    \ClassifRisk(\theta,\delta,\sigma,h)
    \coloneqq
    \EE_{\theta,\delta,\sigma}\Bigg[
        \frac{1}{n}\sum_{i=1}^n
        \mathbf{1}\bigl\{h_i(X_{1:n})\neq\eta_i\bigr\}
    \Bigg].
\end{equation}
For every $h\in\mathcal{H}_n$, the clustering risk of the induced partition can be written as
\begin{equation}
\label{eq:risk-clustering:classif:slow}
    \ClusterRisk(\theta,\delta,\sigma,\pi_n\circ h)
    =
    \EE_{\theta,\delta,\sigma}\Big[\ell\left(h(X_{1:n}), \eta_{1:n}\right)
    \Big]
\end{equation}
where
\[
\ell\left(\widehat{\eta}, \eta\right) := \min_{s\in\{-1,+1\}}
        \frac{1}{n}\sum_{i=1}^n
        \mathbf{1}\bigl\{\widehat{\eta}_i\neq s\eta_i\bigr\}.
\]
Thus, the clustering loss is exactly the Hamming classification loss modulo a global sign. The oracle Bayes classification risk is
\[
    \inf_{h\in\mathcal{H}_n}\ClassifRisk(\theta,\delta,\sigma,h).
\]
The classification risk depends on the chosen orientation of the two latent states, whereas the clustering risk in \eqref{eq:risk-clustering:classif:slow} does not.

\subsection{Online clustering}
\label{sec:online-clustering}

We next consider the sequential setting, in which the assignment of index $i$ must be made when $X_i$ arrives and cannot subsequently be revised. At each horizon $n$, an online rule therefore produces a partition of $[n]$ from $X_{1:n}$, subject to consistency across horizons.

For $\Pi\in\mathcal{P}_2[n]$ and $1\le m\le n$, define the restriction of $\Pi$ to $[m]$ by
\[
    \Pi|_{[m]}
    \coloneqq
    \{C\cap[m]:C\in\Pi\}\setminus\{\varnothing\}.
\]

\begin{definition}[Online clusterer]
\label{def:online-clusterer}
An online clusterer is a sequence $\mathfrak{g}=(g^{(n)})_{n\ge1}$, with $g^{(n)}\in\mathcal{G}_n$, such that
\begin{equation}
\label{eq:partitions:consistency}
    g^{(n)}(x_{1:n})|_{[m]}
    =
    g^{(m)}(x_{1:m})
\end{equation}
for every $1\le m\le n$ and every $x_{1:n}\in(\mathbb{R}^d)^n$. We denote by $\mathcal{G}^{\mathrm{on}}$ the set of all online clusterers.
\end{definition}

For a fixed horizon $n\ge1$, the online clustering risk of $\mathfrak{g}\in\mathcal{G}^{\mathrm{on}}$ is
\begin{equation}
\label{eq:risk-clustering-online}
    \ClusterRiskOn(\theta,\delta,\sigma,\mathfrak{g})
    \coloneqq
    \ClusterRisk(\theta,\delta,\sigma,g^{(n)}),
\end{equation}
and the oracle Bayes online clustering risk is
\[
    \inf_{\mathfrak{g}\in\mathcal{G}^{\mathrm{on}}}
    \ClusterRiskOn(\theta,\delta,\sigma,\mathfrak{g}).
\]
The only distinction from offline clustering is the consistency constraint \eqref{eq:partitions:consistency}: the terminal rule $g^{(n)}$ must arise from a sequential procedure whose earlier assignments are irrevocable.

The corresponding classifier formulation imposes consistency on the labels themselves.

\begin{definition}[Online classifier]
\label{def:online-classifier}
An online classifier is a sequence $\mathfrak{h}=(h^{(n)})_{n\ge1}$, with $h^{(n)}\in\mathcal{H}_n$, such that
\begin{equation}
\label{eq:classifiers:consistency}
    h_i^{(n)}(x_{1:n})
    =
    h_i^{(m)}(x_{1:m})
\end{equation}
for every $1\le i\le m\le n$ and every $x_{1:n}\in(\mathbb{R}^d)^n$. We denote by $\mathcal{H}^{\mathrm{on}}$ the set of all online classifiers.
\end{definition}

Each $\mathfrak{h}\in\mathcal{H}^{\mathrm{on}}$ induces an online clusterer
\[
    \pi\circ\mathfrak{h}
    \coloneqq
    (\pi_n\circ h^{(n)})_{n\ge1}.
\]
For a fixed horizon $n\ge1$, its online classification risk is
\begin{equation}
\label{eq:risk-classif:online}
    \ClassifRiskOn(\theta,\delta,\sigma,\mathfrak{h})
    \coloneqq
    \ClassifRisk(\theta,\delta,\sigma,h^{(n)})
    =
    \EE_{\theta,\delta,\sigma}\Bigg[
        \frac{1}{n}\sum_{i=1}^n
        \mathbf{1}\bigl\{h_i^{(n)}(X_{1:n})\neq\eta_i\bigr\}
    \Bigg].
\end{equation}
By \eqref{eq:risk-clustering:classif:slow},
\begin{equation}
\label{eq:risk-clustering:classif-online}
    \ClusterRiskOn(\theta,\delta,\sigma,\pi\circ\mathfrak{h})
    =
    \EE_{\theta,\delta,\sigma}\Bigg[
        \min_{s\in\{-1,+1\}}
        \frac{1}{n}\sum_{i=1}^n
        \mathbf{1}\bigl\{h_i^{(n)}(X_{1:n})\neq s\eta_i\bigr\}
    \Bigg].
\end{equation}
The oracle Bayes online classification risk is
\[
    \inf_{\mathfrak{h}\in\mathcal{H}^{\mathrm{on}}}
    \ClassifRiskOn(\theta,\delta,\sigma,\mathfrak{h}).
\]

\subsection{Fixed-lag online prediction}
\label{sec:constant-lag-framework}

Finally, we consider an online prediction problem with delayed label feedback. Fix an integer $k\ge 1$. At time $i\ge k$, the statistician observes $X_{1:i}$ together with the revealed labels $\eta_{0:i-k}$ and must predict the current state $\eta_i$. Thus, the $k$ most recent states $\eta_{i-k+1:i}$ remain hidden, and future observations $X_{i+1},X_{i+2},\ldots$ are unavailable.

A $k$-lag predictor is a sequence $\widehat\eta=(\widehat\eta_i)_{i\ge k}$ such that each $\widehat\eta_i$ is measurable with respect to
\[
    \sigma(X_{1:i},\eta_{0:i-k}).
\]
This is a supervised prediction problem rather than a clustering problem, since some past labels are observed. We use it as an auxiliary model that isolates the statistical value of historical supervision under slow mixing.

When the model parameters are known, stationarity and the Markov property reduce the optimal risk to a stationary problem. Conditional on the most recently revealed state $\eta_{i-k}$, observations and labels preceding time $i-k$ contain no further information about $\eta_i$ beyond the intervening observations $X_{i-k+1:i}$. Define
\[
    \mathcal{H}^{\mathrm{lag}}_k
    \coloneqq
    \Bigl\{
        \psi:(\mathbb{R}^d)^k\times\{-1,+1\}\to\{-1,+1\}:
        \psi\text{ is measurable}
    \Bigr\}.
\]
The oracle Bayes fixed-lag online prediction risk is
\begin{equation}
\label{eq:risk-lag}
    \mathcal{R}^{\mathrm{lag}}_k(\theta,\delta,\sigma)
    \coloneqq
    \inf_{\psi\in\mathcal{H}^{\mathrm{lag}}_k}
    \PP_{\theta,\delta,\sigma}
    \bigl(\psi(X_{1:k},\eta_0)\neq\eta_k\bigr).
\end{equation}
Thus, $\mathcal{R}^{\mathrm{lag}}_k(\theta,\delta,\sigma)$ is the minimum error probability for predicting the current state from the most recently revealed state and the intervening block of observations.

\section{Main results}

We begin with the oracle problem, in which the signal vector and the transition probability are known. The delayed-label model is analyzed first because it isolates the value of recent label information and provides a useful benchmark for the online lower bound. We then compare online and offline clustering. The final part of the section removes oracle knowledge and develops adaptive high-dimensional procedures.

Throughout the oracle analysis, we normalize the noise level to $\sigma=1$. The oracle
results for an arbitrary $\sigma>0$ follow from their normalized counterparts by replacing $\theta$ with $\theta/\sigma$. Accordingly, we suppress $\sigma$ from the notation for the oracle risks.
\subsection{Fixed-lag online prediction}\label{sec:semi-supervised}

We begin by studying the fixed-lag online prediction risk
$\mathcal{R}^{\mathrm{lag}}_k(\theta,\delta)$ defined in
\eqref{eq:risk-lag}. We first establish a lower bound capturing the
three relevant signal regimes. We then construct explicit predictors
attaining the corresponding rates.

\begin{proposition}\label{prop:lower:lag}
Assume \(0<\delta\le 1/2\). For every \(k\ge1\),
\[
\mathcal R_k^{\mathrm{lag}}(\theta,\delta)
\gtrsim
\begin{cases}
\displaystyle
k\delta \wedge 1,
&
\displaystyle
\|\theta\|^2
\le
\frac{2\log(1/(k\delta))}{k}\wedge\frac{\log(1/\delta)}{2},
\\[1.4em]
\displaystyle
\frac{\delta}{\|\theta\|^2}
\log\!\left(1+\frac{\|\theta\|^2}{\delta}\right),
&
\displaystyle
\frac{2\log(1/(k\delta))}{k}
\le
\|\theta\|^2
\le
\frac{\log(1/\delta)}{2},
\\[1.4em]
\displaystyle
\delta\,
{\Phi}^c
\!\left(
\|\theta\|
-
\frac{\log(1/\delta)}{2\|\theta\|}
\right),
&
\displaystyle
\|\theta\|^2
\ge
\frac{\log(1/\delta)}{2}.
\end{cases}
\]
\end{proposition}
We now describe matching predictors. In the weak signal regime $\|\theta\|^2\le2\log(1/(k\delta))/k\wedge\log(1/\delta)/2$, propagate the revealed label:
\[
    \widehat\eta_k=\eta_0.
\]
In the strong-signal rule $\|\theta\|^2\ge\log(1/\delta)/2$, use a plug-in recursion based on the one-step Bayes rule
\[
    \widehat\eta_1
    =
    \sign\left(
        \langle X_1,\theta\rangle
        +
        \eta_0\frac{\log((1-\delta)/\delta)}{2}
    \right),
\]
and, for $2\leq i\leq k$,
\[
    \widehat\eta_i
    =
    \sign\left(
        \langle X_i,\theta\rangle
        +
        \widehat\eta_{i-1}
        \frac{\log((1-\delta)/\delta)}{2}
    \right).
\]
Finally, in the remaining regime, use the moving-average predictor:
\[
    \widehat\eta_k
    =
    \sign\left\langle
        \frac1m
        \sum_{j=k-m+1}^k X_j,\theta
    \right\rangle.
\]
where
\[
    m
    =
    k\wedge
    \left\lceil
        \frac{2}{\|\theta\|^2}
        \log\left(1+\frac{\|\theta\|^2}{\delta}\right)
    \right\rceil
\]
These rules attain the corresponding upper rates.
 
\begin{proposition}\label{prop:upper:lag}
Assume \(0<\delta\le 1/2\). For all $k\geq 1$, 
\[
\PP_{\theta, \delta}(\hat{\eta}_k \neq \eta_k)
\lesssim
\begin{cases}
k\delta \wedge 1,
&
\displaystyle
\|\theta\|^2
\le
\frac{2\log(1/(k\delta))}{k}\wedge\frac{\log(1/\delta)}{2},
\\[1.2em]
\displaystyle
\frac{\delta}{\|\theta\|^2}\log\left(1+\frac{\|\theta\|^2}{\delta}\right),
&
\displaystyle
\frac{2\log(1/(k\delta))}{k}
\le
\|\theta\|^2
\le
\frac{\log(1/\delta)}{2},
\\[1.2em]
\displaystyle
\delta
\Phi^c
\left(
\|\theta\|-\frac{\log(1/\delta)}{2\|\theta\|}
\right),
&
\displaystyle
\|\theta\|^2 \geq \frac{\log(1/\delta)}{2}.
\end{cases}
\]
\end{proposition}

Combining Propositions~\ref{prop:lower:lag} and
\ref{prop:upper:lag} yields a concise characterization of the optimal
fixed-lag risk.
Define

$$
    k^\star
    =
    \left\lceil
        \frac{2}{\|\theta\|^2}
        \log\!\left(1+
            \frac{\|\theta\|^2}{2\delta}
        \right)
    \right\rceil.
$$

As the following theorem shows, $k^\star$ acts as an effective memory
horizon.

\begin{theorem}\label{thm:lag:risk}
For all \(k\ge 1\),
\[
\mathcal R_k^{\mathrm{lag}}(\theta,\delta)
\asymp
\begin{cases}
\displaystyle
\left(k\wedge \frac1\delta\right)\delta,
&
\displaystyle
\|\theta\|^2<\delta,
\\[1.2em]
\displaystyle
\left(k\wedge k^\star\right)\delta,
&
\displaystyle
\delta\le \|\theta\|^2\le \frac{\log(1/\delta)}{2},
\\[1.2em]
\displaystyle
\delta\,\Phi^c\!\left(
\|\theta\|-\frac{\log(1/\delta)}{2\|\theta\|}
\right),
&
\displaystyle
\|\theta\|^2\ge \frac{\log(1/\delta)}{2}.
\end{cases}
\]
\end{theorem}

Theorem \ref{thm:lag:risk} describes how delayed supervision competes with the direct evidence contained
in the observations. In the weak and intermediate signal regimes, the optimal risk increases with
the lag \(k\) until the effective horizon is reached. Before this horizon, the most recently revealed
state remains sufficiently close to the target that propagating it is rate-optimal. Beyond this horizon, the optimal rate is attained by a predictor using only a recent observation window, without using the revealed historical label.

More precisely, when $
    \delta\le \|\theta\|^2\le \frac{\log(1/\delta)}{2},
$
the rate in Theorem~\ref{thm:lag:risk} is attained by the predictor
\begin{equation}\label{eq:lag:predictor:intermediate}
\widehat\eta_i
=
\begin{cases}
\eta_{i-k},
&
\text{if } k<k^\star,
\\[1.2em]
\displaystyle
\operatorname{sign}\!\left(
\left\langle
\frac1{k^\star}\sum_{j=i-k^\star+1}^{i}X_j,\theta
\right\rangle
\right),
&
\text{if } k\ge k^\star.
\end{cases}
\end{equation}
In the intermediate-signal regime,
\[
k^\star
\asymp
\frac1{\|\theta\|^2}
\log\left(1+\frac{\|\theta\|^2}{\delta}\right).
\]
Thus, \(k^\star\) is the time required for the accumulated observational evidence to match the
information carried by a recently revealed state. Labels older than order \(k^\star\) therefore provide no
further improvement in the optimal rate.

We now turn to online clustering. The sequential constraint prevents retrospective correction of an earlier assignment, so errors created near a transition can persist until enough evidence has accumulated. This produces an additional localization cost relative to the offline problem. The following theorem gives the lower bound.

\begin{theorem}
\label{Bayes_risk_lb_online}
Assume that $\delta \geq C_0 n^{-1/10}$ for $C_0$ a sufficiently large universal constant. The oracle Bayes risk of online clustering satisfies

$$
\inf_{\mathfrak{g}\in \mathcal{G}^{\mathrm{on}}} \ClusterRiskOn(\theta, \delta, \mathfrak{g}) \gtrsim
\begin{cases}
\displaystyle
1,
&
\displaystyle
\|\theta\|^2\leq2\delta,
\\[1.2em]
\displaystyle
\frac{\delta}{\lVert\theta\rVert^{2}}\left(\log\left(\frac{\lVert\theta\rVert^{2}}{2\delta}\right) + 1\right),
&
\displaystyle
2\delta
\leq
\|\theta\|^2
\leq
\frac{\log(1/\delta)}{2},
\\[1.2em]
\displaystyle
\delta\,
\Phi^c\!\left(
\|\theta\|
-\frac{\log(1/\delta)}{2\|\theta\|}
\right),
&
\displaystyle
\|\theta\|^2
\geq
\frac{\log(1/\delta)}{2}.
\end{cases}
$$
\end{theorem}

The condition on $\delta$ is used only to handle the additional minimization over global label permutations that appears in the definition of the Bayes clustering risk. More precisely, it allows us to establish the equivalence, up to universal constants, between the oracle Bayes risks for clustering and classification. This equivalence was proved for offline clustering in Theorem~9 of \cite{GKN23}, and the same argument extends to the online setting. In our subsequent analysis of offline clustering, we will weaken this restriction on $\delta$.

We now define a matching online construction. If
\(\|\theta\|^2<\frac12\log(1/\delta)\) and \(\|\theta\|^2\leq2\delta\),
set \(\widehat\eta_i=1\) for every \(i\).
In the intermediate regime
\(
    2\delta<\|\theta\|^2<\frac12\log(1/\delta),
\)
choose
\[
    k=
    \left\lceil
        \frac{2}{\|\theta\|^2}
        \left(
            \log\left(\frac{\|\theta\|^2}{2\delta}\right)+1
        \right)
    \right\rceil.
\]
and set
\begin{equation}
\label{online_procedure}
    \widehat\eta_i
    =
    \begin{cases}
        \operatorname{sign}(\langle X_i,\theta\rangle),
        & i\leq k,\\[0.8em]
        \displaystyle
        \operatorname{sign}\!\left(
            \left\langle
                \frac{1}{k}
                \sum_{j=i-k+1}^{i}X_j,
                \theta
            \right\rangle
        \right),
        & i>k.
    \end{cases}
\end{equation} 
In the strong regime \(\|\theta\|^2\geq\frac12\log(1/\delta)\), set
        \begin{equation}
    \hat{\eta}_{i} = 
    \begin{cases}
        \operatorname{sign}\!\left(\langle X_{i}, \theta\rangle\right) 
        & i =1 \\[1ex]
        \operatorname{sign}
\left(
\langle X_i , \theta \rangle
+
\widehat\eta_{i-1}
\frac{\log((1-\delta)/\delta)}{2}
\right) 
        & i > 1
    \end{cases}
    \end{equation}
    Let \(\mathfrak h\) be the online classifier determined by these rules,
and let \(\mathfrak g=\pi\circ\mathfrak h\) be its induced online
clusterer. The next theorem controls its clustering risk.
\begin{theorem}\label{Bayes_risk_ub_online}
Assume that \(1/n\leq\delta\leq1/2\). There exists a universal
constant \(C>0\) such that
\[
\ClusterRiskOn(\theta,\delta,\mathfrak g)
\leq C
\begin{cases}
1,
& \displaystyle
\|\theta\|^2\leq2\delta
\quad\text{and}\quad
\|\theta\|^2<\frac12\log(1/\delta),
\\[1.2em]
\displaystyle
\frac{\delta}{\|\theta\|^2}
\left(\log\left(\frac{\|\theta\|^2}{2\delta}\right)+1\right),
& \displaystyle
2\delta<\|\theta\|^2<\frac12\log(1/\delta),
\\[1.2em]
\displaystyle
\delta\Phi^c\left(
\|\theta\|-\frac{\log(1/\delta)}{2\|\theta\|}
\right),
& \displaystyle
\|\theta\|^2\geq\frac12\log(1/\delta).
\end{cases}
\]
\end{theorem}

Together, Theorem~\ref{Bayes_risk_lb_online} and Theorem~\ref{Bayes_risk_ub_online} characterize the online oracle risk up to universal multiplicative constants. The distinctive feature is the intermediate regime
\[
    2\delta<\|\theta\|^2\le \frac{\log(1/\delta)}{2},
\]
where the risk decays polynomially, with an additional logarithmic factor reflecting the delay required to detect a transition from one-sided data. Figure~\ref{fig:Bayes_risk:online:slow_mixing} summarizes the resulting phase diagram.

\begin{figure}[hbt!]
\centering
\includegraphics[width=0.99\textwidth]{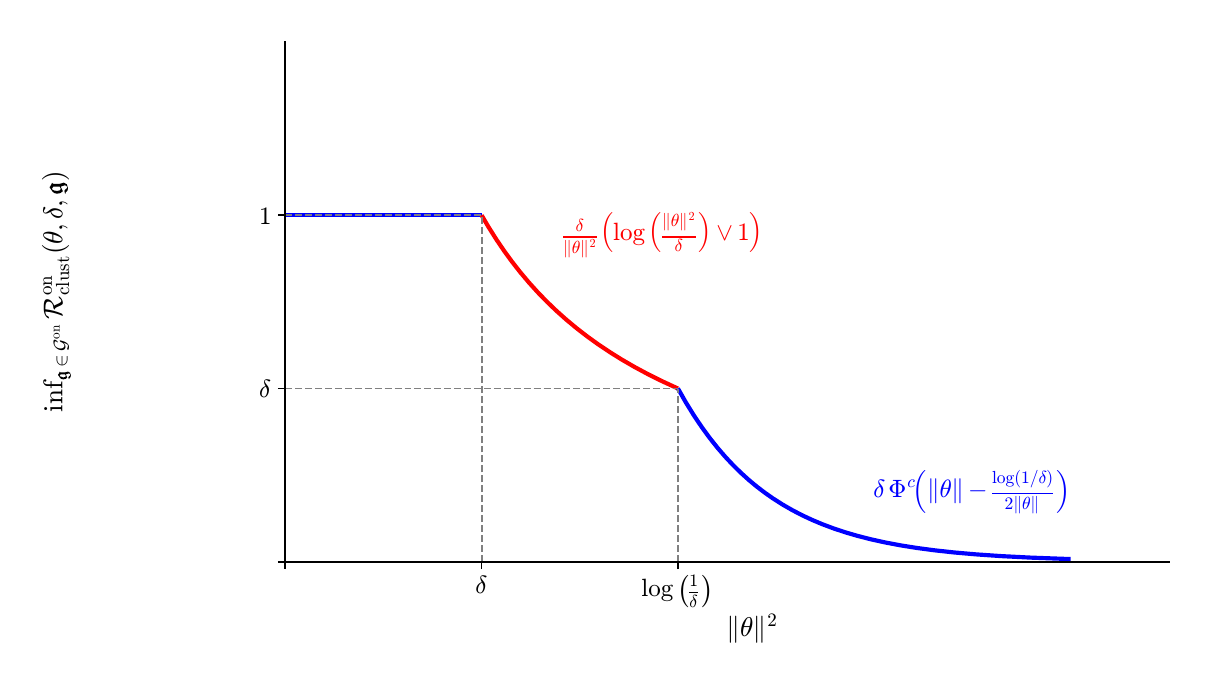}
\caption{Behavior of the Bayes risk of online clustering in the slowly mixing regime.}
\label{fig:Bayes_risk:online:slow_mixing}
\end{figure}
    
The lower-bound argument is closely connected with the delayed-label benchmark studied in Section \ref{sec:semi-supervised}. Roughly speaking, one may reveal suitably chosen past labels without changing the order of the optimal risk in regimes where those labels are already too old to improve the rate. This extra oracle information reduces the online problem to local fixed-lag prediction problems and makes the transition-localization cost explicit. When $\delta$ is bounded away from zero, the slowly mixing intermediate regime collapses, recovering the usual strongly mixing behavior.

\subsection{Offline clustering}

Offline clustering has access to observations on both sides of each candidate transition. This changes the localization problem qualitatively: instead of waiting for one-sided evidence to accumulate, the procedure can compare the two neighboring segments simultaneously. As a consequence, the logarithmic delay factor from the online intermediate regime disappears. We begin with the lower bound.
    \begin{theorem}\label{Bayes_risk_lb_offline}
        Assume that $\delta \geq c\sqrt{\frac{\log(n)}{n}}$ for a sufficiently large constant $c$. The oracle Bayes risk of offline clustering satisfies
        $$
        \inf_{\mathfrak{g}\in \mathcal{G}_n} \ClusterRisk(\theta, \delta, \mathfrak{g}) \gtrsim
        \begin{cases}
        \displaystyle
        1,
        &
        \displaystyle
        \|\theta\|^2\leq2\delta,
        \\[1.2em]
        \displaystyle
        \frac{\delta}{\lVert\theta\rVert^{2}},
        &
        \displaystyle
        2\delta
        \leq
        \|\theta\|^2
        \leq
        1,
        \\[1.2em]
        \displaystyle
        \delta\,
        \Phi^c\!\left(
        \|\theta\|\right),
        &
        \displaystyle
        \|\theta\|^2 > 1.
        \end{cases}
        $$
    \end{theorem}

    Similarly to the online case, the condition on $\delta$ is meant to guarantee the equivalence between the Bayes risks of classification and clustering (see Lemma \ref{lem:Bayes_risk_new_lower_bound} and Corollary \ref{prop:equivalence-from-two-comparisons}).

    In order to match this lower-bound, we first describe a simple oracle construction based on estimating the labels immediately to the left and right of each target location. This procedure is useful for intuition and is already near-optimal, although it incurs an additional logarithmic factor in part of the parameter space.
    Choose \(k\) according to
\[
    k=
    \begin{cases}
    1,
    & \text{if }\|\theta\|^2\leq2\delta,
    \\[1.2ex]
    \displaystyle
    1\vee\left\lceil
        \frac{2}{\|\theta\|^2}
        \log\left(\frac{\|\theta\|^2}{2\delta}\right)
    \right\rceil,
    & \text{if }2\delta<\|\theta\|^2<2\log(1/\delta),
    \\[1.2ex]
    1,
    & \text{if }\|\theta\|^2\geq2\log(1/\delta).
    \end{cases}
\]
We use the convention \(\operatorname{sign}(0)=1\).
    For $a\in\llbracket k+1,n-k\rrbracket$, define the two pilot estimators
\[
    \widetilde\eta_{a-1}
    =\operatorname{sign}\left(\left\langle \frac{1}{k}\sum_{j=a-k}^{a-1}X_j,\theta\right\rangle\right),
    \qquad
    \widetilde\eta_{a+1}
    =\operatorname{sign}\left(\left\langle \frac{1}{k}\sum_{j=a+1}^{a+k}X_j,\theta\right\rangle\right).
\]
Then set
\begin{equation}\label{eq:first_clusterer}
    \widehat\eta_a
    =
    \begin{cases}
        \operatorname{sign}\left(\langle X_a,\theta\rangle+\widetilde\eta_{a-1}\log\left(\frac{1-\delta}{\delta}\right)\right),
        & \text{if } \widetilde\eta_{a-1}=\widetilde\eta_{a+1}\text{ and }a\in\llbracket k+1,n-k\rrbracket,\\[0.8em]
        \displaystyle
        \operatorname{sign}(\langle X_a,\theta\rangle),
        & \text{otherwise}.
    \end{cases}
\end{equation}
   The next theorem bounds the risk of the induced clusterer.
    \begin{theorem}\label{prop:risk_clust_procedure_ub}
        Assume that $\delta \geq 1/n$. The oracle Bayes risk of offline clustering satisfies

        $$
        \ClusterRisk(\theta, \delta, \pi_n\circ \hat{\eta}) \lesssim
        \begin{cases}
        \displaystyle
        1,
        &
        \displaystyle
        \|\theta\|^2\leq2\delta,
        \\[1.2em]
        \displaystyle
        \frac{\delta}{\lVert\theta\rVert^{2}}\left(\log\left(\frac{\lVert\theta\rVert^{2}}{2\delta}\right) + 1\right),
        &
        \displaystyle
        2\delta
        \leq
        \|\theta\|^2
        \leq
        1,
        \\[1.2em]
        \displaystyle
        \delta\,
        \Phi^c\!\left(
        \|\theta\|\right)\left(\log\left(\frac{\lVert\theta\rVert^{2}}{2\delta}\right) + 1\right),
        &
        \displaystyle
        1\leq \|\theta\|^2\leq 2\log(1/\delta)
        \\[1.2em]
        \displaystyle
        \delta\,
        \Phi^c\!\left(
        \|\theta\|\right),
        &
        \displaystyle
        \|\theta\|^2\geq 2\log(1/\delta).
        \end{cases}
        $$
    \end{theorem}
This construction estimates the neighboring states before applying a local Bayes correction at the target. It therefore makes the source of the logarithmic loss transparent: the neighboring states must themselves first be recovered from noisy data.

\subsubsection{A local segmentation procedure}\label{sec:1dim:clustering}

The logarithmic loss in the intermediate regime is not intrinsic. Since the signal direction is known in the oracle setting, the observations may first be projected onto that direction, reducing the local problem to a scalar sequence. A segmentation rule can then search over trajectories with at most one change in a window around the target. For $k\in\mathbb{N}$, define the sets $\mathcal{S}^{+}_{k}$ and $\mathcal{S}^{-}_{k}$ as follows:
\begin{align*}
    \mathcal{S}^{+}_{k} &= \left\{ u \in \{-1,1\}^{2k+1} \;\middle|\; u_{k+1} = 1, \text{ piecewise constant with at most one change} \right\}, \\
    \mathcal{S}^{-}_{k} &= \left\{ u \in \{-1,1\}^{2k+1} \;\middle|\; u_{k+1} = -1, \text{ piecewise constant with at most one change} \right\}.
\end{align*}
Then, for $i\in\llbracket k+1, n-k\rrbracket$, we define the estimator $\hat{\eta}_i$ by:
\[
    \hat{u} \in \operatorname*{argmax}_{u \in \mathcal{S}^{+}_{k} \cup \mathcal{S}^{-}_{k}} \theta^\top X_{i-k:i+k}u, \quad \hat{\eta}_i := (\hat{u})_{k+1},
\]
otherwise we simply define $\hat{\eta}_i = \operatorname{sign}\left(\langle X_i,\theta\rangle\right)$. In the strong signal regime $\lVert\theta\rVert^{2}\geq 2\log\left(1/\delta\right)$, we keep the same $\hat{\eta}_i$ as in the previous construction \eqref{eq:first_clusterer}. 
The following proposition controls the clustering error of $g = \pi_n\circ\hat{\eta}$.

\begin{theorem}\label{thm:risk_clust_new_procedure_ub}
Assume that
  $  \delta\geq \frac{\log n}{n}$,
and, whenever
$
    2\delta<\lVert\theta\rVert^{2}<2\log(1/\delta),
$
choose
\[
    k=
    \left\lceil
        \frac{C_0}{\lVert\theta\rVert^{2}}
        \log\left(\frac{\lVert\theta\rVert^{2}}{\delta}\right)
    \right\rceil,
\]
where $C_0>0$ is a sufficiently large universal constant. Then
\[
\ClusterRisk(\theta,\delta,\pi_n\circ\widehat\eta)
\lesssim
\begin{cases}
    1,
    &
    \lVert\theta\rVert^{2}\leq 2\delta,
    \\[1.2ex]
    \displaystyle
    \frac{\delta}{\lVert\theta\rVert^{2}},
    &
    2\delta<\lVert\theta\rVert^{2}\leq1,
    \\[2ex]
    \displaystyle
    \delta\,\lVert\theta\rVert\,
    \Phi^c(\lVert\theta\rVert),
    &
    1<\lVert\theta\rVert^{2}<2\log(1/\delta).
\end{cases}
\]
\end{theorem}

\begin{remark}
The condition $\delta\geq(\log n)/n$ is only used to control the
$O(k)$ observations lying within distance $k$ of the two endpoints,
where the estimator is defined by the pointwise rule
\[
    \widehat\eta_i
    =
    \sign(\langle X_i,\theta\rangle).
\]
The interior local-segmentation argument itself does not require this
stronger lower bound on $\delta$. The condition could therefore be
removed by using an appropriate one-sided or truncated-window
segmentation rule near the endpoints. 
\end{remark}
\begin{remark}
Moreover, the multiplicative factor $\lVert\theta\rVert$ in the last regime is
asymptotically negligible when $\lVert\theta\rVert\to\infty$, which is the regime of
interest for sharp asymptotic recovery results. Indeed, by Mills'
inequality, there exists $\varepsilon_n\to0$ such that
\[
    \lVert\theta\rVert\,\Phi^c(\lVert\theta\rVert)
    \lesssim
    \Phi^c\!\left(\lVert\theta\rVert(1-\varepsilon_n)\right).
\]
\end{remark}

The local segmentation rule removes the
logarithmic factor incurred by the preceding construction in
the intermediate-signal regime. The key point is that the estimator
can borrow information from a large local window without assuming
that the entire window has a common label: when a transition occurs,
the rule can estimate its location rather than misclassifying all
observations on one side of it. In sufficiently strong signal regimes, neighboring labels can themselves be estimated at a cost below the scale of the localization error of the transition, so the simpler plug-in construction becomes effective again. Figure~\ref{fig:Bayes_risk:offline:slow_mixing} summarizes the oracle offline behavior.

\begin{figure}[hbt!]
\centering
\includegraphics[width=0.99\textwidth]{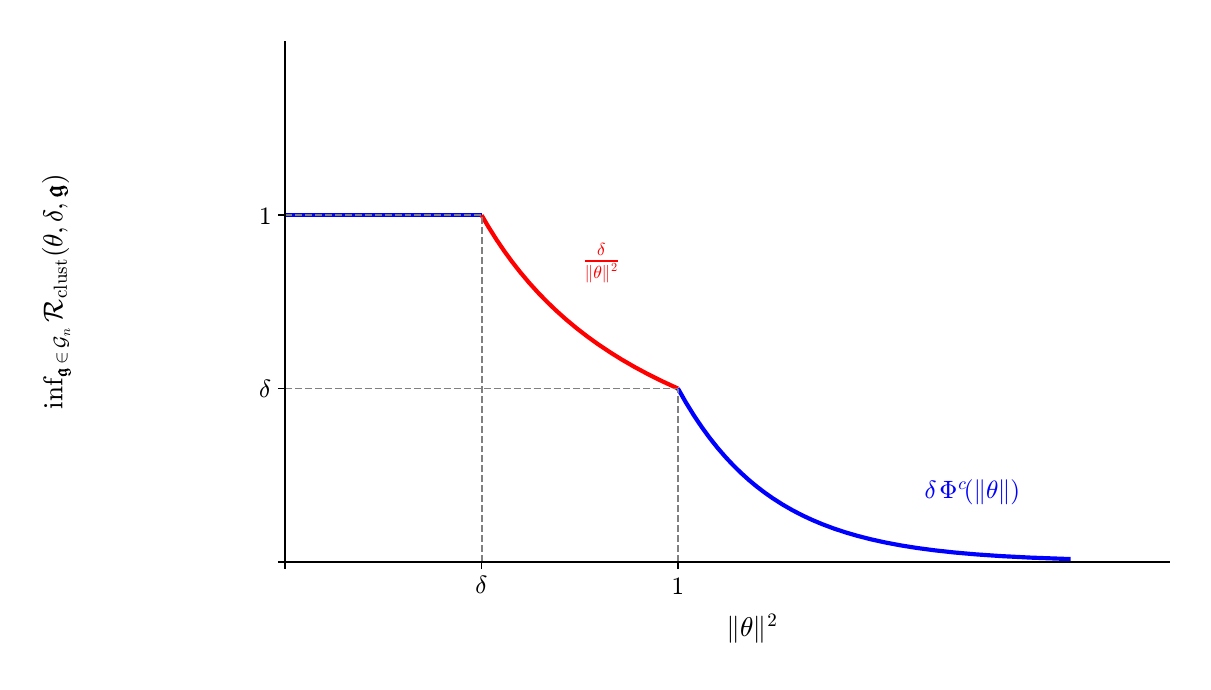}
\caption{Behavior of the Bayes risk of offline clustering in the slowly mixing regime.}
\label{fig:Bayes_risk:offline:slow_mixing}
\end{figure}

\begin{remark}[Almost full recovery in the oracle problem]
The weak-signal boundary is of order $\delta$: in both the online and offline oracle settings, the clustering risk tends to zero as soon as $\|\theta\|^2/\delta\to\infty$. Thus slow mixing lowers the separation required for almost full recovery relative to the i.i.d. Gaussian mixture, where the corresponding threshold is of constant order. The improvement comes from temporal persistence, which permits aggregation over long stretches carrying the same latent label.
\end{remark}

\subsection{Adaptive estimation of the nuisance parameters}
\label{sec:adaptation}

The oracle procedures depend on the signal strength \(\|\theta\|^2\) and on
the transition probability \(\delta\). We estimate both quantities by temporal
aggregation. Block averaging reduces the Gaussian noise while preserving a
nonnegligible fraction of the label signal as long as the block length remains
below the persistence scale \(1/\delta\).

Since this subsection concerns estimation up to universal constants,
we may retain the largest dyadic sample size \(N\le n\).
As \(N\asymp n\), this does not affect the stated rates, and we
continue to write \(n\) for the effective sample size.

\textbf{Estimating \(\|\theta\|^2\) and constructing an initializer.}
Let \(\mathcal L_n=\{1,2,4,\ldots,n\}\). For \(\ell\in\mathcal L_n\),
let \(k=n/\ell\) and define
\[
    \widetilde X_i^{(\ell)}
    =
    \frac1k\sum_{j=k(i-1)+1}^{ki}X_j,
    \qquad i=1,\ldots,\ell.
\]
Let \(\widetilde X^{(\ell)}\) be the \(d\times\ell\) matrix with columns
\(\widetilde X_1^{(\ell)},\ldots,\widetilde X_\ell^{(\ell)}\), and define
\[
 \overline\eta_i^{(\ell)}
 =\frac1k\sum_{j=k(i-1)+1}^{ki}\eta_j,
 \qquad i=1,\ldots,\ell.
\]
Write
\[
    \overline\eta^{(\ell)}
    =
    \bigl(
        \overline\eta_1^{(\ell)},\ldots,
        \overline\eta_\ell^{(\ell)}
    \bigr)^\top .
\]
Conditionally on the latent path,
\[
    \EE\!\left[
        \frac1\ell
        \bigl(\widetilde X^{(\ell)}\bigr)^\top
        \widetilde X^{(\ell)}
        \;\middle|\;
        \eta_{1:n}
    \right]
    =
    \frac{\|\theta\|^2}{\ell}
    \overline\eta^{(\ell)}
    \bigl(\overline\eta^{(\ell)}\bigr)^\top
    +
    \frac{d\sigma^2}{k\ell}I_\ell .
\]
By Lemma~\ref{lem:block-label-energy}, when \(\ell\ge n\delta\), the bucketed label
vector retains constant-order energy:
\[
    \EE\bigl\|\overline\eta^{(\ell)}\bigr\|^2
    \ge \frac{\ell}{3}.
\]
A high-probability version of this energy bound, together with spectral
concentration, shows that the leading eigenvalue of the centered Gram
matrix is of order \(\|\theta\|^2\) when the blocks are shorter than the
persistence scale and the signal dominates the spectral noise; see
Lemmas~\ref{lem:block-label-energy} and~\ref{lem:block-gram}.
For each \(\ell\), define
\[
    \widehat\theta^{\,2}(\ell)
    :=
    \left|
        \lambda_{\max}\left(
            \frac1\ell
            \bigl(\widetilde X^{(\ell)}\bigr)^\top
            \widetilde X^{(\ell)}
            -
            \frac{d\sigma^2}{n}I_\ell
        \right)
    \right|.
\]
We select the finest dyadic scale at which this spike dominates the
spectral noise:
\[
    \widehat\ell
    =
    \max\left\{
        \ell\in\mathcal L_n:
        \widehat\theta^{\,2}(\ell)
        \ge
        A\sigma^2\frac{\ell+\sqrt{d\ell}}{n}
    \right\},
\]
where \(A\) is a sufficiently large fixed universal constant, chosen as in
the proof of Proposition~\ref{prop:adaptive:estimation}; set
\(\widehat\ell=1\) if the set is empty. Set
\[
    \widehat\theta^{\,2}
    :=
    \widehat\theta^{\,2}(\widehat\ell),
\]
Let \(\widehat u\in\mathbb R^{\widehat\ell}\) be a unit leading
eigenvector of the centered Gram matrix at the selected scale. Expanding \(\widehat u\) back to the
original time scale, with \(\widehat k=n/\widehat\ell\), define
\[
    (\widehat u_n)_j
    =
    \frac{\widehat u_i}{\sqrt{\widehat k}},
    \qquad
    j\in
    \{\widehat k(i-1)+1,\ldots,\widehat k i\}.
\]

For later use, define the pointwise effective signal
\[
    r_n^2(\theta)
    :=
    \frac{\|\theta\|^4/\sigma^4}
         {\|\theta\|^2/\sigma^2+d/n},
\]
and set
\[
    \widehat\eta^{(0)}
    :=
    \sign(\widehat u_n).
\]
We use \(\sign(0)=1\), with the sign applied coordinatewise.

\begin{proposition}
\label{prop:adaptive:estimation}
There exist universal constants \(c,c',C>0\) such that, whenever
$
    \delta\ge\frac1{4n}
    $ and $
    r_n^2(\theta)\ge c'\delta,
$
one has
\[
    \PP\left(
        \left|
            \widehat\theta^{\,2}-\|\theta\|^2
        \right|
        \ge
        \frac9{10}\|\theta\|^2
    \right)
    \le
    C\exp\left(
        -cn\bigl(r_n^2(\theta)\wedge1\bigr)
    \right),
\]
and
\[
    \PP\left(
        \frac1n
        \left|
            \langle\widehat\eta^{(0)},\eta\rangle
        \right|
        \le
        \frac13
    \right)
    \le
    C\exp\left(
        -cn\bigl(r_n^2(\theta)\wedge1\bigr)
    \right).
\]
\end{proposition}

Thus a constant-factor estimate of the squared signal magnitude and a
positively correlated label initializer are available throughout the recovery regime.
For later use, set
\[
    \widehat a_{\rm ad}
    :=
    \frac{\widehat\theta^{\,2}}{\sigma^2},
    \qquad
    \widehat r_{\rm ad}^{\,2}
    :=
    \frac{\widehat a_{\rm ad}^{\,2}}
         {\widehat a_{\rm ad}+d/n},
    \qquad
    \widehat\eta_{\rm ad}^{(0)}
    :=
    \widehat\eta^{(0)}.
\]
On the intersection of the two good events in Proposition~\ref{prop:adaptive:estimation},
\[
    \widehat a_{\rm ad}
    \asymp
    \frac{\|\theta\|^2}{\sigma^2},
    \qquad
    \widehat r_{\rm ad}^{\,2}
    \asymp
    r_n^2(\theta).
\]

\textbf{Estimating \(\delta\).}
Throughout this paragraph assume \(n\geq 4\), \(\delta\ge1/(2n)\) and
\(r_n^2(\theta)\ge C\delta\), for a sufficiently large universal constant
\(C\). After projection onto a direction correlated with \(\theta\), sign
disagreements between neighboring observations arise either from a latent
transition or from Gaussian sign errors. We reduce the latter by temporal
averaging and estimate the transition frequency from disagreements between
adjacent block sums.

To decouple direction estimation from transition counting, split the sample
into two consecutive halves \(S_1\) and \(S_2\). Apply
Proposition~\ref{prop:adaptive:estimation} to \(S_1\), and denote the
corresponding selected scale, unit eigenvector, and signal estimate by
\(\widehat\ell,\widehat u,\widehat\theta^{\,2}\), reusing these symbols
only within this paragraph. Write \(m=|S_1|=|S_2|=n/2\).
Then \(\delta\ge1/(4m)\) and
\(r_m^2(\theta)\ge r_n^2(\theta)/2\), so the first proposition applies
with sample size \(m\) after increasing \(C\). Define
\[
    \widehat r^{\,2}
    :=
    \frac{
        \widehat\theta^{\,4}/\sigma^4
    }{
        \widehat\theta^{\,2}/\sigma^2+2d/n
    }.
\]
The factor \(2d/n\) reflects the use of half the sample. With probability at
least
\[
    1-
    C\exp\left(
        -cn\bigl(r_n^2(\theta)\wedge1\bigr)
    \right),
\]
one has
\[
    \widehat r^{\,2}
    \asymp
    r_n^2(\theta).
\]

At the selected scale, define
\[
    \widehat v
    :=
    \widetilde X_{S_1}^{(\widehat\ell)}\widehat u,
    \qquad
    \widehat w
    :=
    \frac{\widehat v}{\|\widehat v\|},
\]
with an arbitrary deterministic unit vector if \(\widehat v=0\). Lemma~\ref{lem:projected-signal} establishes the additional projection bound
\begin{equation}
\label{eq:projected-signal}
    c r_n^2(\theta)
    \le
    \frac{
        \langle\theta,\widehat w\rangle^2
    }{\sigma^2}
    \le
    C r_n^2(\theta)
\end{equation}
jointly with the preceding bound for $\widehat r^{\,2}$, outside an event
of probability at most $C\exp\left(-cn(r_n^2(\theta)\wedge1)\right)$.

Conditionally on the first half of the sample, the projected observations indexed by
\(S_2\) satisfy
\[
    Z_j
    :=
    \langle X_j,\widehat w\rangle
    =
    \mu\eta_j+\sigma g_j,
    \qquad
    \mu
    :=
    \langle\theta,\widehat w\rangle,
\]
where the \(g_j\)'s are i.i.d.\ standard Gaussian and, conditionally on
\(\mathcal F_1=\sigma(X_j:j\in S_1)\), are independent of the latent labels
on \(S_2\). The latent chain on \(S_2\) need not be stationary under this
conditioning. The argument below uses its independent future transition
variables. On this joint good event,
\[
    \frac{\mu^2}{\sigma^2}
    \asymp
    \widehat r^{\,2}
    \asymp
    r_n^2(\theta).
\]

For a dyadic block size \(M\le m/2\) dividing \(m=|S_2|\), partition \(S_2\) into
consecutive blocks \(B_1^{(M)},\ldots,B_{|S_2|/M}^{(M)}\), and define
\[
    S_j^{(M)}
    :=
    \sum_{t\in B_j^{(M)}} Z_t.
\]
Given pair consecutive blocks, let
\[
    D_M
    :=
    \frac1{|S_2|/(2M)}
    \sum_{i=1}^{|S_2|/(2M)}
    \mathbf 1
    \left\{
        S_{2i-1}^{(M)}S_{2i}^{(M)}
        \le0
    \right\},
    \qquad
    \widehat\delta_M
    :=
    \frac{D_M}{M}.
\]
Let
 us denote by $\mathcal M:=\mathcal{L}_{m/2}$.
Fix a sufficiently small universal constant \(\gamma>0\), and then a
sufficiently large universal constant \(B>0\), as specified in the proof.
If \(\{M\in\mathcal M:M\widehat r^{\,2}\ge B\}\) is empty, set
\(\widehat\delta=1/n\) and stop. Otherwise, let
\[
 M_0=\min\{M\in\mathcal M:M\widehat r^{\,2}\ge B\},
 \qquad
 \mathcal M(\widehat r)
 =\{M_0,2M_0,4M_0,\ldots\}\cap\mathcal M.
\]
In particular, \(M_0=1\) when \(\widehat r^{\,2}\ge B\).
Define
\[
 \widehat M
 =\min\left\{M\in\mathcal M(\widehat r):
 D_M\ge \exp(-\gamma M\widehat r^{\,2})\right\}.
\]
If this set is empty, take \(\widehat M\) to be the largest element of
\(\mathcal M(\widehat r)\). Finally, set
\[
 \widehat\delta
 =\left(\frac{D_{\widehat M}}{\widehat M}\vee\frac1n\right)
   \wedge\frac12.
\]

\begin{proposition}
\label{prop:delta}
There exist universal constants \(c,C>0\) such that, whenever
$
    \delta\ge\frac1{2n}
    $ and $
    r_n^2(\theta)\ge C\delta,
$
one has
\[
    \PP\left(
        \frac{\widehat\delta}{\delta}
        \vee
        \frac{\delta}{\widehat\delta}
        \ge100
    \right)
    \le
    C\exp(-cn\delta).
\]
\end{proposition}

Thus \(\delta\) can be estimated within a universal factor throughout the
nontrivial recovery regime. The exponent \(n\delta\) is natural: even with
noiseless labels, only \(O(n\delta)\) transitions are observed, so consistent
estimation of \(\delta\) requires \(n\delta\to\infty\).

\subsection{Adaptive minimax clustering}
\label{sec:minimax}

\subsubsection{Minimax lower bound}

We now allow the signal direction to be unknown and potentially
high-dimensional. For \(\Delta>0\), define
\[
    \Psi_{n,d}(\delta,\Delta,\sigma)
    :=
    \inf_{\mathfrak g\in\mathcal G_n}
    \sup_{\|\theta\|\ge\Delta}
    \ClusterRisk(\theta,\delta,\sigma,\mathfrak g),
\]
and the associated effective signal
\begin{equation}
\label{effective_signal}
    r_n^2
    :=
    r_n^2(\Delta)
    =
    \frac{
        \Delta^4/\sigma^4
    }{
        \Delta^2/\sigma^2+d/n
    }.
\end{equation}
We reserve \(r_n(\theta)\) for the corresponding pointwise quantity.
\begin{theorem}
\label{thm:minimax_lb_high_dim_clustering}
Assume that \(C_0/n<\delta\le1/2\), where \(C_0\) is a sufficiently large
universal constant. Then there exists a universal constant \(c>0\) such that
\[
    \Psi_{n,d}(\delta,\Delta,\sigma)
    \ge
    c
    \begin{cases}
        1,
        & r_n^2\le\delta,
        \\[1ex]
        \displaystyle
        \frac{\delta}{r_n^2},
        & \delta<r_n^2\le1,
        \\[2ex]
        \displaystyle
        \delta\Phi^c\!\left(
            r_n(1+\varepsilon_n)
        \right),
        & r_n^2>1,
    \end{cases}
\]
where \(\varepsilon_n\to0\) as \(n,r_n\to\infty\).
\end{theorem}

\subsubsection{An adaptive iterative estimator}

We first construct a fully adaptive iterative estimator that mirrors the
oracle procedure of Section~\ref{sec:1dim:clustering}. It is rate-optimal up
to a logarithmic factor in the intermediate regime and becomes sharp at the
logarithmic signal scale.

Throughout this subsection we normalize \(\sigma=1\) and write
\[
    a
    :=
    \|\theta\|^2,
    \qquad
    r_n^2(\theta)
    :=
    \frac{a^2}{a+d/n}.
\]

Besides
\[
    \widehat\delta,
    \qquad
    \widehat a_{\rm ad},
    \qquad
    \widehat r_{\rm ad},
    \qquad
    \widehat\eta_{\rm ad}^{(0)}
\]
from Section~\ref{sec:adaptation}, we use a sharp spectral initialization
when \(r_n(\theta)\to\infty\). For a square matrix \(A\), let
\[
    H(A)
    :=
    A-\operatorname{diag}(A),
    \qquad
    G
    :=
    \frac1nH(X^\top X).
\]
Let \(\widehat v\) be a unit leading eigenvector of \(G\), and define
\[
    \widehat\eta_{\rm sp}^{(0)}
    :=
    \sign(\widehat v),
    \qquad
    \widehat a_{\rm sp}
    :=
    [\lambda_{\max}(G)]_+,
    \qquad
    \widehat r_{\rm sp}^{\,2}
    :=
    \frac{
        \widehat a_{\rm sp}^{\,2}
    }{
        \widehat a_{\rm sp}+d/n
    }.
\]
By the spectral perturbation bounds of \cite{Ndaoud18}, uniformly in the
label vector and whenever \(r_n(\theta)\to\infty\),
\[
    \frac{\widehat a_{\rm sp}}a
    =
    1+o_{\PP}(1),
    \qquad
    \frac{\widehat r_{\rm sp}}{r_n(\theta)}
    =
    1+o_{\PP}(1),
\]
and, after a global sign choice,
\[
    \frac1n
    \left|
        \langle
            \widehat\eta_{\rm sp}^{(0)},\eta
        \rangle
    \right|
    \ge
    1-
    \frac{C}{r_n^2(\theta)}
    -
    o(1)
\]
with high probability.

Fix a sufficiently large universal constant \(R_0\), and set
\[
    \bigl(
        \widehat\eta^{(0)},
        \widehat a_*,
        \widehat r_*
    \bigr)
    :=
    \begin{cases}
    \bigl(
        \widehat\eta_{\rm ad}^{(0)},
        \widehat a_{\rm ad},
        \widehat r_{\rm ad}
    \bigr),
    & \widehat r_{\rm ad}\le R_0,
    \\[1ex]
    \bigl(
        \widehat\eta_{\rm sp}^{(0)},
        \widehat a_{\rm sp},
        \widehat r_{\rm sp}
    \bigr),
    & \widehat r_{\rm ad}>R_0.
    \end{cases}
\]

For \(\rho,q>0\), define
\[
    K(\rho,q)
    :=
    \left\lceil
        \frac{C_0}{\rho^2\vee n^{-1}}
        \log\left(
            1+\frac{\rho^2}{2q}
        \right)
    \right\rceil,
\]
where \(C_0>0\) is a sufficiently large universal constant. 

\textbf{Choice of the local scale.}
If $
    \widehat r_{\rm ad}^{\,2}\leq2\widehat\delta,
$
the target bound is trivial after truncation at one and we keep the
preliminary estimator. Otherwise, define
\[
    k=
    \begin{cases}
    \displaystyle
    \left\lceil
        \frac{C_0}{\widehat r_{\rm ad}^{\,2}}
        \log\left(
            1+\frac{\widehat r_{\rm ad}^{\,2}}
                     {2\widehat\delta}
        \right)
    \right\rceil,
    &
    \widehat r_{\rm ad}\leq R_0
    \ \text{or}\
    \widehat r_{\rm sp}^{\,2}
    <
    2\log\left(\dfrac{100}{\widehat\delta}\right),
    \\[3ex]
    1,
    &
    \widehat r_{\rm ad}>R_0
    \ \text{and}\
    \widehat r_{\rm sp}^{\,2}
    \geq
    2\log\left(\dfrac{100}{\widehat\delta}\right).
    \end{cases}
\]
Here $C_0>0$ is a sufficiently large universal constant. The factor
$100$ compensates for the constant-factor uncertainty in
$\widehat\delta$, since on the event of
Proposition~\ref{prop:delta}, we have 
\[
    \frac{\delta}{100}
    \leq
    \widehat\delta
    \leq
    100\delta.
\]
Thus constant-factor estimation of $\delta$ is sufficient for the
regime selection. The sharp spectral estimate
$\widehat r_{\rm sp}$ is used only to decide when it is safe to pass
to the one-step rule $k=1$.

\textbf{Block-hollowed iterations.}
Use two deterministic partitions \(\mathcal P_0,\mathcal P_1\) of
\(\llbracket1,n\rrbracket\) into consecutive blocks, with all non-endpoint
blocks of length \(4k\), and with \(\mathcal P_1\) shifted by \(2k\) relative
to \(\mathcal P_0\). For each
\(a\in\llbracket k+1,n-k\rrbracket\), choose
\(B_a\in\mathcal P_0\cup\mathcal P_1\) such that
\[
    \llbracket a-k,a+k\rrbracket
    \subset
    B_a.
\]
For \(m\ge0\), define the leave-block-out direction
\[
    \widehat\theta_a^{(m)}
    :=
    \frac1{n-|B_a|}
    \sum_{j\notin B_a}
    X_j\widehat\eta_j^{(m)},
\]
and the two pilot estimators
\[
    \widetilde\eta_{a,-}^{(m)}
    :=
    \sign\left\langle
        \frac1k
        \sum_{j=a-k}^{a-1}X_j,
        \widehat\theta_a^{(m)}
    \right\rangle,
    \qquad
    \widetilde\eta_{a,+}^{(m)}
    :=
    \sign\left\langle
        \frac1k
        \sum_{j=a+1}^{a+k}X_j,
        \widehat\theta_a^{(m)}
    \right\rangle.
\]
Set
\[
    \widehat L_\delta
    :=
    \log\left(
        \frac{1-\widehat\delta}{\widehat\delta}
    \right),
    \qquad
    \widehat\kappa
    :=
    \frac{
        \widehat a_*
    }{
        \widehat r_*^{\,2}
    }
    \widehat L_\delta .
\]
In the diverging-signal regime,
\[
    \widehat\kappa
    =
    \frac{a}{r_n^2(\theta)}
    \left\{
        \log\left(\frac{1-\delta}{\delta}\right)
        \bigl(1+o_{\PP}(1)\bigr)
        +
        O(1)
    \right\}.
\]
The $O(1)$ term is the entire effect of estimating $\delta$ only up to
a constant factor; after normalization of the local Gaussian score it
contributes only $O_{\PP}(1/r_n(\theta))$.

For every \(a\in\llbracket k+1,n-k\rrbracket\), set
\[
    \widehat\eta_a^{(m+1)}
    =
    \begin{cases}
    \displaystyle
    \sign\left(
        \langle X_a,\widehat\theta_a^{(m)}\rangle
        +
        \widetilde\eta_{a,-}^{(m)}
        \widehat\kappa
    \right),
    &
    \widetilde\eta_{a,-}^{(m)}
    =
    \widetilde\eta_{a,+}^{(m)},
    \\[2ex]
    \displaystyle
    \sign\left(
        \langle X_a,\widehat\theta_a^{(m)}\rangle
    \right),
    &
    \widetilde\eta_{a,-}^{(m)}
    \neq
    \widetilde\eta_{a,+}^{(m)}.
    \end{cases}
\]
For boundary indices
\(a\notin\llbracket k+1,n-k\rrbracket\), define
\[
    \widehat\theta_a^{(m)}
    :=
    \frac1{n-1}
    \sum_{j\neq a}
    X_j\widehat\eta_j^{(m)}
\]
and use
\[
    \widehat\eta_a^{(m+1)}
    :=
    \sign\left(
        \langle X_a,\widehat\theta_a^{(m)}\rangle
    \right).
\]
\begin{theorem}
\label{upper:minimax}
Assume \(\delta\ge1/n\), and run the iterative procedure for any
\(m\ge4\log n\). There exist universal constants \(c,C>0\) and, along
every sequence such that \(n\to\infty\), \(r_n\to\infty\) a sequence \(\varepsilon_n\to0\) such that
\[
    \sup_{\|\theta\|\ge\Delta}
    \EE_{\theta,\delta}
    \left[
        \ell(\widehat\eta^{(m)},\eta)
    \right]
    \le
    C\exp(-cn\delta)
    +
    C
    \begin{cases}
    \displaystyle
    \left(
        \frac{
            \delta\log(1+r_n^2/\delta)
        }{
            r_n^2
        }
    \right)
    \wedge1,
    &
    r_n^2\le1,
    \\[2ex]
    \displaystyle
    \frac{
        \delta\log(r_n^2/\delta)
    }{
        r_n^2
    }
    \Phi^c\!\left(
        r_n(1-\varepsilon_n)
    \right)
    \wedge1,
    &
    1<r_n^2<2\log(1/\delta),
    \\[2ex]
    \displaystyle
    \delta
    \Phi^c\!\left(
        r_n(1-\varepsilon_n)
    \right),
    &
    r_n^2\ge2\log(1/\delta),
    \end{cases}
\]
where \(r_n=r_n(\Delta)\) is defined in \eqref{effective_signal}.
\end{theorem}

The iterative estimator is sharp at the logarithmic and strong-signal
scales, but incurs an additional logarithmic factor in the intermediate
regime. We now remove this loss by cross-fitting.

\subsubsection{Optimal adaptive recovery}
\label{sec:adaptive:sharp}

The sharp construction uses three common ingredients: a coarse direction, a
sharp direction, and the one-dimensional segmentation rule of
Section~\ref{sec:1dim:clustering}.

For \(T\subset\llbracket n\rrbracket\), let \(\widetilde\eta^T\) denote
the preliminary labels obtained by applying
Proposition~\ref{prop:adaptive:estimation} to the observations indexed by
\(T\), and define the coarse direction
\[
    \widehat w_{\rm ad}(T)
    :=
    \frac{
        \sum_{j\in T}
        \widetilde\eta_j^T X_j
    }{
        \left\|
            \sum_{j\in T}
            \widetilde\eta_j^T X_j
        \right\|
    }.
\]
Whenever a denominator in this subsection vanishes, we use an arbitrary
deterministic unit vector.

For the sharp construction, let \(m:=|T|\) and let \(\widehat v_T\) be a
unit leading eigenvector of
\[
    \frac1m
    H(X_T^\top X_T).
\]
Define
\[
    \widetilde\eta_{\rm sp}^T
    =
    \sign(\widehat v_T), \qquad \widehat\theta_{\rm sp}(T)
    :=
    \frac1{\sqrt m}
    X_T\widehat v_T,
    \qquad
    \widehat w_{\rm sp}(T)
    :=
    \frac{
        \widehat\theta_{\rm sp}(T)
    }{
        \|\widehat\theta_{\rm sp}(T)\|
    }.
\]
The spectral perturbation argument of \cite{Ndaoud18}, together with
Gaussian concentration, gives, whenever \(r_m(\theta)\to\infty\),
\[
    \bigl|
        \langle
            \theta,\widehat w_{\rm sp}(T)
        \rangle
    \bigr|
    =
    r_m(\theta)\bigl(1-o(1)\bigr)
\]
with probability at least \(1-Ce^{-cm}\). In particular, if
\(m=n(1-o(1))\), then
\[
    \bigl|
        \langle
            \theta,\widehat w_{\rm sp}(T)
        \rangle
    \bigr|
    =
    r_n(\theta)\bigl(1-o(1)\bigr).
\]

For an interval \(S\), a unit vector \(w\), and a radius \(k\), let
\[
    \operatorname{Seg}_k(S,w)
\]
denote the label vector obtained by projecting \((X_i)_{i\in S}\) onto
\(w\) and applying the local segmentation rule of
Section~\ref{sec:1dim:clustering}. Whenever two local reconstructions
overlap, we orient them by majority agreement on an interior part of the
overlap, trimmed by \(k\) at each artificial boundary.

\textbf{Bounded signal.}
Fix a sufficiently large universal constant \(R_0\). When
\(\widehat r_{\rm ad}\le R_0\) and
\(\widehat r_{\rm ad}^{\,2}\le2\widehat\delta\), set
\(\widehat\eta^{\rm cf,w}:=\widehat\eta_{\rm ad}^{(0)}\). Otherwise,
divide the sample into three consecutive blocks \(A,B,C\) of comparable
sizes and set
\[
    S_1
    :=
    A\cup B,
    \qquad
    S_2
    :=
    B\cup C.
\]
Use the cross-fitted directions
\[
    \widehat w_1
    :=
    \widehat w_{\rm ad}(C),
    \qquad
    \widehat w_2
    :=
    \widehat w_{\rm ad}(A),
\]
and the radius
\[
    k_{\rm w}
    :=
    \left\lceil
        \frac{
            C_0
        }{
            \widehat r_{\rm ad}^{\,2}\vee n^{-1}
        }
        \log\left(
            1+
            \frac{
                \widehat r_{\rm ad}^{\,2}
            }{
                \widehat\delta
            }
        \right)
    \right\rceil.
\]
Set
\[
    \widehat\eta^{(a)}
    :=
    \operatorname{Seg}_{k_{\rm w}}
    (S_a,\widehat w_a),
    \qquad
    a=1,2,
\]
align the two reconstructions on their trimmed overlap, and denote the
combined label vector by \(\widehat\eta^{\rm cf,w}\). Since
\[
    r_{n/3}^2(\theta)
    \asymp
    r_n^2(\theta),
\]
the constant-factor loss in direction estimation is harmless in this
regime.

\textbf{Diverging intermediate signal.}
Fix constants
\[
    0<c_-<c_\star<c_+<2.
\]
Suppose \(\widehat r_{\rm ad}>R_0\) and
\[
    \widehat r_{\rm sp}^{\,2}
    <
    c_\star\log(1/\widehat\delta).
\]
Let
\[
    L_n
    :=
    \left\lceil
        \log(e+\widehat r_{\rm sp})
    \right\rceil,
    \qquad
    M_n
    :=
    \left\lfloor
        \frac n{L_n}
    \right\rfloor.
\]
Cover \(\llbracket n\rrbracket\) by \(J\asymp L_n\) consecutive
overlapping intervals \(S_1,\ldots,S_J\), each of length comparable to
\(M_n\), with consecutive overlaps also comparable to \(M_n\). For each
\(j\), let
\[
    T_j
    :=
    \llbracket n\rrbracket\setminus S_j,
    \qquad
    \widehat w_j
    :=
    \widehat w_{\rm sp}(T_j),
\]
and define
\[
    k_{\rm s}
    :=
    \left\lceil
        \frac{
            C_0
        }{
            \widehat r_{\rm ad}^{\,2}\vee n^{-1}
        }
        \log\left(
            1+
            \frac{
                \widehat r_{\rm ad}^{\,2}
            }{
                \widehat\delta
            }
        \right)
    \right\rceil,
    \qquad
    \widehat\eta^{(j)}
    :=
    \operatorname{Seg}_{k_{\rm s}}
    (S_j,\widehat w_j).
\]
Align consecutive reconstructions on their trimmed overlaps and denote
the combined label vector by \(\widehat\eta^{\rm cf,s}\). In this regime,
\[
    L_n
    \asymp
    \log r_n(\theta),
    \qquad
    M_n
    \asymp
    \frac n{\log r_n(\theta)},
    \qquad
    k_{\rm s}
    \asymp
    \frac1{r_n^2(\theta)}
    \log\left(
        \frac{r_n^2(\theta)}{\delta}
    \right)
    =
    o(M_n).
\]
Thus each local direction is estimated from \(n(1-o(1))\) observations (when $r_n(\theta) \to \infty$),
whereas only \(O(\log r_n(\theta))\) alignments are required.

\textbf{Data-driven selection.} 
Let \(\widehat\eta^{\rm cf,w}\) and \(\widehat\eta^{\rm cf,s}\) denote the
bounded-signal and diverging-signal cross-fitted estimators above, respectively,
and let \(\widehat\eta^{\rm it}\) denote the iterative estimator of
Theorem~\ref{upper:minimax}. The final estimator is
\[
    \widehat\eta
    :=
    \begin{cases}
    \widehat\eta^{\rm cf,w},
    &
    \widehat r_{\rm ad}
    \le
    R_0,
    \\[1ex]
    \widehat\eta^{\rm cf,s},
    &
    \widehat r_{\rm ad}
    >
    R_0
    \quad\text{and}\quad
    \widehat r_{\rm sp}^{\,2}
    <
    c_\star\log(1/\widehat\delta),
    \\[1ex]
    \widehat\eta^{\rm it},
    &
    \text{otherwise}.
    \end{cases}
\]
The cross-fitted construction is sharp throughout
\(r_n^2(\theta)\le c_+\log(1/\delta)\), whereas
Theorem~\ref{upper:minimax} is already sharp when
\(r_n^2(\theta)\ge c_-\log(1/\delta)\). The overlap
\(c_-<c_\star<c_+\), together with consistency of
\(\widehat r_{\rm sp}\) and constant-factor estimation of \(\delta\),
therefore makes the regime selection asymptotically harmless.

\begin{theorem}[Adaptive minimax recovery]
\label{upper:minimax:optimal}
Assume $
    \delta
    \ge
    \frac{\log n}{n}.$
There exist universal constants \(c,C>0\) such that the fully adaptive
estimator above satisfies
\[
    \sup_{\|\theta\|\ge\Delta}
    \EE_{\theta,\delta}
    \bigl[
        \ell(\widehat\eta,\eta)
    \bigr]
    \le
    C e^{-cn\delta}
    +
    C
    \begin{cases}
    \displaystyle
    \frac{\delta}{r_n^2}\wedge1,
    &
    r_n^2\le1,
    \\[2ex]
    \displaystyle
    \delta
    \Phi^c\!\left(
        r_n(1-\varepsilon_n)
    \right),
    &
    r_n^2>1,
    \end{cases}
\]
where \(r_n=r_n(\Delta)\) is defined in \eqref{effective_signal}. In the
Gaussian-tail regime, \(\varepsilon_n\to0\) along every sequence such that
\(n\to\infty\), \(r_n\to\infty\).
\end{theorem}

The assumption \(\delta\ge(\log n)/n\) is used only to control the two true
sample boundaries. The polynomial factors introduced by local segmentation
and alignment are absorbed into the vanishing perturbation of the Gaussian
tail.

\subsubsection{Recovery phase transitions}
\label{sec:adaptive:sharp:transition}

\begin{definition}[Almost full and exact recovery]
Consider a sequence of parameter classes
$\Theta(\Delta_n)=\{\theta\in\mathbb R^{d_n}:\|\theta\|\ge\Delta_n\}$,
with transition probabilities $\delta_n$ and known noise levels
$\sigma_n$. \underline{Almost full recovery} is possible if and only if there exists a sequence of
estimators $\widehat\eta\in\{-1,+1\}^n$ such that
\[
 \sup_{\theta\in\Theta(\Delta_n)}
 \mathbb E_{\theta,\delta_n,\sigma_n}
 [\ell(\widehat\eta,\eta)]\underset{n \to \infty}{\to} 0.
\]
Following the expected Hamming loss convention of \cite{Ndaoud18},
\underline{exact recovery} is
possible if and only if there exists a sequence of estimators such that
\[
 \sup_{\theta\in\Theta(\Delta_n)}
 n\mathbb E_{\theta,\delta_n,\sigma_n}
 [\ell(\widehat\eta,\eta)]\underset{n \to \infty}{\to} 0
\]
\end{definition}

\begin{corollary}[Almost full recovery]
\label{cor:almost-full-recovery}
Assume that \(n\delta_n\to\infty\). Then there exists a fully adaptive
polynomial-time estimator that
achieves almost full recovery
whenever $
    \frac{r_n^2}{\delta_n}
    \underset{n \to \infty}{\to}
    \infty. $
    
Conversely, if \(r_n^2/\delta_n\) remains bounded, the minimax clustering
risk is bounded away from zero. Hence \(r_n^2\gg\delta_n\) is the sharp
condition for almost full recovery.
\end{corollary}


\begin{corollary}[Sharp exact recovery threshold]
\label{cor:exact-recovery2}
Assume that  
$
    n\delta
    \ge
    C_0\log n
$
for a sufficiently large universal constant \(C_0\). If
\[
    \liminf_{n\to\infty}
    \frac{
        r_n^2
    }{
        2\log(n\delta)
    }
    >1,
\]
then there exists a fully adaptive polynomial-time estimator
\(\widehat\eta\) that achieves exact recovery. 
Conversely, if
\[
    \limsup_{n\to\infty}
    \frac{
        r_n^2
    }{
        2\log(n\delta)
    }
    <1,
\]
then exact recovery is impossible over the minimax parameter class:
\[
 \liminf_{n\to\infty}
 n\Psi_{n,d}(\delta,\Delta,\sigma)>0.
\]
\end{corollary}

Consequently,
\(
    r_n^2
    =
    2\log(n\delta)
    \bigl(1+o(1)\bigr)
\)
is the sharp first-order exact-recovery threshold.
Thus the effective number of difficult locations is the expected number of
transitions, of order \(n\delta\), rather than the sample size \(n\). In
particular, provided \(n\delta\) is at least logarithmic in $n$, adaptation to both
the unknown signal and the unknown transition probability does not alter the
first-order threshold.

 \begin{supplement}

 Proofs of the oracle and scalar-model results are collected in Appendix A, while the high-dimensional minimax and adaptive results are proved in Appendix B.
 \end{supplement}


\bibliographystyle{spr-ims-nameyear} 

\bibliography{Bibliography.bib}

@article{meila2001experimental,
  title={An experimental comparison of model-based clustering methods},
  author={Meil{\u{a}}, Marina and Heckerman, David},
  journal={Machine learning},
  volume={42},
  pages={9--29},
  year={2001},
  publisher={Springer}
}

@inproceedings{meila2005comparing,
  title={Comparing clusterings: an axiomatic view},
  author={Meil{\u{a}}, Marina},
  booktitle={Proceedings of the 22nd international conference on Machine learning},
  pages={577--584},
  year={2005}
}

@incollection {MR3889693,
    AUTHOR = {Gr\"{u}n, Bettina},
     TITLE = {Model-based clustering},
 BOOKTITLE = {Handbook of mixture analysis},
    SERIES = {Chapman \& Hall/CRC Handb. Mod. Stat. Methods},
     PAGES = {157--192},
 PUBLISHER = {CRC Press, Boca Raton, FL},
      YEAR = {2019},
   MRCLASS = {62H30},
  MRNUMBER = {3889693},
}

@article{Ndaoud18,
author = {Mohamed Ndaoud},
title = {{Sharp optimal recovery in the two component Gaussian mixture model}},
volume = {50},
journal = {Ann. Statist.},
number = {4},
publisher = {Institute of Mathematical Statistics},
pages = {2096 -- 2126},
year = {2022},
doi = {10.1214/22-AOS2178},
}

@book{CMT05,
author = {Capp\'{e}, Olivier and Moulines, Eric and Ryden, Tobias},
title = {Inference in Hidden Markov Models (Springer Series in Statistics)},
year = {2005},
isbn = {0387402640},
publisher = {Springer-Verlag},
}

@article{LH16,
    author = " Lu, Yu and H. Zhou, Harrison ",
    title = "Statistical and Computational Guarantees of Lloyd’s Algorithm and Its Variants",
    journal = {arXiv preprint},
    DOI = {10.48550/arXiv.1612.02099},
    year = "2016",
}

@article{VFLR20,
  title={Optimal change-point detection and localization},
  author={Verzelen, Nicolas and Fromont, Magalie and Lerasle, Matthieu and Reynaud-Bouret, Patricia},
  journal={The Annals of Statistics},
  volume={51},
  number={4},
  pages={1586--1610},
  year={2023},
  publisher={Institute of Mathematical Statistics},
  doi={10.1214/23-AOS2297},
  url={https://projecteuclid.org/journals/annals-of-statistics/volume-51/issue-4/Optimal-change-point-detection-and-localization/10.1214/23-AOS2297.short}
}

@article{chao,
author = {Haoyang Liu and Chao Gao and Richard J. Samworth},
title = {{Minimax rates in sparse, high-dimensional change point detection}},
volume = {49},
journal = {The Annals of Statistics},
number = {2},
publisher = {Institute of Mathematical Statistics},
pages = {1081 -- 1112},
year = {2021},
doi = {10.1214/20-AOS1994},
URL = {https://doi.org/10.1214/20-AOS1994}
}

@misc{yu2020,
      title={A review on minimax rates in change point detection and localisation}, 
      author={Yi Yu},
      year={2020},
      eprint={2011.01857},
      archivePrefix={arXiv},
      primaryClass={math.ST},
      url={https://arxiv.org/abs/2011.01857}, 
}

@article{MVW17,
  title={Clustering subgaussian mixtures by semidefinite programming},
  author={Mixon, Dustin G and Villar, Soledad and Ward, Rachel},
  journal={Information and Inference: A Journal of the IMA},
  volume={6},
  number={4},
  pages={389--415},
  year={2017},
  publisher={Oxford University Press}
}

@inproceedings{KSH05,
  title={The spectral method for general mixture models},
  author={Kannan, Ravindran and Salmasian, Hadi and Vempala, Santosh},
  booktitle={International conference on computational learning theory},
  pages={444--457},
  year={2005},
  organization={Springer}
}

@article{SBY09,
  title={Data spectroscopy: Eigenspaces of convolution operators and clustering},
  author={Shi, Tao and Belkin, Mikhail and Yu, Bin},
  journal={The Annals of Statistics},
  pages={3960--3984},
  year={2009},
  publisher={JSTOR}
}

@book{G21,
  title={Introduction to high-dimensional statistics},
  author={Giraud, Christophe},
  year={2021},
  edition = {Second},
  publisher={Chapman and Hall/CRC}, 
}

@article{KN24,
    title={Adaptive Mean Estimation in the Hidden Markov sub-Gaussian Mixture Model},
    author={Karagulyan, Vahe and Ndaoud, Mohamed},
    year={2024},
    eprint={2406.12446},
    journal={arXiv},
    primaryClass={stat.ST}
}

@article{GKN23,
author = {{\'E}lisabeth Gassiat and Ibrahim Kaddouri and Zacharie Naulet},
title = {{Clustering risk in nonparametric hidden Markov and I.I.D. models}},
volume = {53},
journal = {The Annals of Statistics},
number = {6},
publisher = {Institute of Mathematical Statistics},
pages = {2409 -- 2429},
year = {2025},
doi = {10.1214/25-AOS2546},
URL = {https://doi.org/10.1214/25-AOS2546}
}

@article{WZ21,
  title={Randomly initialized EM algorithm for two-component Gaussian mixture achieves near optimality in O ($\sqrt{n}$) iterations},
  author={Wu, Yihong and Zhou, Harrison H},
  journal={Mathematical Statistics and Learning},
  volume={4},
  number={3},
  year={2021}
}

@article{ZW22,
  title={Mean estimation in high-dimensional binary Markov Gaussian mixture models},
  author={Zhang, Yihan and Weinberger, Nir},
  journal={Advances in Neural Information Processing Systems},
  volume={35},
  pages={19673--19686},
  year={2022}
}

@misc{wu2017lecture,
  author       = {Yihong Wu},
  title        = {Lecture Notes on Information-Theoretic Methods for High-Dimensional Statistics},
  year         = {2017},
  note         = {Lecture Notes for ECE598YW (UIUC), Version 16},
  url          = {https://yihongwu.stat.yale.edu/teaching/ece598yw-fall17/notes.html}
}

@article{BickelRitovRyden1998,
  author  = {Bickel, Peter J. and Ritov, Ya'acov and Ryd{\'e}n, Tobias},
  title   = {Asymptotic normality of the maximum-likelihood estimator
             for general hidden {M}arkov models},
  journal = {The Annals of Statistics},
  year    = {1998},
  volume  = {26},
  number  = {4},
  pages   = {1614--1635},
  doi     = {10.1214/aos/1024691255}
}

@article{DoucMoulinesRyden2004,
  author  = {Douc, Randal and Moulines, {\'E}ric and Ryd{\'e}n, Tobias},
  title   = {Asymptotic properties of the maximum likelihood estimator
             in autoregressive models with {M}arkov regime},
  journal = {The Annals of Statistics},
  year    = {2004},
  volume  = {32},
  number  = {5},
  pages   = {2254--2304},
  doi     = {10.1214/009053604000000021}
}

@article{AlexandrovichHolzmannLeister2016,
  author  = {Alexandrovich, Grigory and Holzmann, Hajo and Leister, Anna},
  title   = {Nonparametric identification and maximum likelihood estimation
             for hidden {M}arkov models},
  journal = {Biometrika},
  year    = {2016},
  volume  = {103},
  number  = {2},
  pages   = {423--434},
  doi     = {10.1093/biomet/asw001}
}

@article{SunCai2009,
  author  = {Sun, Wenguang and Cai, T. Tony},
  title   = {Large-scale multiple testing under dependence},
  journal = {Journal of the Royal Statistical Society:
             Series B (Statistical Methodology)},
  year    = {2009},
  volume  = {71},
  number  = {2},
  pages   = {393--424},
  doi     = {10.1111/j.1467-9868.2008.00694.x}
}

@article{WangSamworth2018,
  author  = {Wang, Tengyao and Samworth, Richard J.},
  title   = {High dimensional change point estimation via sparse projection},
  journal = {Journal of the Royal Statistical Society:
             Series B (Statistical Methodology)},
  year    = {2018},
  volume  = {80},
  number  = {1},
  pages   = {57--83},
  doi     = {10.1111/rssb.12243}
}

@article{YuChen2021,
  author  = {Yu, Mengjia and Chen, Xiaohui},
  title   = {Finite sample change point inference and identification
             for high-dimensional mean vectors},
  journal = {Journal of the Royal Statistical Society:
             Series B (Statistical Methodology)},
  year    = {2021},
  volume  = {83},
  number  = {2},
  pages   = {247--270},
  doi     = {10.1111/rssb.12406}
}

@article{ZhaoJiangShao2022,
  author  = {Zhao, Zifeng and Jiang, Feiyu and Shao, Xiaofeng},
  title   = {Segmenting time series via self-normalisation},
  journal = {Journal of the Royal Statistical Society:
             Series B (Statistical Methodology)},
  year    = {2022},
  volume  = {84},
  number  = {5},
  pages   = {1699--1725},
  doi     = {10.1111/rssb.12552}
}

@InProceedings{Perchet18,
  title = 	 {Hidden Integrality of SDP Relaxations for Sub-Gaussian Mixture Models},
  author =       {Fei, Yingjie and Chen, Yudong},
  booktitle = 	 {Proceedings of the 31st  Conference On Learning Theory},
  pages = 	 {1931--1965},
  year = 	 {2018},
  editor = 	 {Bubeck, Sébastien and Perchet, Vianney and Rigollet, Philippe},
  volume = 	 {75},
  series = 	 {Proceedings of Machine Learning Research},
  month = 	 {06--09 Jul},
  publisher =    {PMLR},
  url = 	 {https://proceedings.mlr.press/v75/fei18a.html}
}

@article{GV19,
  title={Partial recovery bounds for clustering with the relaxed K-means},
  author={Giraud, Christophe and Verzelen, Nicolas},
  DOI={},
  year={2018},
  journal={Math. Stat. Learn.},
  FJOURNAL = {Mathematical Statistics and Learning},
  volume = {3},
  number = {1},
  pages = {317–374},
  }

@inproceedings{Royer17,
 author = {Royer, Martin},
 booktitle = {Advances in Neural Information Processing Systems},
 editor = {I. Guyon and U. Von Luxburg and S. Bengio and H. Wallach and R. Fergus and S. Vishwanathan and R. Garnett},
 pages = {},
 publisher = {Curran Associates, Inc.},
 title = {Adaptive Clustering through Semidefinite Programming},
 url = {https://proceedings.neurips.cc/paper_files/paper/2017/file/3a15c7d0bbe60300a39f76f8a5ba6896-Paper.pdf},
 volume = {30},
 year = {2017}
}
\newpage
\appendix

\section{Oracle and Scalar Model Proofs}
Constants denoted by
\(c,C,c_0,C_0\) are positive numerical constants whose values may change from
line to line. For a binary vector \(\widehat\eta\in\{-1,+1\}^n\), recall that
\[
    \ell(\widehat\eta,\eta)
    =
    \min_{s\in\{-1,+1\}}
    \frac1n\sum_{j=1}^n
    \mathbf 1_{\{s\widehat\eta_j\neq\eta_j\}}.
\]

\subsection{Proof of Proposition~\ref{prop:lower:lag}}
To simplify notation, we write $R_k=\mathcal R_k^{\mathrm{lag}}(\theta,\delta)$. Throughout the proof, constants are numerical and may change from line to line. If \(\theta=0\), the observations contain no information about
the latent labels, and
\[
    R_k=\frac{1-(1-2\delta)^k}{2}
    \asymp k\delta\wedge1.
\]
Henceforth assume \(\theta\neq0\).

\textbf{The case \(k=1\).}
Conditionally on \(\eta_{i-1}\), the two possibilities
\(\eta_i=\eta_{i-1}\) and \(\eta_i=-\eta_{i-1}\) have respective prior
weights \(1-\delta\) and \(\delta\).  Let \(f_1\) and \(f_{-1}\) denote the
corresponding conditional densities of \(X_i\), that is,
\[
        f_1(x)=\frac{1}{(2\pi)^{d/2}}
        \exp\left(-\frac{\|x-\eta_{i-1}\theta\|^2}{2}\right),
        \qquad
        f_{-1}(x)=\frac{1}{(2\pi)^{d/2}}
        \exp\left(-\frac{\|x+\eta_{i-1}\theta\|^2}{2}\right).
\]
Then
\[
        \frac{(1-\delta)f_1(X_i)}{\delta f_{-1}(X_i)}
        =
        \frac{1-\delta}{\delta}
        \exp\bigl(2\eta_{i-1}\langle X_i,\theta\rangle\bigr).
\]
Hence the Bayes rule is
\[
        \eta_i^\star
        =
        \eta_{i-1}\operatorname{sign}\left(
        \eta_{i-1}\langle X_i,\theta\rangle
        +\frac{\log((1-\delta)/\delta)}{2}
        \right).
\]

Therefore,
\begin{align*}
R_1
&\gtrsim
\delta\,
\mathbb P\left(
\eta_{i-1}\langle X_i,\theta\rangle
+
\frac{\log((1-\delta)/\delta)}{2}
\ge 0
\,\middle|\,
\eta_i=-\eta_{i-1}
\right)
\\
&=
\delta\,
\mathbb P\left(
-\lVert\theta\rVert^2+
\eta_{i-1}\langle\xi_i,\theta\rangle
+
\frac{\log((1-\delta)/\delta)}{2}
\ge 0
\right)
\\
&=
\delta\,
\mathbb P\left(
\eta_{i-1}\left\langle\xi_i,\frac{\theta}{\lVert\theta\rVert}\right\rangle
\ge
\lVert\theta\rVert-
\frac{\log((1-\delta)/\delta)}{2\lVert\theta\rVert}
\right)
\\
&= \delta
\Phi^c
\left(
\|\theta\|-\frac{\log((1-\delta)/\delta)}{2\|\theta\|}
\right)\\
&\gtrsim\delta\Phi^c
\left(
\|\theta\|-\frac{\log(1/\delta)}{2\|\theta\|}
\right)
\end{align*}
where the last step uses \(0<\delta\le 1/2\).  Indeed,
\(\log((1-\delta)/\delta)=\log(1/\delta)+\log(1-\delta)\), and
\(-\log(1-\delta)\le \log 2\), so replacing $\log((1-\delta)/\delta)$ by
\(\log(1/\delta)\) changes the displayed Gaussian tail by at most a universal
constant.  In particular, if
\[
        \|\theta\|^2\le \frac{\log(1/\delta)}{2},
\]
then
\[
        R_1\gtrsim \delta.
\]

\textbf{Monotonicity in the lag.}
We next show that \(R_k\) is nondecreasing in \(k\).  Recall that
\[
        R_k
        :=
        \inf_{\widehat\eta}
        \PP
        \bigl(\widehat\eta(X_{1:k},\eta_0)\ne \eta_k\bigr),
\]
where the infimum is over all measurable estimators based on
\((X_{1:k},\eta_0)\). Fix \(1\le m<k\) and write \(k'=k-m\).  Let
\(\widehat\eta=\widehat\eta(X_{1:k},\eta_0)\) be arbitrary.  By the tower
property,
\[
\mathbb P(\widehat\eta\ne\eta_k)
=
\mathbb E\left[
\mathbb P\left(
\widehat\eta(X_{1:k},\eta_0)\ne\eta_k
\,\middle|\,
\eta_0,\ldots,\eta_{k'}
\right)
\right].
\]
Conditionally on \(\eta_0,\ldots,\eta_{k'}\), the future process
\((X_{k'+1:k},\eta_k)\) depends on the past only through \(\eta_{k'}\).
The observations \(X_{1:k'}\) are therefore irrelevant for the conditional
Bayes risk.  Moreover, giving the decision rule the additional
information \(\eta_{k'}\) can only decrease the optimal risk.  Hence
\[
\mathbb P\left(
\widehat\eta(X_{1:k},\eta_0)\ne\eta_k
\,\middle|\,
\eta_0,\ldots,\eta_{k'}
\right)
\ge
\inf_{\psi}
\mathbb P\left(
\psi(X_{k'+1:k},\eta_{k'})\ne\eta_k
\,\middle|\,
\eta_{k'}
\right).
\]
Taking expectations gives
\[
\mathbb P(\widehat\eta\ne\eta_k)
\ge
\inf_{\psi}
\mathbb P\left(
\psi(X_{k'+1:k},\eta_{k'})\ne\eta_k
\right).
\]
By stationarity and time-homogeneity, the last term is equal to
\[
\inf_{\psi}
\mathbb P\left(
\psi(X_{1:m},\eta_0)\ne\eta_m
\right)
=R_m.
\]
Since the bound holds for every \(\widehat\eta\), taking the infimum over
\(\widehat\eta\) yields \(R_k\ge R_m\). In particular,

\begin{equation*}
    \begin{split}
        R_k\geq R_1\gtrsim \delta\Phi^c\left(\|\theta\|-\frac{\log(1/\delta)}{2\|\theta\|}\right).
    \end{split}
\end{equation*}
This gives the claimed lower bound in the strong-signal regime
\(
        \|\theta\|^2
        \ge
        \log(1/\delta)/2.
\)
\textbf{The intermediate regime lower-bound.} 
For this lower-bound we use a binary-testing argument based on the overlap between the no-transition distribution and the mixture of single-transition distributions. Fix an integer \(k\ge1\), and consider an arbitrary estimator
\(\widehat\eta=\widehat\eta(X_{1:k},\eta_0)\) of \(\eta_k\). By symmetry it is
enough to condition on \(\eta_0=a\), with \(a\in\{-1,1\}\). We lower bound
the risk by retaining only the path with no transition and the paths with a
single transition.  For \(j=1,\ldots,k\), let
\[
        E_j(a)=\{\eta_0=\cdots=\eta_{j-1}=a,
                 \ \eta_j=\cdots=\eta_k=-a\},
\]
and let
\[
        E_0(a)=\{\eta_0=\cdots=\eta_k=a\}.
\]
Since \(\mathbb P(E_0(a)\mid \eta_0=a)=(1-\delta)^k\) and
\(\mathbb P(E_j(a)\mid \eta_0=a)=\delta(1-\delta)^{k-1}\), we get
\begin{align*}
\mathbb P(\widehat\eta\ne\eta_k)
&\ge
\frac{(1-\delta)^{k-1}}{2}
\sum_{a\in\{-1,1\}}
\Bigg[
(1-\delta)
\mathbb P\left(
        \widehat\eta(a,X_{1:k})=-a
        \,\middle|\,
        E_0(a)
\right)                                                       \\
&\hspace{5em}
+\delta\sum_{j=1}^k
\mathbb P\left(
        \widehat\eta(a,X_{1:k})=a
        \,\middle|\,
        E_j(a)
\right)
\Bigg].
\end{align*}
Let \(f_\mu\) denote the \(\mathcal N(\mu,I_d)\) density.  The last display is
bounded from below by the value of the binary testing problem between the
density corresponding to \(E_0(a)\) and the mixture of densities corresponding
to \(E_j(a)\), \(1\le j\le k\).  The Bayes rule for this restricted problem is
therefore
\[
\eta_k^\star(a,x_{1:k})
=
a\,\operatorname{sign}\left(
(1-\delta)\prod_{s=1}^k f_{a\theta}(x_s)
-
\delta\sum_{j'=1}^k
\prod_{s=1}^{j'-1}f_{a\theta}(x_s)
\prod_{s=j'}^k f_{-a\theta}(x_s)
\right),
\]
with an arbitrary convention when the expression in parentheses is zero.
Hence, for every estimator \(\widehat\eta\),
\begin{align*}
\mathbb P(\widehat\eta\ne\eta_k)
&\ge
\frac{(1-\delta)^{k-1}}{2}
\sum_{a\in\{-1,1\}}
\delta\sum_{j=1}^k
\mathbb P\left(
        \eta_k^\star(a,X_{1:k})=a
        \,\middle|\,
        E_j(a)
\right)                                                       \\
&\ge
(1-\delta)^{k-1}k\delta
\left(1-
\frac{\delta}{k(1-\delta)}
\sum_{j=1}^k\sum_{j'=1}^k
\mathbb E\left[
        L_{j'}(X_{1:k})
        \,\middle|\,
        E_j(1)
\right]
\right),
\end{align*}
where
\[
        L_{j'}(X_{1:k})
        =
        \frac{
        \prod_{s=1}^{j'-1}f_{\theta}(X_s)
        \prod_{s=j'}^k f_{-\theta}(X_s)
        }{
        \prod_{s=1}^k f_{\theta}(X_s)
        }
        =
        \exp\left(-2\left\langle\theta,\sum_{s=j'}^k X_s\right\rangle\right).
\]
On \(E_j(1)\), we have \(X_s=\eta_s\theta+\xi_s\), with
\(\eta_s\in\{-1,1\}\).  Therefore, for all \(j,j'\),
\begin{align*}
\mathbb E\left[
        L_{j'}(X_{1:k})
        \,\middle|\,
        E_j(1)
\right]
&=
\mathbb E\left[
\exp\left(
        -2\|\theta\|^2\sum_{s=j'}^k\eta_s
        -2\left\langle\theta,\sum_{s=j'}^k\xi_s\right\rangle
\right)
\,\middle|\,E_j(1)
\right]                                                       \\
&\le
\exp(2k\|\theta\|^2)
\mathbb E\exp\left(
        -2\left\langle\theta,\sum_{s=j'}^k\xi_s\right\rangle
\right)                                                       \\
&\le
\exp(4k\|\theta\|^2).
\end{align*}
Consequently,
\[
        R_k
        \ge
        (1-\delta)^{k-1}k\delta
        \left(
        1-\frac{k\delta}{1-\delta}\exp(4k\|\theta\|^2)
        \right).
\]
We now consider all regimes:
\begin{itemize}
    \item If \(k\delta\le c^{\prime}\) and
    \(
        \|\theta\|^2
        \le
        \frac{\log(1/(k\delta))}{5k},
    \)
    for a sufficiently small constant \(c^{\prime}>0\), then the factor in
    parentheses is bounded below by a positive numerical constant and
    \((1-\delta)^{k-1}\) is also bounded below.  Thus
    \[
        R_k\gtrsim k\delta .
    \]
    
    \item Suppose that
    $\|\theta\|^2
        \leq \delta\wedge\frac12\log(1/\delta)$ and $k\delta\geq c'$.
    If \(\delta\geq c'/2\), the one-step lower bound gives
    \[
        R_k\geq R_1\gtrsim\delta\gtrsim1.
    \]
    Otherwise, set $m=\left\lfloor\frac{c'}{\delta}\right\rfloor$.
    Then \(1\leq m\leq k\) and $\frac{c'}2\leq m\delta\leq c'$.
    Taking \(c'\) sufficiently small, we have
    \[
        5m\|\theta\|^2
        \leq5m\delta
        \leq5c'
        \leq\log(1/c')
        \leq\log\frac1{m\delta}.
    \]
    The first bullet therefore applies at lag \(m\), and monotonicity gives
    \[
        R_k\geq R_m\gtrsim m\delta\gtrsim1.
    \]
    Since
    \[
        \frac{\delta}{\|\theta\|^2}
        \log\left(1+\frac{\|\theta\|^2}{\delta}\right)
        \asymp1
        \qquad\text{for }0<\|\theta\|^2\leq\delta,
    \]
    this is the required bound.

    \item Suppose that
    \(
        \delta\leq\|\theta\|^2
        \leq\frac12\log(1/\delta).
    \)
    We first prove
    \[
        R_k\gtrsim
        \delta\left(
            k\wedge
            \frac1{\|\theta\|^2}
            \log\left(1+\frac{\|\theta\|^2}{\delta}\right)
        \right).
    \]
    Fix a sufficiently small numerical constant \(\varepsilon>0\).
    If
    \[
        k\wedge
        \frac1{\|\theta\|^2}
        \log\left(1+\frac{\|\theta\|^2}{\delta}\right)
        \leq\frac2\varepsilon,
    \]
    this follows from \(R_k\geq R_1\gtrsim\delta\).
    Otherwise, choose
    \[
        m=
        \left\lfloor
            \varepsilon\left(
                k\wedge
                \frac1{\|\theta\|^2}
                \log\left(1+\frac{\|\theta\|^2}{\delta}\right)
            \right)
        \right\rfloor.
    \]
    Then \(1\leq m\leq k\),
    \[
        m\asymp
        k\wedge
        \frac1{\|\theta\|^2}
        \log\left(1+\frac{\|\theta\|^2}{\delta}\right),
    \]
    and
    \[
        m\delta
        \leq
        \varepsilon\frac{\delta}{\|\theta\|^2}
        \log\left(1+\frac{\|\theta\|^2}{\delta}\right)
        \leq\varepsilon,
        \qquad
        m\|\theta\|^2
        \leq
        \varepsilon\log\left(1+\frac{\|\theta\|^2}{\delta}\right).
    \]
    Using
    \[
        \frac{\|\theta\|^2}{\delta}
        \geq
        \log\left(1+\frac{\|\theta\|^2}{\delta}\right)
        \sqrt{1+\frac{\|\theta\|^2}{\delta}},
    \]
    we obtain
    \[
    \begin{aligned}
        \log\frac1{m\delta}
        &\geq
        \log\left(
            \frac{\|\theta\|^2}
            {\varepsilon\delta
             \log(1+\|\theta\|^2/\delta)}
        \right) \geq
        \frac12\log\left(1+\frac{\|\theta\|^2}{\delta}\right)
        +\log\left(\frac1\varepsilon\right).
    \end{aligned}
    \]
    Taking \(\varepsilon\leq c'\wedge1/10\), we obtain
    \[
        m\delta\leq c',
        \qquad
        5m\|\theta\|^2
        \leq
        \frac12\log\left(1+\frac{\|\theta\|^2}{\delta}\right)
        \leq\log\frac1{m\delta}.
    \]
    Thus the first bullet and monotonicity yield
    \[
        R_k\geq R_m\gtrsim m\delta
        \asymp
        \delta\left(
            k\wedge
            \frac1{\|\theta\|^2}
            \log\left(1+\frac{\|\theta\|^2}{\delta}\right)
        \right).
    \]
        It remains to check that
    \[
        k\gtrsim
        \frac1{\|\theta\|^2}
        \log\left(1+\frac{\|\theta\|^2}{\delta}\right)
    \]
    under the conditions of the second branch of the proposition.
    Set
    \[
        x=\frac{\|\theta\|^2}{\delta},
        \qquad y=k\|\theta\|^2.
    \]
    Since
    \[
        \|\theta\|^2\ge\frac2k\log\frac1{k\delta},
    \]
    we have
    \[
        y\ge2\log\frac{x}{y},
        \qquad\text{hence}\qquad
        x\le y e^{y/2}.
    \]
    Consequently,
    \[
    \begin{aligned}
        \log(1+x)
        \le \log(1+y e^{y/2})
        \le \frac y2+\log(1+y)
        \le \frac32 y.
    \end{aligned}
    \]
    Therefore,
    \[
        k\ge
        \frac{2}{3\|\theta\|^2}
        \log\left(1+\frac{\|\theta\|^2}{\delta}\right).
    \]
    Combining this with the lower bound already established
    for \(\|\theta\|^2\ge\delta\) yields
    \[
        R_k\gtrsim
        \frac{\delta}{\|\theta\|^2}
        \log\left(1+\frac{\|\theta\|^2}{\delta}\right).
    \]

    Finally, suppose that \(0<\|\theta\|^2<\delta\).
    If \(k\delta\ge c'\), the preceding constant-order
    lower bound applies and gives the required result.
    If instead \(k\delta<c'\), the assumption
    \[
        \|\theta\|^2\ge\frac2k\log\frac1{k\delta}
    \]
    would imply
    \[
        k\delta>k\|\theta\|^2
        \ge2\log\frac1{k\delta},
    \]
    which is impossible for \(c'\le1/2\).
\end{itemize}

We finally prove the weak-signal branch.
Assume
\(
    \|\theta\|^2
    \le
    \frac{2}{k}\log\frac1{k\delta}
    \ \text{and}\
    \|\theta\|^2\le\frac12\log\frac1\delta.\
\)
In particular, \(k\delta\le1\).
Fix \(0<\alpha\le \min\{c'/2,1/10\}\). If \(k<2/\alpha\), the one-step bound gives
\[
    R_k\ge R_1\gtrsim\delta\gtrsim k\delta.
\]
Otherwise, set \(j=\lfloor\alpha k\rfloor\). Then
\(
    \frac{\alpha k}{2}\le j\le\alpha k,\ j\delta\le\alpha\le c',
\)
and
\[
    5j\|\theta\|^2
    \le 5\alpha k\|\theta\|^2
    \le 10\alpha\log\frac1{k\delta}
    \le \log\frac1{j\delta}.
\]
The overlap bound therefore applies at lag \(j\).
By monotonicity,
\[
    R_k\ge R_j\gtrsim j\delta\gtrsim k\delta.
\]
Since \(k\delta\le1\), this proves the first branch as stated.

\subsection{Proof of Proposition \ref{prop:upper:lag}}

We start first by analyzing the weak signal regime.

\textbf{Weak-signal regime.}
If
\[
\|\theta\|^2
\le
\frac{2\log(1/(k\delta))}{k}\wedge\frac{\log(1/\delta)}{2},
\]
take
\[
\widehat\eta_k=\eta_{0}.
\]
Then
\[
\mathbb P(\widehat\eta_k\neq\eta_k)\le k\delta\wedge1.
\]

\textbf{Intermediate regime.}
Assume
\[
\frac{2\log(1/(k\delta))}{k}
<
\|\theta\|^2
\le
\frac{\log(1/\delta)}{2}.
\]
Note that $k\geq 2$ in this regime.
Set
\[
m
=
k\wedge
\left\lceil
\frac{2}{\|\theta\|^2}\log\left(1 + \frac{\|\theta\|^2}{\delta}\right)
\right\rceil.
\]
Consider
\[
\widehat\eta_k
=
\operatorname{sign}
\left(
\left\langle
\frac{1}{m}
\sum_{j=k-m+1}^{k}X_j,
\theta
\right\rangle
\right).
\]
If there is at least one sign change among the hidden labels in this window,
the corresponding event has probability at most $m\delta$.
Otherwise, all observations in the window have the same hidden label, and
\begin{align*}
\mathbb P(\widehat\eta_k\ne\eta_k)
&\le
m\delta
+
\mathbb P\left(
\left\langle
\frac{1}{m}
\sum_{j=k-m+1}^{k}\xi_j,
\frac{\theta}{\|\theta\|}
\right\rangle
+
\|\theta\|
\le0
\right)
\\
&\le
m\delta
+
\exp\left(-\frac{m\|\theta\|^2}{2}\right).
\end{align*}
If
\[
m
=
\left\lceil
\frac{2}{\|\theta\|^2}\log\left(1 + \frac{\|\theta\|^2}{\delta}\right)
\right\rceil
\]
then

\[
        \exp\left(-\frac{m\|\theta\|^2}{2}\right)
        \le
        \left(1+\frac{\|\theta\|^2}{\delta}\right)^{-1}
        \le
        \frac{\delta}{\|\theta\|^2}
        \log\left(1+\frac{\|\theta\|^2}{\delta}\right),
\]
where the last inequality follows from
\(\|\theta\|^2/(\|\theta\|^2+\delta)\le \log(1+\|\theta\|^2/\delta)\).  Moreover,
\[
        m\delta
        \lesssim
        \frac{\delta}{\|\theta\|^2}
        \log\left(1+\frac{\|\theta\|^2}{\delta}\right),
\]
using the ceiling and the fact that
\(\|\theta\|^2\le \frac{\log(1/\delta)}{2}\). Hence, the desired bound follows in this case.

If instead \(m=k\), then the lower bound defining the intermediate regime gives
\[
        \exp\left(-\frac{k\|\theta\|^2}{2}\right)
        \le
        k\delta.
\]
Also, from the definition of \(m\),
\[
        k
        \le
        \left\lceil
        \frac{2}{\|\theta\|^2}
        \log\left(1+\frac{\|\theta\|^2}{\delta}\right)
        \right\rceil,
\]
and therefore, again using \(\|\theta\|^2\le \frac{\log(1/\delta)}{2}\),
\[
        k\delta
        \lesssim
        \frac{\delta}{\|\theta\|^2}
        \log\left(1+\frac{\|\theta\|^2}{\delta}\right).
\]
Thus in both cases,
\[
        \mathbb P(\widehat\eta_k\ne\eta_k)
        \lesssim
        \frac{\delta}{\|\theta\|^2}
        \log\left(1+\frac{\|\theta\|^2}{\delta}\right).
\]
\textbf{Strong signal regime.}
Assume
\[
        \|\theta\|^2
        \ge
        \frac{\log(1/\delta)}{2} .
\]
For \(k=1\), the one-step Bayes rule gives, by the calculation in the lower
bound proof,
\[
        \mathbb P(\widehat\eta_1\ne\eta_1)
        \lesssim
        \delta \Phi^c\left(
        \|\theta\|-
        \frac{\log(1/\delta)}{2\|\theta\|}
        \right) .
\]
We prove that the same bound propagates along the recursion.  More precisely,
for a sufficiently large numerical constant \(C_0\) and for $q = \Phi^c\left(
        \|\theta\|-
        \frac{\log(1/\delta)}{2\|\theta\|}
        \right)$, we claim that
\[
        \mathbb P(\widehat\eta_i\ne\eta_i)
        \le
        C_0\delta q,
        \qquad 1\le i\le k .
\]
The claim holds for \(i=1\).  Assume that it holds at time \(i-1\).  Then
\[
        \mathbb P(\widehat\eta_{i-1}\ne\eta_i)
        \le
        \mathbb P(\widehat\eta_{i-1}\ne\eta_{i-1})
        +\mathbb P(\eta_{i-1}\ne\eta_i)
        \le
        C_0\delta q+\delta .
\]
Conditionally on the past and on the hidden transition at time \(i\), the noise
\(\xi_i\) is independent and Gaussian.  If
\(\widehat\eta_{i-1}=\eta_i\), then the conditional error probability is
\[
        \Phi^c\left(\|\theta\|+\frac{\log((1-\delta)/\delta)}{2\|\theta\|}\right),
\]
whereas if \(\widehat\eta_{i-1}\ne\eta_i\), it is
\[
        \Phi^c\left(\|\theta\|-\frac{\log((1-\delta)/\delta)}{2\|\theta\|}\right).
\]
Therefore
\begin{align*}
\mathbb P(\widehat\eta_i\ne\eta_i)
&\le
\Phi^c\left(\|\theta\|+\frac{\log((1-\delta)/\delta)}{2\|\theta\|}\right)
+
(C_0\delta q+\delta)
\Phi^c\left(\|\theta\|-\frac{\log((1-\delta)/\delta)}{2\|\theta\|}\right).
\end{align*}
Put
\(
    b=\frac12\log\left((1-\delta)/\delta\right).
\)
Since \(b\leq \frac12\log(1/\delta)\),
\[
    \Phi^c\left(\|\theta\|-\frac{b}{\|\theta\|}\right)\leq q.
\]
Moreover, by shifting the Gaussian density,
\[
\begin{aligned}
    \Phi^c\left(\|\theta\|+\frac{b}{\|\theta\|}\right)
    =
    \int_{\|\theta\|-b/\|\theta\|}^{\infty}
        \phi\left(u+\frac{2b}{\|\theta\|}\right)\,du \leq
    e^{-2b}\Phi^c\left(\|\theta\|-\frac{b}{\|\theta\|}\right)
    \leq 2\delta q,
\end{aligned}
\]
where we used \(e^{-2b}=\delta/(1-\delta)\leq2\delta\).
Writing \(e_i=\mathbb P(\widehat\eta_i\neq\eta_i)\), the preceding
conditional-error decomposition therefore gives
\[
    e_i\leq 3\delta q+q e_{i-1},
    \qquad i\geq2.
\]
For the one-step Bayes initialization,
\[
    e_1
    =(1-\delta)\Phi^c\left(\|\theta\|+\frac{b}{\|\theta\|}\right)
      +\delta\Phi^c\left(\|\theta\|-\frac{b}{\|\theta\|}\right)
    \leq 2\delta q.
\]
In the present regime,
\(\|\theta\|^2\geq\frac12\log(1/\delta)\), so \(q\leq1/2\).
Consequently, \(e_{i-1}\leq6\delta q\) implies
\[
    e_i\leq3\delta q+6\delta q^2\leq6\delta q.
\]
Induction yields
\[
    \mathbb P(\widehat\eta_k\neq\eta_k)
    \leq6\delta\,
    \Phi^c\left(
        \|\theta\|-\frac{\log(1/\delta)}{2\|\theta\|}
    \right),
\]
which proves the strong-signal upper bound.

\subsection{Proof of Theorem~\ref{thm:lag:risk}}
Write
\[
        R_k=\mathcal R_k^{\mathrm{lag}}(\theta,\delta).
\]
The strong signal regime follows from Propositions~\ref{prop:lower:lag} and
\ref{prop:upper:lag}.  We prove the weak and intermediate signal regimes.

\textbf{Weak signal regime.}
Assume that
\[
        \|\theta\|^2<\delta\wedge\frac{\log(1/\delta)}{2}.
\]
For the upper bound, take \(\widehat\eta_k=\eta_0\).  Then
\[
        \mathbb P(\widehat\eta_k\ne\eta_k)
        \le k\delta\wedge1,
\]
and therefore
\[
        R_k\lesssim
        \left(k\wedge\frac1\delta\right)\delta.
\]

For the lower bound, recall that the overlap argument of Proposition \ref{prop:lower:lag} gives, for a sufficiently
small numerical constant \(c_0>0\),
\[
        R_m\gtrsim m\delta
\]
whenever
\[
        m\delta\le c_0
        \qquad\text{and}\qquad
        \|\theta\|^2
        \le
        \frac{\log(1/(m\delta))}{5m}.
\]
If \(k\delta\le c_0\), these conditions hold with \(m=k\), after decreasing
\(c_0\) if necessary, because \(\|\theta\|^2<\delta\).  Hence
\[
        R_k\gtrsim k\delta.
\]
If $k\delta>c_0$ and $\delta<c_0/2$, choose
\(
m=\left\lfloor\frac{c_0}{\delta}\right\rfloor.
\)
Then $1\le m\le k$ and
\(
\frac{c_0}{2}\le m\delta\le c_0.
\)
Taking $c_0$ sufficiently small, the preceding overlap bound
applies and gives
\[
R_k\ge R_m\gtrsim m\delta\gtrsim1.
\]
If instead $\delta\ge c_0/2$, the one-step lower bound gives
\[
R_k\ge R_1
\ge\delta\Phi^c(\|\theta\|)
\ge\frac{c_0}{2}\Phi^c(1),
\]
because $\|\theta\|^2<\delta\le1/2$. Thus the same constant
lower bound holds in this case.
Thus
\[
        R_k\asymp
        \left(k\wedge\frac1\delta\right)\delta.
\]

\textbf{Intermediate signal regime.}
Assume that
\[
        \delta\le \|\theta\|^2
        \le
        \frac{\log(1/\delta)}{2}.
\]
Recall that
\[
        k^\star
        =
        \left\lceil
        \frac{2}{\|\theta\|^2}
        \log\left(1+\frac{\|\theta\|^2}{2\delta}\right)
        \right\rceil .
\]

For the upper bound, if \(k<k^\star\), take \(\widehat\eta_k=\eta_0\).  Then
\[
        \mathbb P(\widehat\eta_k\ne\eta_k)
        \le k\delta.
\]
If \(k\ge k^\star\), take
\[
        \widehat\eta_k
        =
        \operatorname{sign}\left(
        \left\langle
        \frac1{k^\star}\sum_{j=k-k^\star+1}^{k}X_j,
        \theta
        \right\rangle
        \right).
\]
The probability of at least one transition in the averaging window is at most
\(k^\star\delta\).  On the complementary event, all labels in the window are
constant and equal to \(\eta_k\), so
\[
\mathbb P(\widehat\eta_k\ne\eta_k)
\le
k^\star\delta
+
\Phi^c\bigl(\sqrt{k^\star}\,\|\theta\|\bigr).
\]
By the definition of \(k^\star\),
\[
        \Phi^c\bigl(\sqrt{k^\star}\,\|\theta\|\bigr)
        \le
        \exp\left(-\frac{k^\star\|\theta\|^2}{2}\right)
        \le
        \left(1+\frac{\|\theta\|^2}{2\delta}\right)^{-1}
        \lesssim
        k^\star\delta.
\]
Consequently,
\[
        R_k\lesssim (k\wedge k^\star)\delta.
\]

For the lower bound, recall again that
\[
        R_m\gtrsim m\delta
\]
whenever
\[
        m\delta\le c_0
        \qquad\text{and}\qquad
        \|\theta\|^2
        \le
        \frac{1}{5m}\log\frac{1}{m\delta}.
\]
If \(k\wedge k^\star\) is bounded by a sufficiently large numerical constant,
then the one-step bound \(R_k\ge R_1\gtrsim\delta\) already gives
\[
        R_k\gtrsim (k\wedge k^\star)\delta.
\]
We may therefore assume that \(k\wedge k^\star\) is larger than this constant.
Fix a sufficiently small numerical constant \(a>0\) and set
\[
        m=\left\lfloor a\,(k\wedge k^\star)\right\rfloor.
\]
Then
\[
        1\le m\le k,
        \qquad
        m\asymp k\wedge k^\star,
        \qquad
        m\le a k^\star.
\]
We verify explicitly that the two conditions of the overlap bound hold at
scale \(m\).

Set
\[
        x=\frac{\|\theta\|^2}{2\delta},
        \qquad
        H=\log(1+x).
\]
Since \(\|\theta\|^2\ge\delta\), we have \(x\ge1/2\).  Furthermore, throughout
the present regime,
\[
        \|\theta\|^2\le 2H.
\]
Indeed, this is immediate when \(\|\theta\|^2\le2\log(3/2)\), since
\(H\ge\log(3/2)\); otherwise the bound
\(\|\theta\|^2\le \frac12\log(1/\delta)\) implies
\(\delta\le e^{-2\|\theta\|^2}\), and hence
\[
        H
        =
        \log\left(1+\frac{\|\theta\|^2}{2\delta}\right)
        \ge
        \log\left(\frac{\|\theta\|^2}{2}e^{2\|\theta\|^2}\right)
        \ge
        \frac12\|\theta\|^2.
\]
By the definition of \(k^\star\),
\[
        k^\star\|\theta\|^2
        \le
        2H+\|\theta\|^2
        \le
        4H.
\]
Consequently,
\[
        m\|\theta\|^2
        \le
        4aH.
\]
Also,
\[
\begin{aligned}
        k^\star\delta
        &\le
        \delta+\frac{2\delta}{\|\theta\|^2}H
        =
        \delta+\frac{H}{x}.
\end{aligned}
\]
Hence
\[
        m\delta
        \le
        a k^\star\delta
        \le
        \frac{2aH}{x}.
\]
The function \(H/x=\log(1+x)/x\) is bounded on \([1/2,\infty)\).  Thus,
after choosing \(a>0\) sufficiently small,
\[
        m\delta\le c_0.
\]

It remains to verify the signal condition.  From the preceding bound,
\[
        \log\frac1{m\delta}
        \ge
        \log\frac{x}{2aH}.
\]
Since \(x=e^H-1\) and
\[
        e^H-1
        \ge He^{H/2},
\]
we obtain
\[
        \log\frac{x}{2aH}
        \ge
        \frac H2+\log\frac1{2a}.
\]
In particular, if \(a\le1/2\), then
\[
        \log\frac1{m\delta}
        \ge
        \frac H2.
\]
On the other hand,
\[
        5m\|\theta\|^2
        \le
        20aH.
\]
Choosing \(a\le1/40\), and decreasing it further if necessary to ensure
\(m\delta\le c_0\), gives
\[
        5m\|\theta\|^2
        \le
        \frac H2
        \le
        \log\frac1{m\delta}.
\]
Equivalently,
\[
        \|\theta\|^2
        \le
        \frac{1}{5m}\log\frac1{m\delta}.
\]
Thus the overlap bound applies at scale \(m\).  By monotonicity,
\[
        R_k\ge R_m\gtrsim m\delta
        \asymp
        (k\wedge k^\star)\delta.
\]
Combining the upper and lower bounds proves
\[
        R_k\asymp (k\wedge k^\star)\delta
\]
throughout the intermediate signal regime.

\subsection{On the equivalence between the Bayes risks of classification and clustering}\label{sec:equiv}
    It was shown in \cite{GKN23} that the Bayes risks of clustering and classification ensure:
  \begin{align*}
      (1 - \widetilde{\alpha}_n) \inf_{h \in \mathcal{H}_n}\ClassifRisk(\theta,h) \leq \inf_{g \in \mathcal{G}_n}\ClusterRisk(\theta,g) \leq \inf_{h \in \mathcal{H}_n}\ClassifRisk(\theta,h),
  \end{align*}
  where $\widetilde{\alpha}_n = \mathcal{O}\left(\frac{1}{n^{1/2}\delta^{5}}\right)$. This ensures that as long as $\delta \gtrsim n^{-1/10}$, the two Bayes risks are of the same magnitude. This applies also to the online Bayes risks. The following
lemma and corollary sharpen this comparison in the offline two-state setting:
they replace the preceding requirement by the weaker sufficient condition
\[
        \delta\gtrsim \sqrt{\frac{\log n}{n}} .
\]

  \begin{lemma}\label{lem:Bayes_risk_new_lower_bound}
    Under model $\eqref{eq:model_hmm}$, the Bayes risks ensure
    \begin{equation*}
        \begin{split}
            \inf_{g \in \mathcal{G}_n}\ClusterRisk(\theta, \delta, g) &\geq\left(1 - \frac{c}{\delta\sqrt{n}}\right) \inf_{h \in \mathcal{H}_n}\ClassifRisk(\theta, \delta, h) -c^{\prime}e^{-c^{\prime\prime} n \delta^2}\\
            \inf_{g\in \mathcal{G}_n} \ClusterRisk(\theta, \delta, g) &\geq \frac{1}{5} \inf_{h\in \mathcal{H}_n} \ClassifRisk(\theta, \delta, h) - \frac{1}{5}e^{-n \left(\frac{\|\theta\|^2}{8} - \log 2\right)}
        \end{split}
    \end{equation*}
    where $c$, $c^\prime$ and $c^{\prime\prime}$ are absolute positive constants.
    \end{lemma}

\begin{proof}
In what follows, when the indices are not specified, $X$ and $\eta$ stand for $X_{1:n}$ and $\eta_{1:n}$.\\
The first inequality follows from \cite[Theorem~11]{GKN23}. We focus on the second inequality. If $\theta=0$, the claimed bound is trivial. Hence, assume in the sequel that $\theta\neq0$.
For the Bayes clusterer (a minimizer of the risk of clustering) $g^\star_{\theta, \delta} \in \mathcal{G}_n$, there exists a classifier $h^{\star}_{\theta, \delta}\in\mathcal{H}_n$ such that $g^\star_{\theta, \delta} = \pi_n\circ h^\star_{\theta, \delta}$ (See Section \ref{sec:setting_def} for notations). The Bayes risk of clustering ensures 
$$
\inf_{g\in \mathcal{G}_n} \ClusterRisk(\theta, \delta, g) = \EE\left[\ell\left(h^{\star}_{\theta, \delta}(X), \eta\right)\right].
$$ 
In what follows, we omit $\theta$ and $\delta$ from the notation of $g^\star_{\theta, \delta}$ and $h^\star_{\theta, \delta}$ for simplicity.
Define
\begin{equation*}
    s^\star(X, \eta) := \argmin_{s \in \{-1,1\}} \frac{1}{n}\sum_{i=1}^n \mathbf{1}_{h^*_{i}(X) \neq s\eta_i}
\end{equation*}
as being the oracle sign that aligns the optimal classifier to the true hidden labels. Note that $s^\star(X, \eta) = \text{sign}\left( \sum_{i=1}^n \eta_i h_i^*(X) \right)$. This can be naturally estimated by projecting the observations onto the signal direction $\theta$:
\[
\hat{s}(X) = \text{sign}\left( \sum_{i=1}^n \langle X_i,\theta\rangle h_i^*(X) \right).
\]
We then define a plug-in classifier $\hat{h}(X) = \hat{s}(X) \cdot h^*(X)$. One has:
\begin{equation*}
    \begin{split}
        \inf_{h\in \mathcal{H}_n} \ClassifRisk(\theta, \delta, h) &\leq \EE\left[\frac{1}{n}\sum_{i=1}^n\mathbf{1}(\hat{h}_i(X)\ne \eta_i)\right]\\
        &\leq \EE\left[\frac{1}{n}\sum_{i=1}^n\mathbf{1}(s^\star(X, \eta) h_i^{\star}(X)\ne \eta_i)\right] + \PP\left(\hat{s}(X) \neq s^\star(X, \eta)\right)\\
        &= \inf_{g\in \mathcal{G}_n} \ClusterRisk(\theta, \delta, g) + \PP\left(\hat{s}(X) \neq s^\star(X, \eta)\right).
    \end{split}
\end{equation*}
Note that 
\begin{equation}
\label{eq:sign_control}
    \PP(\hat{s}(X) \neq s^*(X, \eta)) \le \PP\left(|\langle h^{\star}(X), \eta\rangle| < \frac{n}{2}\right) + \PP\left(\left\{\hat{s}(X) \neq s^\star(X, \eta)\right\} \cap \left\{|\langle h^{\star}(X), \eta\rangle| \ge \frac{n}{2}\right\}\right).
\end{equation}
For the first term of \eqref{eq:sign_control}, since 
\begin{equation*}
    \begin{split}
        |\langle h^{\star}(X), \eta\rangle| &= |n - 2\sum_{i=1}^{n}\mathbf{1}_{h^{\star}(X)\neq \eta_i}|
        = n(1 - 2\ell\left(h^{\star}(X), \eta\right)),
    \end{split}
\end{equation*}
one has
\[
\PP\left(|\langle h^{\star}(X), \eta\rangle| < \frac{n}{2}\right) = \PP\left(\ell\left(h^{\star}(X), \eta\right) > \frac{1}{4}\right) \le 4\EE[\ell\left(h^{\star}(X), \eta\right)] = 4 \inf_{g\in \mathcal{G}_n} \ClusterRisk(\theta, \delta, g).
\]
For the second term, since $\text{sign}\left(\langle h^{\star}(X), \eta\rangle\right) = s^\star(X, \eta)$ and $X_i=\theta\eta_i+\xi_i$, 
\begin{equation*}
    \begin{split}
        \{\hat{s}(X) \neq s^\star(X, \eta)\} &= \left\{\text{sign}\left(\langle h^{\star}(X), \eta\rangle\right) \sum_{i=1}^n \langle X_i,\theta\rangle h_i^\star(X) \le 0\right\}\\
         &=\left\{\|\theta\|^2 |\langle h^{\star}(X), \eta\rangle| + \text{sign}\left(\langle h^{\star}(X), \eta\rangle\right) \sum_{i=1}^n \langle \xi_i,\theta\rangle h_i^\star(X) \le 0\right\}.
    \end{split}
\end{equation*}
To control the second term, we consider a union bound over all the $2^n$ possible values of $h^{\star}(X)$ in order to avoid dealing with dependent terms:
\begin{equation*}
    \begin{split}
        &\PP\left(\left\{\hat{s}(X) \neq s^\star(X, \eta)\right\} \cap \left\{|\langle h^{\star}(X), \eta\rangle| \ge \frac{n}{2}\right\}\right)\\
        &\leq\sum_{h\in\left\{-1, 1\right\}^n}\PP\left( \left\{\text{sign}(\langle h, \eta\rangle) \sum_{i=1}^n \langle \xi_i,\theta\rangle h_i \le -\|\theta\|^2 |\langle h, \eta\rangle|\right\}\cap\left\{|\langle h, \eta\rangle| \ge \frac{n}{2}\right\}\right).
    \end{split}
\end{equation*}
For any fixed $h \in \left\{-1, 1\right\}^n$,
\begin{equation*}
    \text{sign}(\langle h, \eta\rangle) \sum_{i=1}^n \langle \xi_i,\theta\rangle h_i \mid \eta \sim \mathcal{N}(0, n\|\theta\|^2).
\end{equation*}
Thus, on the event $\left\{|\langle h, \eta\rangle| \ge \frac{n}{2}\right\}$,
\begin{equation*}
    \begin{split}
        \PP\left(\text{sign}(\langle h, \eta\rangle) \sum_{i=1}^n \langle \xi_i,\theta\rangle h_i \le -\|\theta\|^2 |\langle h, \eta\rangle| \mid\eta\right) &\le e^{-\frac{n\|\theta\|^2}{8}}.
    \end{split}
\end{equation*}
Consequently,
\[
\PP\left(\left\{\hat{s}(X) \neq s^\star(X, \eta)\right\} \cap \left\{|\langle h^{\star}(X), \eta\rangle| \ge \frac{n}{2}\right\}\right) \le e^{-n \left(\frac{\|\theta\|^2}{8} - \log 2\right)}.
\]
We finally obtain
\begin{equation*}
    \inf_{h\in \mathcal{H}_n} \ClassifRisk(\theta, \delta, h) \le 5 \inf_{g\in \mathcal{G}_n} \ClusterRisk(\theta, \delta, g) + e^{-n \left(\frac{\|\theta\|^2}{8} - \log 2\right)}.
\end{equation*}

\end{proof}
    \begin{corollary}
\label{prop:equivalence-from-two-comparisons}
 Suppose that
\[
    n\delta^2 \geq C\log\!\left(\frac1\delta\right)
\]
for a sufficiently large universal constant $C>0$. Then there exist universal
constants $0<c<\infty$ such that,
\[
    c\,\inf_{h\in\mathcal H_n}\ClassifRisk(\theta,\delta,h)
    \leq
    \inf_{g\in\mathcal G_n}\ClusterRisk(\theta,\delta,g)
    \leq
    \inf_{h\in\mathcal H_n}\ClassifRisk(\theta,\delta,h).
\]
In particular, the equivalence holds when $\delta \gtrsim\sqrt{\frac{\log(n)}{n}}$.
\end{corollary}

\begin{proof}
The upper bound is immediate. It remains to prove the reverse inequality up to a universal constant.
By Lemma~\ref{lem:Bayes_risk_new_lower_bound},
\begin{equation}
\label{eq:first-comparison-after-lemma}
    \inf_{g\in\mathcal G_n}\ClusterRisk(\theta,\delta,g)
    \geq
    \left(1-\frac{c}{\delta\sqrt n}\right)
    \inf_{h\in\mathcal H_n}\ClassifRisk(\theta,\delta,h)
    -c'e^{-c''n\delta^2},
\end{equation}
and
\begin{equation}
\label{eq:second-comparison-after-lemma}
    \inf_{g\in\mathcal G_n}\ClusterRisk(\theta,\delta,g)
    \geq
    \frac15 \inf_{h\in\mathcal H_n}\ClassifRisk(\theta,\delta,h)
    -\frac15 e^{-n\left(\frac{\|\theta\|^2}{8}-\log 2\right)}.
\end{equation}
Choose $C$ large enough so that
\[
    \frac{c}{\delta\sqrt n}\leq \frac14
    \qquad\text{and}\qquad
    c'e^{-c''n\delta^2}\leq c_1\delta,
\]
where $c_1>0$ is a sufficiently small universal constant. This is possible
because $n\delta^2\geq C\log(1/\delta)$.

We first consider the range $\|\theta\|^2\leq 8\log 2+1$.
In this range, the same argument in the proof of Theorem \ref{Bayes_risk_lb_offline} (See Equations \eqref{eq1:lb:offline} and \eqref{eq2:lb:offline}) implies
\[
    \inf_{h\in\mathcal H_n}\ClassifRisk(\theta,\delta,h)
    \gtrsim
    \delta.
\]
Thus, by taking $c_1$ sufficiently small,
\[
    c'e^{-c''n\delta^2}
    \leq
    \frac14 \inf_{h\in\mathcal H_n}\ClassifRisk(\theta,\delta,h).
\]
Plugging this into \eqref{eq:first-comparison-after-lemma} gives
\[
    \inf_{g\in\mathcal G_n}\ClusterRisk(\theta,\delta,g)
    \geq
    \frac12 \inf_{h\in\mathcal H_n}\ClassifRisk(\theta,\delta,h)
\]
after changing universal constants.

It remains to treat the range \(\|\theta\|^2>8\log 2+1\). Since \(\|\theta\|^2>8\log 2+1\), there exists a universal
constant \(\alpha>0\) such that
\[
    \frac{\|\theta\|^2}{8}-\log 2 \geq \alpha \|\theta\|^2.
\]
Hence
\[
    e^{-n\left(\frac{\|\theta\|^2}{8}-\log 2\right)}
    \leq e^{-\alpha n\|\theta\|^2}.
\]
On the other hand, the same argument used in the proof of Theorem~\ref{Bayes_risk_lb_offline} (See Equations \eqref{eq1:lb:offline} and \eqref{eq2:lb:offline}) implies
\[
    \inf_{h\in\mathcal H_n}\ClassifRisk(\theta,\delta,h)
    \geq
    c_0\delta
    \frac{\|\theta\|}{\|\theta\|^2+1}
    e^{-\|\theta\|^2/2}.
\]
Choosing the universal constant \(C\) in the condition
\(n\delta^2\geq C\log(1/\delta)\) sufficiently large, we have
\[
    e^{-\alpha n\|\theta\|^2}
    \leq
    \frac{c_0}{2}\delta
    \frac{\|\theta\|}{\|\theta\|^2+1}
    e^{-\|\theta\|^2/2}.
\]
Consequently,
\[
    \frac15
    e^{-n\left(\frac{\|\theta\|^2}{8}-\log 2\right)}
    \leq e^{-\alpha n\|\theta\|^2}
    \leq
    \frac1{10}
    \inf_{h\in\mathcal H_n}\ClassifRisk(\theta,\delta,h).
\]
Using \eqref{eq:second-comparison-after-lemma}, we obtain
\[
    \inf_{g\in\mathcal G_n}\ClusterRisk(\theta,\delta,g)
    \geq
    \frac1{10}\inf_{h\in\mathcal H_n}\ClassifRisk(\theta,\delta,h).
\]
Combining the bounded-signal and strong-signal cases yields
\[
    \inf_{g\in\mathcal G_n}\ClusterRisk(\theta,\delta,g)
    \geq
    c \inf_{h\in\mathcal H_n}\ClassifRisk(\theta,\delta,h)
\]
for a universal constant $c>0$. 
\end{proof}
    Since it is easier to study the Bayes risk of classification, we will focus in the offline clustering lower-bound proofs on the regime $\delta\gtrsim\sqrt{\frac{\log(n)}{n}}$: in this case, the bounds on the Bayes risk of classification can be directly extrapolated to the Bayes risk of clustering.

\subsection{Proof of Theorem~\ref{Bayes_risk_lb_online}}
Let
\[
        R_k=\mathcal R_k^{\mathrm{lag}}(\theta,\delta)
\]
denote the Bayes risk of the
\(k\)-lag prediction problem.  We first lower bound the online classification
risk in terms of these lag risks.  By stationarity of the hidden chain, we may
extend the process to two-sided time.  Fix an integer \(1\le k<n\).  Then
\begin{align*}
\inf_{\mathfrak h\in\mathcal H^{\mathrm{on}}}
\ClassifRiskOn(\theta,\delta,\mathfrak h)
&\ge
\frac1n\sum_{i=k+1}^n
\inf_{\widehat\eta_i(\eta_{i-k},X_{1:i})}
\mathbb P\bigl(\widehat\eta_i(\eta_{i-k},X_{1:i})\ne\eta_i\bigr).
\end{align*}
By stationarity,
\begin{equation}
\label{eq:online-class-lag-lower}
\inf_{\mathfrak h\in\mathcal H^{\mathrm{on}}}
\ClassifRiskOn(\theta,\delta,\mathfrak h)
\ge
\left(1-\frac{k}{n}\right)R_k .
\end{equation}

We now choose the lag.  The overlap bound established in
Proposition~\ref{prop:lower:lag} implies that there are numerical constants
\(c_0,c_1>0\) such that, whenever
\[
        k\delta\le c_0
        \qquad\text{and}\qquad
        \|\theta\|^2\le c_1\frac{\log(1/(k\delta))}{k},
\]
one has \(R_k\gtrsim k\delta\).  In particular, if
\(\|\theta\|^2\le 2\delta\), choose \(k\asymp1/\delta\) with
\(k\delta\le c_0\).  Then the preceding condition is satisfied, and
\(R_k\gtrsim1\).  Since \(\delta\ge c n^{-1/10}\), this choice has
\(k=o(n)\), and \eqref{eq:online-class-lag-lower} gives
\[
\inf_{\mathfrak h\in\mathcal H^{\mathrm{on}}}
\ClassifRiskOn(\theta,\delta,\mathfrak h)
\gtrsim 1 .
\]

Next suppose $2\delta<\|\theta\|^2\le \frac{\log(1/\delta)}{2}$.
If $\frac{1}{\|\theta\|^2}\left(\log\left(\frac{\|\theta\|^2}{2\delta}\right)+1\right)$ is bounded, then the one-step lower bound gives \(R_1\gtrsim\delta\), which is
already of the desired order.  We may therefore assume that this quantity is
larger than a sufficiently large numerical constant.

Fix a sufficiently small absolute constant \(a>0\), and choose
\[
        k
        =
        \left\lfloor
        \frac{a}{\|\theta\|^2}
        \left(\log\left(\frac{\|\theta\|^2}{2\delta}\right)+1\right)
        \right\rfloor.
\]
Under the preceding assumption, the expression inside the floor is at least
\(2\), after increasing the numerical threshold if necessary.  Hence
\[
        k
        \asymp
        \frac{1}{\|\theta\|^2}
        \left(\log\left(\frac{\|\theta\|^2}{2\delta}\right)+1\right).
\]

We now verify explicitly the two conditions required by the overlap bound.

Since $\frac{\|\theta\|^2}{2\delta}>1$,
\[
k\delta\leq a\frac{\delta}{\|\theta\|^2}
        \left(\log\left(\frac{\|\theta\|^2}{2\delta}\right)+1\right) \leq \frac{a}{2}
\]
Thus \(k\delta\le c_0\) after choosing \(a\le2c_0\).

Let $x = \frac{\|\theta\|^2}{2\delta}$ and $u = \log(x)$. The same upper bound on \(k\delta\) gives
\[
        \frac1{k\delta}
        \ge
        \frac{2x}{a(1+\log x)},
\]
and hence
\[
\begin{aligned}
        \log\frac1{k\delta}
        &\ge
        \log\left(\frac{2x}{a(1+\log x)}\right) =
        u-\log(1+u)+\log\frac2a.
\end{aligned}
\]
We claim that, for \(a\le1\), the last expression is at least
\(\frac12(1+u)\).  Indeed,
\[
\begin{aligned}
&u-\log(1+u)+\log\frac2a-\frac12(1+u) = \frac u2-\log(1+u)+\log\frac2a-\frac12.
\end{aligned}
\]
The function $u\longmapsto \frac u2-\log(1+u)$ attains its minimum over \([0,\infty)\) at \(u=1\), where its value is
\(\frac12-\log2\).  Consequently,
\[
\begin{aligned}
\frac u2-\log(1+u)+\log\frac2a-\frac12
&\ge
\left(\frac12-\log2\right)+\log\frac2a-\frac12 =
\log\frac1a
\ge0.
\end{aligned}
\]
It follows that
\[
        \log\frac1{k\delta}
        \ge
        \frac12
        \left(\log\left(\frac{\|\theta\|^2}{2\delta}\right)+1\right).
\]
On the other hand, the definition of \(k\) gives
\[
        k\|\theta\|^2
        \le
        a\left(\log\left(\frac{\|\theta\|^2}{2\delta}\right)+1\right).
\]
Therefore, after choosing \(a\le1/10\),
\[
\begin{aligned}
        5k\|\theta\|^2 \le
        5a\left(\log\left(\frac{\|\theta\|^2}{2\delta}\right)+1\right) \le
        \frac12
        \left(\log\left(\frac{\|\theta\|^2}{2\delta}\right)+1\right) \le
        \log\frac1{k\delta}.
\end{aligned}
\]
Equivalently,
\[
        \|\theta\|^2
        \le
        \frac{1}{5k}\log\frac1{k\delta}.
\]
Thus both conditions of the overlap bound hold at the chosen scale \(k\), and
we obtain
\[
        R_k
        \gtrsim
        k\delta
        \asymp
        \frac{\delta}{\|\theta\|^2}
        \left(\log\left(\frac{\|\theta\|^2}{2\delta}\right)+1\right).
\]
Since \(\delta\ge c n^{-1/10}\),
\[
        k
        \lesssim
        \frac{1}{\delta}\log\frac1\delta
        =o(n).
\]
Applying
\eqref{eq:online-class-lag-lower}, we obtain
\[
\inf_{\mathfrak h\in\mathcal H^{\mathrm{on}}}
\ClassifRiskOn(\theta,\delta,\mathfrak h)
\gtrsim
\frac{\delta}{\|\theta\|^2}
\left(\log\left(\frac{\|\theta\|^2}{2\delta}\right)+1\right).
\]

Finally, in the strong-signal regime $\|\theta\|^2\ge \frac{\log(1/\delta)}{2}$,
we take \(k=1\) in \eqref{eq:online-class-lag-lower}.  The one-step testing
bound from Proposition~\ref{prop:lower:lag} yields
\[
        R_1
        \gtrsim
        \delta\,
        \Phi^c\left(
        \|\theta\|-
        \frac{\log(1/\delta)}{2\|\theta\|}
        \right),
\]
and hence
\[
\inf_{\mathfrak h\in\mathcal H^{\mathrm{on}}}
\ClassifRiskOn(\theta,
\delta,\mathfrak h)
\gtrsim
\delta\,
\Phi^c\left(
\|\theta\|-
\frac{\log(1/\delta)}{2\|\theta\|}
\right).
\]

It remains only to pass from classification to clustering.  By the equivalence
between the online Bayes risks of classification and clustering recalled in
Section~\ref{sec:equiv},
\[
\inf_{\mathfrak g\in\mathcal G^{\mathrm{on}}}
\ClusterRiskOn(\theta,\delta,\mathfrak g)
\ge
(1-\widetilde\alpha_n)
\inf_{\mathfrak h\in\mathcal H^{\mathrm{on}}}
\ClassifRiskOn(\theta,\delta,\mathfrak h),
\]
where \(\widetilde\alpha_n=O((\sqrt n\,\delta^5)^{-1})\).  Under the assumption
\(\delta\ge c n^{-1/10}\), with \(c\) sufficiently large, we have
\(\widetilde\alpha_n\le1/2\).  Combining this comparison with the three
classification lower bounds above proves the theorem.

        \subsection{Proof of Theorem~\ref{Bayes_risk_ub_online}}
First, note that
\[
\ClusterRiskOn(\theta,\delta,\pi_n\circ\widehat\eta)
\le
\frac1n\sum_{i=1}^n
\mathbb P(\widehat\eta_i\ne\eta_i).
\]

\textbf{Intermediate signal regime.}
Assume $2\delta<\|\theta\|^2
        <
        \frac{\log(1/\delta)}{2}.$
Take
\[
        k
        =
        \left\lceil
        \frac{2}{\|\theta\|^2}
        \left(
        \log\left(\frac{\|\theta\|^2}{2\delta}\right)+1
        \right)
        \right\rceil.
\]
If \(k\ge n\), then the definition of
\(k\), together with \(\delta\ge n^{-1}\), implies that
\[
        \frac{\delta}{\|\theta\|^2}
        \left(
        \log\left(\frac{\|\theta\|^2}{2\delta}\right)+1
        \right)
        \gtrsim1,
\]
so the claimed upper bound follows from the trivial bound by one.  We may
therefore assume that \(k<n\). For \(i\le k\), the pointwise rule gives
\[
        \mathbb P(\widehat\eta_i\ne\eta_i)
        =
        \Phi^c(\|\theta\|)
        \le
        \exp\left(-\frac{\|\theta\|^2}{2}\right).
\]
For \(i>k\), let \(A_i\) be the event that the hidden chain has at least one
transition in the window \(\{i-k+1,\ldots,i\}\).  Then
\[
        \mathbb P(A_i)\le k\delta.
\]
On \(A_i^c\), all labels in the window are equal to \(\eta_i\), and therefore
\begin{align*}
\mathbb P(\widehat\eta_i\ne\eta_i)
&\le
k\delta
+
\mathbb P\left(
\left\langle
\frac1k\sum_{j=i-k+1}^i\xi_j,
\frac{\theta}{\|\theta\|}
\right\rangle
+\|\theta\|
\le0
\right) \le
k\delta+
\exp\left(-\frac{k\|\theta\|^2}{2}\right).
\end{align*}
Consequently,
\begin{align*}
\frac1n\sum_{i=1}^n
\mathbb P(\widehat\eta_i\ne\eta_i)
&\le
\frac{k}{n}\exp\left(-\frac{\|\theta\|^2}{2}\right)
+k\delta
+
\exp\left(-\frac{k\|\theta\|^2}{2}\right).
\end{align*}
Since \(n^{-1}\le\delta\), the first term is at most \(k\delta\).  Moreover,
by the definition of \(k\),
\[
        k\delta
        \lesssim
        \frac{\delta}{\|\theta\|^2}
        \left(
        \log\left(\frac{\|\theta\|^2}{2\delta}\right)+1
        \right)
\]
and
\[
        \exp\left(-\frac{k\|\theta\|^2}{2}\right)
        \le
        e^{-1}\frac{2\delta}{\|\theta\|^2}
        \lesssim
        \frac{\delta}{\|\theta\|^2}
        \left(
        \log\frac{\|\theta\|^2}{2\delta}+1
        \right).
\]
It follows that
\[
\ClusterRiskOn(\theta,\delta,\pi_n\circ\widehat\eta)
\lesssim
\frac{\delta}{\|\theta\|^2}
\left(
\log\left(\frac{\|\theta\|^2}{2\delta}\right)+1
\right).
\]

\textbf{Strong signal regime.}
Assume $\|\theta\|^2
        \ge
        \frac{\log(1/\delta)}{2}.$
Put
\[
        b=\frac12\log\left(\frac{1-\delta}{\delta}\right),
        \qquad
        q=
        \Phi^c\left(
        \|\theta\|-
        \frac{\log(1/\delta)}{2\|\theta\|}
        \right),
        \qquad
        e_i=\mathbb P(\widehat\eta_i\ne\eta_i).
\]
The initialization \(\widehat\eta_1=\operatorname{sign}\left(\langle X_1,\theta\rangle\right)\)
satisfies $e_1=\Phi^c(\|\theta\|)\le q$.
For \(i\ge2\), conditionally on the past and on the hidden transition at time
\(i\), the noise \(\xi_i\) is independent and Gaussian.  If
\(\widehat\eta_{i-1}=\eta_i\), the conditional error probability is
$\Phi^c\left(\|\theta\|+\frac{b}{\|\theta\|}\right)$, whereas if \(\widehat\eta_{i-1}\ne\eta_i\), it is \(\Phi^c\left(\|\theta\|-\frac{b}{\|\theta\|}\right)\).
Since
\[
        \mathbb P(\widehat\eta_{i-1}\ne\eta_i)
        \le e_{i-1}+\delta,
\]
we obtain
\[
        e_i
        \le
        \Phi^c\left(\|\theta\|+\frac{b}{\|\theta\|}\right)
        +(e_{i-1}+\delta)
        \Phi^c\left(\|\theta\|-\frac{b}{\|\theta\|}\right).
\]
The Gaussian-tail comparisons established in the proof of the strong-signal
lag bound give
\[
        \Phi^c\left(\|\theta\|+\frac{b}{\|\theta\|}\right)
        \lesssim \delta q
\]
and, since \(\log((1-\delta)/\delta)\le\log(1/\delta)\),
\[
        \Phi^c\left(\|\theta\|-\frac{b}{\|\theta\|}\right)
        \le q.
\]
Thus, for a numerical constant \(C>0\),
\[
        e_i\le C\delta q+q e_{i-1}.
\]
In the strong-signal regime the argument of \(q\) is nonnegative, so
\(q\le1/2\).  Iterating the preceding inequality yields
\[
        e_i
        \le
        2C\delta q+q^{i-1}e_1.
\]
Averaging over \(i\) gives
\begin{align*}
\frac1n\sum_{i=1}^n e_i
&\le
2C\delta q
+
\frac{e_1}{n}\sum_{i=1}^n q^{i-1} \le
2C\delta q+
\frac{2e_1}{n} \le
2C\delta q+2\delta q
\lesssim
\delta q,
\end{align*}
where the last inequality uses \(n^{-1}\le\delta\) and \(e_1\le q\).
Therefore
\[
\ClusterRiskOn(\theta,\delta,\pi_n\circ\widehat\eta)
\lesssim
\delta\,\Phi^c\left(
\|\theta\|-
\frac{\log(1/\delta)}{2\|\theta\|}
\right).
\]
Combining the three regimes completes the proof.

\subsection{Proof of Theorem~\ref{Bayes_risk_lb_offline}}

By Corollary \ref{prop:equivalence-from-two-comparisons}, it is enough
to prove the corresponding lower bound for the oracle classification risk.

\textbf{Strongsignal regime.}
For \(1\le i\le n\), write
\(\eta_{-i}=(\eta_j)_{j\ne i}\).
Revealing \(\eta_{-i}\) can only decrease the Bayes risk. Moreover,
conditionally on \(\eta_{-i}\), the observations \(X_j\), \(j\ne i\), carry
no additional information about \(\eta_i\). Hence
\begin{equation}
\label{eq1:lb:offline}
    \inf_{h\in\mathcal H_n}
\ClassifRisk(\theta,\delta,h)
\ge
\frac1n\sum_{i=1}^n B_i,
\end{equation}

where
\(
B_i
=
\mathbb E\left[
\min\left\{
\mathbb P(\eta_i=1\mid X_i,\eta_{-i}),
\mathbb P(\eta_i=-1\mid X_i,\eta_{-i})
\right\}
\right].
\)
Let \(f_\mu\) denote the \(\mathcal N(\mu,I_d)\) density. We first compute
its overlap for the two possible signals. Since
\[
        \frac{f_\theta(x)}{f_{-\theta}(x)}
        =\exp\bigl(2\langle x,\theta\rangle\bigr),
\]
the two densities agree on the hyperplane
\(\langle x,\theta\rangle=0\), and therefore
\begin{align*}
\int_{\mathbb R^d} f_\theta(x)\wedge f_{-\theta}(x)\,dx
&=
\mathbb P_\theta\bigl(\langle X,\theta\rangle\le0\bigr)
+
\mathbb P_{-\theta}\bigl(\langle X,\theta\rangle\ge0\bigr)
=2\,\Phi^c(\|\theta\|).
\end{align*}

Consider first an interior index \(2\le i\le n-1\). Conditionally on
\((\eta_{i-1},\eta_{i+1})=(a,b)\), the two unnormalized posterior densities
corresponding to \(\eta_i=1\) and \(\eta_i=-1\) are
\(
        Q_{a,1}Q_{1,b}f_\theta
        \text{ and }
        Q_{a,-1}Q_{-1,b}f_{-\theta},
\)
respectively. Averaging the conditional Bayes error over \((a,b)\), and
retaining only the two configurations \((a,b)=(-1,1)\) and \((a,b)=(1,-1)\),
gives

\begin{align*}
B_i
&\ge\int_{\mathbb R^d}
\delta(1-\delta)f_\theta(x)
\wedge
\delta(1-\delta)f_{-\theta}(x)\,dx
\ge
\delta\,\Phi^c(\|\theta\|).
\end{align*}
For the boundary indices \(i=1,n\), conditioning on the unique neighboring
state gives similarly
\begin{equation}
\label{eq2:lb:offline}
    B_i
    \ge
    \delta
    \int_{\mathbb R^d}f_\theta(x)\wedge f_{-\theta}(x)\,dx
    =
    2\delta\,\Phi^c(\|\theta\|).
\end{equation}
Thus \(B_i\ge\delta\Phi^c(\|\theta\|)\) for every \(i\), and
averaging over \(i\) proves the desired lower-bound.

\textbf{Intermediate signal regime.}
Assume
\(
        2\delta\le\|\theta\|^2\le1
\)
and set
\[
        k=\left\lfloor\frac{1}{\|\theta\|^2}\right\rfloor.
\]
Then
\[
        k\asymp\frac{1}{\|\theta\|^2},
        \qquad
        k\delta\le\frac{\delta}{\|\theta\|^2}\le\frac12.
\]
Moreover, the assumption
\(\delta\ge\sqrt{\log(n)/n}\) implies \(k=o(n)\), so, after changing
constants, we may assume that \(4k<n\).

For \(2k+1\le i\le n-2k\), reveal the two boundary states
\(\eta_{i-2k}\) and \(\eta_{i+2k}\), and define
\[
A_i
=
\inf_{\widehat\eta_i}
\mathbb P\left(
        \widehat\eta_i(X_{i-2k:i+2k},\eta_{i-2k},\eta_{i+2k})
        \ne\eta_i
\right).
\]
Giving this additional information can only decrease the Bayes risk. By
symmetry, it is enough to retain paths satisfying
\(\eta_{i-2k}=1\) and \(\eta_{i+2k}=-1\).

For \(r=1,\ldots,k\), let
\[
        j_r^-=i-2k+r-1,
        \qquad
        j_r^+=i+k+r-1,
\]
and let \(P_r^-\) and \(P_r^+\) denote the laws of
\(X_{i-2k:i+2k}\) under the two single-transition paths
\[
\begin{aligned}
P_r^-:
&\quad
\eta_{i-2k:j_r^-}=1,
\qquad
\eta_{j_r^-+1:i+2k}=-1,
\\
P_r^+:
&\quad
\eta_{i-2k:j_r^+}=1,
\qquad
\eta_{j_r^++1:i+2k}=-1.
\end{aligned}
\]
Under \(P_r^-\), one has \(\eta_i=-1\), whereas under \(P_r^+\), one has
\(\eta_i=1\). Each of these paths has probability
\[
        \frac12\,\delta(1-\delta)^{4k-1}.
\]
Consequently, the binary-testing bound gives
\[
        A_i
        \ge
        \frac{\delta(1-\delta)^{4k-1}}{2}
        \sum_{r=1}^k
        \bigl(1-\|P_r^--P_r^+\|_{\mathrm{TV}}\bigr).
\]

We now compute the Kullback--Leibler divergence explicitly. For
\(s\in\{i-2k,\ldots,i+2k\}\), let \(\mu_{r,s}^-\) and \(\mu_{r,s}^+\)
denote the means of \(X_s\) under \(P_r^-\) and \(P_r^+\), respectively.
The two paths agree outside the interval
\[
        \{j_r^-+1,\ldots,j_r^+\}.
\]
On this interval, the path defining \(P_r^-\) has label \(-1\), whereas the
path defining \(P_r^+\) has label \(1\). Thus
\[
        \mu_{r,s}^-=-\theta,
        \qquad
        \mu_{r,s}^+=\theta,
        \qquad
        \mu_{r,s}^--\mu_{r,s}^+=-2\theta
\]
for \(j_r^-+1\le s\le j_r^+\), while
\(\mu_{r,s}^-=\mu_{r,s}^+\) at every other coordinate. Moreover,
\[
        j_r^+-j_r^-
        =
        (i+k+r-1)-(i-2k+r-1)
        =3k,
\]
so there are exactly \(3k\) coordinates at which the two mean sequences
differ.

Conditionally on either path, the observations are independent with covariance
\(I_d\). Hence
\[
        P_r^-
        =
        \bigotimes_{s=i-2k}^{i+2k}
        \mathcal N(\mu_{r,s}^-,I_d),
        \qquad
        P_r^+
        =
        \bigotimes_{s=i-2k}^{i+2k}
        \mathcal N(\mu_{r,s}^+,I_d).
\]
The additivity of the Kullback--Leibler divergence for product measures and the
identity
\[
        \mathrm{KL}\bigl(\mathcal N(\mu,I_d)\,\|\,
        \mathcal N(\nu,I_d)\bigr)
        =
        \frac12\|\mu-\nu\|^2
\]
therefore give
\begin{align*}
        \mathrm{KL}(P_r^-\|P_r^+) =
        \frac12
        \sum_{s=i-2k}^{i+2k}
        \|\mu_{r,s}^--\mu_{r,s}^+\|^2
        =
        \frac12
        \sum_{s=j_r^-+1}^{j_r^+}
        \|-2\theta\|^2
        =
        \frac12(3k)(4\|\theta\|^2)
        =
        6k\|\theta\|^2
        \le6,
\end{align*}
where the last inequality follows from
\(k=\lfloor1/\|\theta\|^2\rfloor\), and hence
\(k\|\theta\|^2\le1\). By the Bretagnolle--Huber inequality,
\[
        1-\|P_r^--P_r^+\|_{\mathrm{TV}}
        \ge
        \frac12
        \exp\bigl(-\mathrm{KL}(P_r^-\|P_r^+)\bigr)
        \ge
        \frac12e^{-6}.
\]
Thus the overlap of each pair \((P_r^-,P_r^+)\) is bounded below by a positive
numerical constant, uniformly in \(r\). Since \(k\delta\le1/2\), the factor
\((1-\delta)^{4k-1}\) is also bounded below by a positive numerical constant.
Therefore, uniformly over \(2k+1\le i\le n-2k\),
\[
        A_i\gtrsim k\delta
        \asymp
        \frac{\delta}{\|\theta\|^2}.
\]
Averaging over the interior indices gives
\begin{align*}
\inf_{h\in\mathcal H_n}
\ClassifRisk(\theta,\delta,h)
&\ge
\frac1n\sum_{i=2k+1}^{n-2k}A_i
\gtrsim
\left(1-\frac{4k}{n}\right)
\frac{\delta}{\|\theta\|^2}
\gtrsim
\frac{\delta}{\|\theta\|^2}.
\end{align*}

\textbf{Weak signal regime.}
Assume
\(
        \|\theta\|^2\le2\delta.
\)
The oracle Bayes risk is nonincreasing in the signal magnitude. More precisely,
if \(\theta=t v\) and \(\theta_0=t_0v\), with \(0<t\le t_0\), then from an
observation \(X=\eta\theta_0+\xi\) one can generate
\[
        \frac{t}{t_0}X
        +
        \sqrt{1-\frac{t^2}{t_0^2}}\,Z,
        \qquad Z\sim\mathcal N(0,I_d),
\]
which has distribution \(\mathcal N(\eta\theta,I_d)\). Thus the experiment at
signal \(\theta\) is a garbling of the experiment at signal \(\theta_0\), and
its Bayes risk is no smaller. Applying the intermediate-regime bound at signal
strength \(\|\theta_0\|^2=2\delta\) therefore yields
\[
\inf_{h\in\mathcal H_n}
\ClassifRisk(\theta,\delta,h)
\ge
\inf_{h\in\mathcal H_n}
\ClassifRisk(\theta_0,\delta,h)
\gtrsim
\frac{\delta}{\|\theta_0\|^2}
\asymp1.
\]
The case \(\theta=0\) follows by the same monotonicity argument. 

\subsection{Proof of Theorem \ref{prop:risk_clust_procedure_ub}}

Throughout the proof, we assume that $\delta\geq 1/n$ and write
\[
    s:=\|\theta\|^2,
    \qquad
    L_\delta:=\log\left(\frac{1-\delta}{\delta}\right),
    \qquad
    q_s:=\Phi^c(\sqrt{s}).
\]
We work under the normalization $\sigma=1$.

Recall that the pilot estimators $\widetilde\eta_{a-1}$ and
$\widetilde\eta_{a+1}$ are constructed from the $k$ observations
immediately to the left and to the right of $a$, respectively.
Whenever the labels are constant on the corresponding block,
\begin{equation}
\label{eq:pilot-error}
    \PP\bigl(\widetilde\eta_{a\pm1}\neq\eta_{a\pm1}\bigr)
    =
    \Phi^c(\sqrt{ks})
    \leq e^{-ks/2}.
\end{equation}
Moreover, conditionally on the labels, the two pilot errors and the
noise appearing in $X_a$ are independent.

Recall also that
\[
\widehat\eta_a
=
\begin{cases}
\displaystyle
\sign\left(
    \langle X_a,\theta\rangle
    +\widetilde\eta_{a-1}L_\delta
\right),
&
\widetilde\eta_{a-1}=\widetilde\eta_{a+1},
\\[2mm]
\sign\bigl(\langle X_a,\theta\rangle\bigr),
&
\widetilde\eta_{a-1}\neq\widetilde\eta_{a+1}.
\end{cases}
\]

\textbf{Intermediate signal regime.}
Assume first that
\[
    2\delta < s < 2\log(1/\delta),
\]
and set
\[
    k
    :=
    1\vee
    \left\lceil
        \frac{2}{s}\log\left(\frac{s}{2\delta}\right)
    \right\rceil.
\]
If $2k\ge n$, all coordinates are treated by the pointwise rule.
For $n\ge4$, the inequalities
\[
\frac n2\le k
\le1+\frac2s\log\left(\frac{s}{2\delta}\right)
\]
and $n\delta\ge1$ imply
\[
\frac{\delta}{s}
\left(1+\log\frac{s}{2\delta}\right)
\ge\frac18.
\]
For $s\le1$, the claimed bound then follows from the trivial
risk bound. For $s>1$, the claimed bound is at least a constant
multiple of the pointwise risk $\Phi^c(\sqrt s)$.
The cases $n=2,3$ follow directly from $\delta\ge1/n$.
We may therefore assume $2k<n$ in the interior analysis.
Fix $a\in\llbracket k+1,n-k\rrbracket$ and define
\[
    N_a
    :=
    \sum_{j=a-k}^{a+k-1}
    \mathbf 1\{\eta_{j+1}\neq\eta_j\},
\]
the number of transitions in the $2k$ edges surrounding $a$.
Since the transition indicators are i.i.d.\ Bernoulli$(\delta)$,
\begin{equation*}
    \PP(N_a\geq2)
    \leq
    \binom{2k}{2}\delta^2
    \leq2k^2\delta^2.
\end{equation*}

Set
\[
    q_k:=\Phi^c(\sqrt{ks}),
    \qquad
    q_+:=
    \Phi^c\left(
        \sqrt{s}+\frac{L_\delta}{\sqrt{s}}
    \right).
\]

On the event $\{N_a=0\}$, all the labels in
$\eta_{a-k:a+k}$ coincide with $\eta_a$. If both pilot estimators are
correct, the procedure uses the correct Markov correction and its error
probability is $q_+$. If exactly one pilot is wrong, the two pilots
disagree and the procedure reduces to
$\sign(\langle X_a,\theta\rangle)$; by conditional independence, the
corresponding contribution is at most $2q_kq_s$. Finally, the event
that both pilots are wrong has probability at most $q_k^2$. Therefore
\begin{equation*}
    \PP(\widehat\eta_a\neq\eta_a,\ N_a=0)
    \leq
    q_+ +2q_kq_s+q_k^2.
\end{equation*}

Consider now $\{N_a=1\}$. In this case at least one of the two
length-$k$ blocks adjacent to $a$ is constant and has all its labels
equal to $\eta_a$. We call the corresponding pilot the clean pilot.
By \eqref{eq:pilot-error}, its error probability is at most $q_k$.

If the clean pilot is correct, then, irrespective of the other pilot,
the final decision is either the local classifier
$\sign(\langle X_a,\theta\rangle)$ or the classifier using the correct
Markov correction. In either case its error probability is at most
$q_s$. Hence
\begin{equation*}
\begin{split}
    \PP(\widehat\eta_a\neq\eta_a,\ N_a=1)
    \leq
    \PP(N_a=1)(q_s+q_k) \leq
    2k\delta(q_s+q_k).
\end{split}
\end{equation*}
Combining the inequalities above, we obtain
\begin{equation}
\label{eq:generic-intermediate-bound}
\begin{split}
    \PP(\widehat\eta_a\neq\eta_a)
    &\leq
    q_+
    +2q_kq_s
    +q_k^2
    +2k\delta(q_s+q_k)
    +2k^2\delta^2
\\
    &\leq
    q_+
    +2e^{-(k+1)s/2}
    +e^{-ks}
    +4k\delta e^{-s/2}
    +2k^2\delta^2.
\end{split}
\end{equation}

We now simplify the right-hand side. By the definition of $k$,
\begin{equation}
\label{eq:pilot-balance}
    e^{-ks/2}\leq \frac{2\delta}{s},
\end{equation}
and
\begin{equation*}
    k\delta
    \lesssim
    \frac{\delta}{s}
    \left(
        1+\log\frac{s}{2\delta}
    \right).
\end{equation*}

We shall also use
\begin{equation}
\label{eq:kdelta-gaussian}
    k\delta\lesssim e^{-s/2}.
\end{equation}
Indeed, if $s\leq2$, writing $x=s/(2\delta)\geq1$ gives
\(
    k\delta
    \lesssim
    \delta+\frac{2\delta}{s}\log\frac{s}{2\delta}
    =
    \delta+\frac{\log x}{x}
    \lesssim1
    \lesssim e^{-s/2}.
\)
If $s>2$, let
\(
    u:=\log(1/\delta)-\frac{s}{2}\geq0.
\)
Then
\[
\begin{split}
    k\delta e^{s/2}
    &\lesssim
    \delta e^{s/2}
    +
    \frac{2\delta e^{s/2}}{s}
    \log\frac{s}{2\delta} =
    e^{-u}
    +
    \frac{2e^{-u}}{s}
    \left(
        \frac{s}{2}+u+\log\frac{s}{2}
    \right)
    \lesssim1.
\end{split}
\]
This proves \eqref{eq:kdelta-gaussian}. In particular,
\begin{equation*}
    (k\delta)^2
    \lesssim
    k\delta e^{-s/2}.
\end{equation*}

Moreover, using \eqref{eq:pilot-balance} and $k\geq1$,
\(
    e^{-(k+1)s/2}
    \leq
    \frac{2\delta}{s}e^{-s/2},
\)
and
\(
    e^{-ks}
    =
    e^{-ks/2}e^{-ks/2}
    \leq
    \frac{2\delta}{s}e^{-s/2}.
\)

Finally,
\begin{align*}
    q_+
    &\leq
    \exp\left\{
        -\frac12
        \left(
            \sqrt{s}+\frac{L_\delta}{\sqrt{s}}
        \right)^2
    \right\} =
    e^{-s/2}e^{-L_\delta}e^{-L_\delta^2/(2s)} \lesssim
    \delta e^{-s/2}.
\end{align*}
Combining the previous bounds with
\eqref{eq:generic-intermediate-bound}, we conclude that
\begin{equation}
\label{eq:intermediate-exponential}
    \PP(\widehat\eta_a\neq\eta_a)
    \lesssim
    \frac{\delta}{s}
    \left(
        1+\log\frac{s}{2\delta}
    \right)e^{-s/2}.
\end{equation}

If $2\delta\leq s\leq1$, we simply use $e^{-s/2}\leq1$ and obtain
\begin{equation*}
    \PP(\widehat\eta_a\neq\eta_a)
    \lesssim
    \frac{\delta}{s}
    \left(
        1+\log\frac{s}{\delta}
    \right).
\end{equation*}

If instead
\(
    1<s\leq2\log(1/\delta),
\)
Mills' inequalities imply
\(
    e^{-s/2}
    \lesssim
    \sqrt{s}\,\Phi^c(\sqrt{s}).
\)
Hence \eqref{eq:intermediate-exponential} gives
\[
\begin{split}
    \PP(\widehat\eta_a\neq\eta_a)
    &\lesssim
    \frac{\delta}{\sqrt{s}}
    \left(
        1+\log\frac{s}{\delta}
    \right)
    \Phi^c(\sqrt{s})\leq
    C\delta
    \left(
        1+\log\frac{s}{\delta}
    \right)
    \Phi^c(\sqrt{s}),
\end{split}
\]
where in the last inequality we used $s>1$. Thus
\begin{equation*}
    \PP(\widehat\eta_a\neq\eta_a)
    \lesssim
    \delta
    \left(
        1+\log\frac{s}{\delta}
    \right)
    \Phi^c(\sqrt{s}).
\end{equation*}

\textbf{Strong signal regime.}
Assume now that
\(
    s \geq 2\log(1/\delta).
\)
In this regime we take $k=1$. Thus
\[
    \widetilde\eta_{a\pm1}
    =
    \sign\langle X_{a\pm1},\theta\rangle
\]
and
\begin{equation}
\label{eq:strong-pilot-error}
    \PP(\widetilde\eta_{a\pm1}\neq\eta_{a\pm1})
    =
    q_s
    \leq e^{-s/2}
    \leq\delta.
\end{equation}
Thus a pilot error is no more likely than an actual transition of the
hidden Markov chain.

Let
\[
    U:=\mathbf 1\{\eta_{a-1}\neq\eta_a\},
    \qquad
    V:=\mathbf 1\{\eta_a\neq\eta_{a+1}\}.
\]
Then $U$ and $V$ are independent Bernoulli$(\delta)$.

If $U=V=0$, both true neighbors coincide with $\eta_a$. When both
pilots are correct, the error probability is
\[
    q_+
    =
    \Phi^c\left(
        \sqrt{s}+\frac{L_\delta}{\sqrt{s}}
    \right).
\]
If exactly one pilot is wrong, the two pilots disagree and the local
classifier is used, while the event that both pilots are wrong has
probability $q_s^2$. Consequently,
\begin{equation}
\label{eq:strong-no-change}
    \PP(\widehat\eta_a\neq\eta_a\mid U=V=0)
    \leq q_+ +3q_s^2.
\end{equation}

If $U\neq V$, the two true neighbors are different. When both pilots
are correct, the procedure uses the local classifier and has error
$q_s$. The probability that at least one pilot is wrong is at most
$2q_s$. Hence
\begin{equation}
\label{eq:strong-one-change}
    \PP(\widehat\eta_a\neq\eta_a\mid U\neq V)
    \leq3q_s.
\end{equation}

Finally, if $U=V=1$, both true neighbors are equal to $-\eta_a$. When
both pilots are correct, the error probability is
\[
    q_-:=
    \Phi^c\left(
        \sqrt{s}-\frac{L_\delta}{\sqrt{s}}
    \right),
\]
whereas the probability that at least one pilot is wrong is at most
$2q_s$. Therefore
\begin{equation}
\label{eq:strong-two-changes}
    \PP(\widehat\eta_a\neq\eta_a\mid U=V=1)
    \leq q_-+2q_s.
\end{equation}

Combining \eqref{eq:strong-no-change}--\eqref{eq:strong-two-changes},
\begin{equation}
\label{eq:strong-decomposition}
    \PP(\widehat\eta_a\neq\eta_a)
    \leq
    q_+
    +3q_s^2
    +6\delta q_s
    +\delta^2q_-
    +2\delta^2q_s.
\end{equation}

Set
\[
    x:=\sqrt{s},
    \qquad
    t:=\frac{L_\delta}{\sqrt{s}}.
\]
Since
\(
    L_\delta<\log(1/\delta)<s/2,
\)
we have $0\le t <x/2$. First,
\[
\begin{split}
    \Phi^c(x+t)
    &=
    \int_x^\infty\phi(u+t)\,du\leq
    e^{-xt-t^2/2}\Phi^c(x)\leq
    e^{-L_\delta}\Phi^c(x)
    \lesssim
    \delta\Phi^c(x).
\end{split}
\]
Thus
\begin{equation*}
    q_+\lesssim\delta q_s.
\end{equation*}

On the other hand, Mills' inequalities and $t<x/2$ yield
\[
\begin{split}
    \frac{\Phi^c(x-t)}{\Phi^c(x)}
    &\leq
    C\frac{x}{x-t}
    \exp\left(xt-\frac{t^2}{2}\right)
\leq
    Ce^{L_\delta}
    \lesssim
    \frac1\delta.
\end{split}
\]
Hence
\begin{equation*}
    \delta^2q_-
    \lesssim
    \delta q_s.
\end{equation*}
Finally, by \eqref{eq:strong-pilot-error},
\[
    q_s^2\leq\delta q_s.
\]
Plugging these estimates into \eqref{eq:strong-decomposition}, we get
\begin{equation*}
    \PP(\widehat\eta_a\neq\eta_a)
    \lesssim
    \delta\Phi^c(\sqrt{s}).
\end{equation*}

\textbf{Boundary points.}
For the first and last $k$ coordinates, we use the local classifier
$\sign\langle X_a,\theta\rangle$, whose error probability is $q_s$.

Since $\delta\geq1/n$, we have $k/n\leq k\delta$. Therefore the boundary
contribution is absorbed by the preceding bounds in the intermediate
regime. In the strong regime, $k=1$ and
\[
    \frac{2}{n}\Phi^c(\sqrt{s})
    \leq
    2\delta\Phi^c(\sqrt{s}).
\]

We conclude that
\[
\ClusterRisk(\theta,\delta,\pi_n\circ\widehat\eta)
\lesssim
\begin{cases}
\displaystyle
\frac{\delta}{s}
\left(
    1+\log\frac{s}{\delta}
\right),
&
2\delta\leq s\leq1,
\\[3mm]
\displaystyle
\delta
\left(
    1+\log\frac{s}{\delta}
\right)
\Phi^c(\sqrt{s}),
&
1<s\leq2\log(1/\delta),
\\[3mm]
\displaystyle
\delta\Phi^c(\sqrt{s}),
&
s>2\log(1/\delta).
\end{cases}
\]
This concludes the proof.

\subsection{Proof of Theorem \ref{thm:risk_clust_new_procedure_ub}}
Write
\[
    t:=\|\theta\|,
    \qquad
    s:=t^2.
\]
The regime $s\leq2\delta$ is immediate from the trivial bound by one,
so throughout the proof we assume
\[
    2\delta<s<2\log(1/\delta).
\]

For an interior index
\(
    i\in\llbracket k+1,n-k\rrbracket,
\)
project the observations onto the known signal direction and write
\[
    Y_j
    :=
    \frac{\langle\theta,X_j\rangle}{\|\theta\|}
    =
    t\eta_j+\xi_j,
    \qquad
    \xi_j\stackrel{\mathrm{i.i.d.}}{\sim}\mathcal{N}(0,1).
\]
Thus the local estimator maximizes
\(
    L_i(u)
    :=
    \sum_{j=i-k}^{i+k}u_jY_j
\)
over
$\mathcal S_k^+\cup\mathcal S_k^-$.

For two paths $u,v\in\{-1,1\}^{2k+1}$, let
\(
    D(u,v)
    :=
    \#\{j:u_j\neq v_j\}.
\)
Conditionally on the true path $v$,
\[
    L_i(u)-L_i(v)
    \sim
    \mathcal{N}\!\left(-2tD(u,v),\,4D(u,v)\right).
\]
Consequently,
\begin{equation}
\label{eq:pairwise-gaussian-comparison}
    \PP_v\bigl(L_i(u)\geq L_i(v)\bigr)
    =
    \Phi^c\!\left(
        t\sqrt{D(u,v)}
    \right).
\end{equation}

We first record a simple maximal estimate that will only be needed in
the weak-signal regime.

\medskip
\noindent
{\bf Gaussian random-walk bound.}
Let $(G_j)_{j\geq1}$ be i.i.d.\ $\mathcal{N}(0,1)$ and
$S_m=\sum_{j=1}^mG_j$. There exist universal constants $c,C>0$ such
that, for every $t>0$ and $q\geq1$,
\begin{equation}
\label{eq:drifted-walk-1}
    \PP\left(
        Z_q-tq+
        \sup_{m\geq0}(S_m-tm)
        \geq0
    \right)
    \leq
    Ce^{-ct^2q},
\end{equation}
where $Z_q\sim \mathcal{N}(0,q)$ is independent of $(S_m)_m$.

Moreover, if $Z_m\sim N(0,m)$ is independent of $(S_q)_q$, then
\begin{equation}
\label{eq:drifted-walk-2}
    \PP\left(
        \sup_{q\geq1}(S_q-tq)
        \geq tm-Z_m
    \right)
    \leq
    Ce^{-ct^2m}.
\end{equation}

Indeed, embed the Gaussian random walk into a standard Brownian
motion $B$. We use the classical identity
\[
    \PP\left(
        \sup_{u\geq0}(B_u-tu)\geq x
    \right)
    =
    e^{-2tx},
    \qquad x\geq0.
\]
For example, on the event $Z_q\leq tq/2$, the event in
\eqref{eq:drifted-walk-1} requires
\[
    \sup_{u\geq0}(B_u-tu)\geq tq/2.
\]
Hence
\[
    \PP\left(
        Z_q-tq+\sup_{m\geq0}(S_m-tm)\geq0
    \right)
    \leq
    \Phi^c\left(\frac{t\sqrt q}{2}\right)
    +
    e^{-t^2q},
\]
which proves \eqref{eq:drifted-walk-1}. The proof of
\eqref{eq:drifted-walk-2} is identical after splitting according to
$Z_m\leq tm/2$.

\medskip
We now return to the local segmentation problem. Let
\[
    B_i
    :=
    \left\{
        \text{at least two transitions occur in }
        \eta_{i-k:i+k}
    \right\}.
\]
Since the transition indicators are i.i.d.\
Bernoulli$(\delta)$,
\begin{equation}
\label{eq:two-transitions}
    \PP(B_i)
    \leq
    2k^2\delta^2.
\end{equation}

On $B_i^c$, the true local path belongs to
$\mathcal S_k^+\cup\mathcal S_k^-$. Hence, if the center is
misclassified, there exists a candidate path with the opposite center
whose score is not smaller than the score of the true path.

We distinguish three configurations.

\medskip
\noindent
{\bf 1. The true path is constant.}

Any path with the opposite center differs from the true path on a
whole half-window containing the center, hence on at least $k+1$
coordinates. As the location of the candidate change-point varies,
the additional disagreement coordinates form a nested interval.
Therefore \eqref{eq:drifted-walk-1} gives
\begin{equation}
\label{eq:constant-path-bound}
    \PP\left(
        \widehat\eta_i\neq\eta_i,\,
        \eta_{i-k:i+k}\text{ constant}
    \right)
    \leq
    Ce^{-csk}.
\end{equation}

\medskip
\noindent
{\bf 2. The true and candidate change-points lie on the same side of
$i$.}

For every fixed true change-point, a wrong-center candidate again
differs from the true path on the whole half-window lying on the
opposite side of $i$, which contains $k+1$ coordinates. Varying the
candidate change-point only adds a nested interval of disagreements.
Thus the same maximal estimate yields a conditional error probability
at most $Ce^{-csk}$.

Constant candidates with the wrong center label also disagree with
the true path on at least \(k+1\) coordinates and satisfy the same
bound; for non-constant candidates, the nested-family argument is
applied separately according to whether the candidate change-point
lies before or after the true change-point.

There are at most $2k$ possible locations of the unique true
transition, each having probability at most $\delta$. Consequently,
\begin{equation}
\label{eq:same-side-bound}
    \PP\left(
        \widehat\eta_i\neq\eta_i,\,
        B_i^c,\,
        \text{change-points on the same side}
    \right)
    \leq
    Ck\delta e^{-csk}.
\end{equation}

\medskip
\noindent
{\bf 3. The true and candidate change-points lie on opposite sides of
$i$.}

Assume for definiteness that the true change occurs to the right of
\(i\) and the candidate change to the left; the other configuration
is symmetric.
By global sign symmetry, we may also assume that the true center
label is \(\eta_i=-1\).
Let $m\geq0$ denote the distance from $i$ to the true
change-point and let $q\geq1$ denote the distance from $i$ to the
candidate change-point. The two paths disagree on an interval of
length $m+q$. Hence, for fixed $(m,q)$,
\begin{equation}
\label{eq:opposite-pair}
    \PP\left(
        L_i(u)\geq L_i(\eta)
        \mid m,q
    \right)
    =
    \Phi^c\!\left(
        t\sqrt{m+q}
    \right).
\end{equation}

We treat separately the cases $s\leq1$ and $s>1$.

Suppose first that $s\leq1$. If
\[
    m\leq s^{-1},
\]
we simply bound the conditional error probability by one. Since each
possible location of the true transition has probability at most
$\delta$, the contribution of these change-points is bounded by
\begin{equation}
\label{eq:near-transition}
    C\delta
    \left(1+\frac1s\right)
    \leq
    C\frac{\delta}{s}.
\end{equation}

Assume now that $m>s^{-1}$. Writing
\(
    S_q=\sum_{j=i-q+1}^{i}\xi_j, 
    Z_m=\sum_{j=i+1}^{i+m}\xi_j,
\)
the existence of a wrong-center candidate on the opposite side
implies
\[
    \sup_{q\geq1}
    \left(S_q-tq\right)
    \geq
    tm-Z_m.
\]
By \eqref{eq:drifted-walk-2},
\[
    \PP\left(
        \text{some opposite-side candidate wins}\mid m
    \right)
    \leq
    Ce^{-csm}.
\]
It follows that
\begin{equation}
\label{eq:far-transition}
    \delta
    \sum_{m>s^{-1}}
    Ce^{-csm}
    \leq
    C\frac{\delta}{s}.
\end{equation}

Combining \eqref{eq:near-transition} and
\eqref{eq:far-transition}, we obtain
\begin{equation}
\label{eq:opposite-weak}
    \PP\left(
        \widehat\eta_i\neq\eta_i,\,
        B_i^c,\,
        \text{opposite-side changes}
    \right)
    \leq
    C\frac{\delta}{s},
    \qquad s\leq1.
\end{equation}

Suppose now that $s>1$. In this regime a direct union bound is
sufficient. From \eqref{eq:opposite-pair},
\[
\begin{split}
    &\PP\left(
        \widehat\eta_i\neq\eta_i,\,
        B_i^c,\,
        \text{opposite-side changes}
    \right) \leq
    C\delta
    \sum_{m=0}^{k}
    \sum_{q=1}^{k}
    \Phi^c\!\left(
        t\sqrt{m+q}
    \right).
\end{split}
\]
For $t\geq1$, Mills' inequalities imply
\[
    \frac{
        \Phi^c(t\sqrt{\ell})
    }{
        \Phi^c(t)
    }
    \leq
    \frac{C}{\sqrt{\ell}}
    \exp\left(
        -\frac{s}{2}(\ell-1)
    \right),
    \qquad \ell\geq1.
\]
Since at most $\ell$ pairs $(m,q)$ satisfy $m+q=\ell$,
\[
\begin{split}
    \sum_{m\geq0}\sum_{q\geq1}
    \Phi^c\!\left(t\sqrt{m+q}\right)
    &\leq
    C\Phi^c(t)
    \sum_{\ell\geq1}
    \sqrt{\ell}\,
    e^{-\frac{s}{2}(\ell-1)}\leq
    C\Phi^c(t),
\end{split}
\]
because $s\geq1$. Thus
\begin{equation}
\label{eq:opposite-strong}
    \PP\left(
        \widehat\eta_i\neq\eta_i,\,
        B_i^c,\,
        \text{opposite-side changes}
    \right)
    \leq
    C\delta\Phi^c(t).
\end{equation}

Putting together
\eqref{eq:two-transitions},
\eqref{eq:constant-path-bound},
\eqref{eq:same-side-bound},
\eqref{eq:opposite-weak}, and
\eqref{eq:opposite-strong}, we obtain the generic interior bound
\begin{equation}
\label{eq:generic-interior-bound}
    \PP(\widehat\eta_i\neq\eta_i)
    \leq
    C\left[
        \delta^2k^2
        +(1+k\delta)e^{-csk}
    \right]
    +
    \begin{cases}
        \displaystyle C\frac{\delta}{s},
        & s\leq1,
        \\[2mm]
        \displaystyle C\delta\Phi^c(t),
        & s>1.
    \end{cases}
\end{equation}

We now use the choice
\[
    k
    =
    \left\lceil
        \frac{C_0}{s}
        \log\frac{s}{\delta}
    \right\rceil.
\]

If $2\delta<s\leq1$, put
\[
    x:=\frac{\delta}{s}\leq\frac12.
\]
Then
\(
    sk\gtrsim
    C_0\log(1/x),\
    k\delta
    \lesssim
    x\log(1/x),
\)
and
\(
    \delta^2k^2
    \lesssim
    x^2\log^2(1/x)
    \lesssim x.
\)
Choosing $C_0$ sufficiently large also gives
\[
    (1+k\delta)e^{-csk}
    \lesssim x.
\]
Hence
\begin{equation}
\label{eq:interior-weak-final}
    \PP(\widehat\eta_i\neq\eta_i)
    \leq
    C\frac{\delta}{s},
    \qquad
    2\delta<s\leq1.
\end{equation}

Consider next
\(
    1<s<2\log(1/\delta).
\)
Write
\[
    q:=\log(1/\delta),
    \qquad
    L:=\log(s/\delta)=q+\log s.
\]
Since $s<2q$, we have $q>s/2$. Moreover,
\[
    k\asymp \frac{L}{s}.
\]
For $t\geq1$, Mills' inequalities give
\begin{equation}
\label{eq:mills-tphi}
    t\Phi^c(t)\asymp e^{-s/2}.
\end{equation}
Using $q\geq s/2$,
\[
\begin{split}
    \delta k^2
    &\lesssim
    e^{-q}
    \frac{(q+\log s)^2}{s^2}\lesssim
    e^{-s/2}.
\end{split}
\]
Consequently,
\begin{equation}
\label{eq:two-change-absorbed}
    \delta^2k^2
    \lesssim
    t\,\delta\Phi^c(t).
\end{equation}
Furthermore, choosing $C_0$ sufficiently large,
\(
    e^{-csk}
    \leq
    e^{-cC_0L}
    \lesssim
    \delta e^{-s/2},
\)
and $k\delta$ is uniformly bounded. Therefore
\begin{equation}
\label{eq:exp-absorbed}
    (1+k\delta)e^{-csk}
    \lesssim
    t\,\delta\Phi^c(t).
\end{equation}
Since $t\geq1$, the term
$\delta\Phi^c(t)$ in
\eqref{eq:generic-interior-bound} is also absorbed. We conclude that
\begin{equation}
\label{eq:interior-strong-final}
    \PP(\widehat\eta_i\neq\eta_i)
    \leq
    Ct\,\delta\Phi^c(t),
    \qquad
    1<s<2\log(1/\delta).
\end{equation}

It remains to control the two sample boundaries. For
\(
    i\notin\llbracket k+1,n-k\rrbracket
\),
the estimator uses the pointwise rule, and therefore
\[
    \PP(\widehat\eta_i\neq\eta_i)
    =
    \Phi^c(t).
\]
There are at most $2k$ such indices. Hence their contribution to the
normalized Hamming risk is at most
\begin{equation}
\label{eq:boundary-contribution}
    \frac{2k}{n}\Phi^c(t).
\end{equation}

The condition
\(
    \delta\geq\frac{\log n}{n}
\)
is used only at this point. Since
\(
    \log(1/\delta)\leq\log n
\)
and, throughout the regime considered here,
\[
    \log(s/\delta)
    \lesssim
    \log(1/\delta),
\]
we have
\[
    \frac{k}{n}
    \lesssim
    \frac{\delta}{s}.
\]
For $s\leq1$, \eqref{eq:boundary-contribution} is therefore bounded by
$C\delta/s$. For $s>1$,
\[
    \frac{k}{n}\Phi^c(t)
    \lesssim
    \frac{\delta}{s}\Phi^c(t)
    \leq
    t\,\delta\Phi^c(t).
\]

Averaging the pointwise bounds over $i$ and using the definition of
the clustering loss proves
\[
\ClusterRisk(\theta,\delta,\pi_n\circ\widehat\eta)
\lesssim
\begin{cases}
    1,
    & s\leq2\delta,
    \\[1mm]
    \delta/s,
    & 2\delta<s\leq1,
    \\[1mm]
    \sqrt{s}\,\delta\Phi^c(\sqrt{s}),
    & 1<s<2\log(1/\delta),
\end{cases}
\]
as claimed.

\section{High-Dimensional and Adaptive Recovery}
\subsection{Proof of Theorem \ref{thm:minimax_lb_high_dim_clustering}}
Let \(\pi_\eta\) be the stationary Markov law with transition probability
\(\delta\), and let \(\pi_\theta\) be any probability measure on
\(\mathbb R^d\), independent of the latent chain. Set
\(\pi=\pi_\theta\otimes\pi_\eta\). We write \(\mathbb E_\pi\) for
expectation under the joint law obtained by drawing \((\theta,\eta)\sim\pi\)
and then \(X_j=\eta_j\theta+\sigma\xi_j\). The reduction below applies
to both the Gaussian and the point-mass priors used subsequently.
All constants are universal, and \(C_0\) in the assumption
\(n\delta\geq C_0\) is taken sufficiently large.

\begin{lemma}
    Let \(\mathcal I\subset[n]\) be a collection of target indices, and for each
\(i\in\mathcal I\) let \(I_i\subset[n]\) be a set containing \(i\), with
\(|I_i^c|\geq n/2\).  Define
\begin{equation*}
    A_i(I_i)
    :=
    \inf_{T_i:(\mathbb R^d)^n\times\{-1,+1\}^{I_i^c}\to[-1,1]}
    \EE_\pi
    \left[
        \frac{1-\eta_iT_i(X_{1:n},\eta_{I_i^c})}{2}
    \right].
\end{equation*}
Then,
\begin{equation}
\label{eq:merged-common-reduction}
    \Psi_{n,d}(\delta,\Delta,\sigma)
    \geq
    \frac1{5n}\sum_{i\in\mathcal I} A_i(I_i)
    -\pi_\theta(\|\theta\|<\Delta).
\end{equation}
\end{lemma}

\begin{proof}
    Fix an arbitrary binary classifier \(\widehat\eta=\widehat\eta(X_{1:n})\).  Let
\[
    s_\star\in\argmin_{s\in\{-1,+1\}}
    \sum_{j=1}^n\mathbf 1_{\{s\widehat\eta_j(X_{1:n})\neq\eta_j\}}.
\]
Hence
\[
    \ell(\widehat\eta,\eta)
    =
    \frac1n\sum_{j=1}^n
    \mathbf 1_{\{s_\star\widehat\eta_j(X_{1:n})\neq\eta_j\}}.
\]
Let \(I\subset[n]\) be any index set with \(|I^c|\geq n/2\).  Define the
outside-orientation sign
\[
    s_{I^c}\in\argmin_{s\in\{-1,+1\}}
    \sum_{j\in I^c}
    \mathbf 1_{\{s\widehat\eta_j(X_{1:n})\neq\eta_j\}}.
\]
Then, for every \(i\in I\),
\begin{equation}
\label{eq:merged-outside-orientation-coordinate}
    \mathbf 1_{\{s_{I^c}\widehat\eta_i\neq\eta_i\}}
    \leq
    \mathbf 1_{\{s_\star\widehat\eta_i\neq\eta_i\}}
    +4\ell(\widehat\eta,\eta).
\end{equation}
Indeed, if \(s_{I^c}=s_\star\), the claim is immediate.  If \(s_{I^c}\neq s_\star\),
then by the definition of \(s_{I^c}\),
\[
    \sum_{j\in I^c}
    \mathbf 1_{\{s_\star\widehat\eta_j\neq\eta_j\}}
    \geq \frac{|I^c|}{2}.
\]
Consequently,
\[
    \mathbf 1_{\{s_{I^c}\neq s_\star\}}
    \leq
    \frac{2}{|I^c|}
    \sum_{j\in I^c}
    \mathbf 1_{\{s_\star\widehat\eta_j\neq\eta_j\}}
    \leq
    \frac4n
    \sum_{j=1}^n
    \mathbf 1_{\{s_\star\widehat\eta_j\neq\eta_j\}}
    =4\ell(\widehat\eta,\eta),
\]
where we used \(|I^c|\geq n/2\).  Since
\[
    \mathbf 1_{\{s_{I^c}\widehat\eta_i\neq\eta_i\}}
    \leq
    \mathbf 1_{\{s_\star\widehat\eta_i\neq\eta_i\}}
    +\mathbf 1_{\{s_{I^c}\neq s_\star\}},
\]
\eqref{eq:merged-outside-orientation-coordinate} follows.

Now for each
\(i\in\mathcal I\),
\[
    A_i(I_i)
    \leq
    \EE_\pi
    \left[
        \frac{1-\eta_i s_{I_i^c}\widehat\eta_i(X_{1:n})}{2}
    \right]
    =
    \EE_\pi
    \left[
        \mathbf 1_{\{s_{I_i^c}\widehat\eta_i(X_{1:n})\neq\eta_i\}}
    \right].
\]
Summing \eqref{eq:merged-outside-orientation-coordinate} over
\(i\in\mathcal I\), dividing by \(n\), we obtain
\[
    \frac1n\sum_{i\in\mathcal I}A_i(I_i)
    \leq
    5\mathbb E_\pi[\ell(\widehat\eta(X_{1:n}),\eta)].
\]
Therefore,
\[
    \inf_{\widehat\eta}
    \mathbb E_\pi[\ell(\widehat\eta(X_{1:n}),\eta)]
    \geq
    \frac1{5n}\sum_{i\in\mathcal I}A_i(I_i).
\]
Finally, since \(0\leq \ell \leq1\), for every fixed \(\widehat\eta\),
\[
\begin{aligned}
    \sup_{\|\theta\|\geq\Delta}
    \mathbb E_\theta[\ell(\widehat\eta(X_{1:n}),\eta)]
    &\geq
    \int_{\{\|\theta\|\geq\Delta\}}
    \mathbb E_\theta[\ell(\widehat\eta(X_{1:n}),\eta)]
    \,d\pi_\theta(\theta) \geq
    \mathbb E_\pi[\ell(\widehat\eta(X_{1:n}),\eta)]
    -\pi_\theta(\|\theta\|<\Delta).
\end{aligned}
\]
Taking the infimum over \(\widehat\eta\) proves
\eqref{eq:merged-common-reduction}.
\end{proof}

\textbf{Weak signal regime.}
Assume \(r_n^2\leq1\), and use the Gaussian prior
\(\pi_\theta=\mathcal N(0,\alpha^2I_d)\), with its variance specified below.
Set
\[
 k:=\left\lfloor\frac1\delta\wedge\frac1{r_n^2}\right\rfloor,
 \qquad \mathcal I_k:=\{2k+1,\ldots,n-2k\},
 \qquad I_i:=\{i-2k,\ldots,i+2k\}.
\]
Since \(\delta\leq1/2\) and \(r_n^2\leq1\),
\[
 1\leq k\leq\frac1\delta\leq\frac n{C_0},\qquad
 k\delta\leq1,\qquad kr_n^2\leq1,\qquad
 k\asymp\frac1\delta\wedge\frac1{r_n^2}.
\]
For sufficiently large \(C_0\), these inequalities imply
\(|\mathcal I_k|\geq n/2\) and
\(|I_i^c|=n-4k-1\geq n/2\), so
\eqref{eq:merged-common-reduction} applies.

Fix \(i\in\mathcal I_k\).  For
\(t\in\{i-2k,\dots,i+2k-1\}\), define the local one-change cylinder
\[
    C_t:=
    \{\eta_{i-2k}=\cdots=\eta_t=+1,
      \eta_{t+1}=\cdots=\eta_{i+2k}=-1\}.
\]
The cylinders \(C_t\) are disjoint and
\begin{equation*}
    \mathbb P_\pi(C_t)
    =
    \frac12\delta(1-\delta)^{4k-1}
    \geq
    c\delta(1-\delta)^{4k}.
\end{equation*}
Let \(m\in\llbracket1,k\rrbracket\).  On \(C_{i-m}\),
\(\eta_i=-1\), whereas on \(C_{i+m}\), \(\eta_i=+1\).  Let \(P_{i+m}\) denote
the conditional law of \((X_{1:n},\eta_{I_i^c})\) given \(C_{i+m}\), and define \(P_{i-m}\)
analogously.  Restricting the expectation in \(A_i(I_i)\) to these disjoint
cylinders, we get
\begin{equation*}
\begin{aligned}
    A_i(I_i)
    &\geq
    c\delta(1-\delta)^{4k}
    \sum_{m=1}^k
    \inf_{T:(\mathbb R^d)^n\times\{-1,+1\}^{I_i^c}\to[-1,1]}
    \Bigg\{
    \mathbb E_{P_{i+m}}\left[\frac{1-T}{2}\right]
    +
    \mathbb E_{P_{i-m}}\left[\frac{1+T}{2}\right]
    \Bigg\}.
\end{aligned}
\end{equation*}
Thus
\begin{equation}
\label{eq:merged-Ai-TV}
    A_i(I_i)
    \geq
    c\delta(1-\delta)^{4k}
    \sum_{m=1}^k
    \left(1-\|P_{i+m}-P_{i-m}\|_{\rm TV}\right).
\end{equation}

We now bound the total variation distance. For a deterministic path
\(u\in\{-1,+1\}^n\), let \(P_u^X\) be the marginal law of \(X_{1:n}\) after
integrating out \(\theta\sim \mathcal{N}(0,\alpha^2I_d)\) in
\(X_j=u_j\theta+\sigma\xi_j\).  If \(X\) is viewed as a \(d\times n\) matrix,
then the rows are i.i.d. \(\mathcal{N}(0,\Sigma_u)\), where
\[
    \Sigma_u=\sigma^2I_n+\alpha^2uu^\top.
\]
Setting \(\beta=\alpha^2/\sigma^2\), one has
\(\Sigma_u=\sigma^2(I_n+\beta uu^\top)\).  Since \(\|u\|^2=n\),
\(\det(\Sigma_u)=\sigma^{2n}(1+\beta n)\), which is independent of \(u\).

We detail the computation of the Kullback--Leibler divergence between the two
integrated laws associated with deterministic label paths
\(u,v\in\{-1,+1\}^n\).  Under the path \(u\), after integrating out
\(\theta\sim \mathcal{N}(0,\alpha^2 I_d)\), each row of the \(d\times n\) observation
matrix \(X\) is distributed as a centered Gaussian vector in \(\mathbb R^n\)
with covariance
\[
    \Sigma_u=\sigma^2I_n+\alpha^2uu^\top.
\]
The rows are independent.  Hence
\[
    P_u^X=\mathcal{N}(0,\Sigma_u)^{\otimes d},
    \qquad
    P_v^X=\mathcal{N}(0,\Sigma_v)^{\otimes d}.
\]
For two centered Gaussian laws on \(\mathbb R^n\), one has
\begin{equation}
\label{eq:kl-row-product-expanded}
    \mathrm{KL}(P_u^X,P_v^X)
    =
    \frac d2
    \left[
        \operatorname{tr}(\Sigma_v^{-1}\Sigma_u)-n
    \right].
\end{equation}
It remains to compute the trace.  By the Sherman--Morrison formula,
\[
    \Sigma_v^{-1}
    =
    \frac1{\sigma^2}
    \left(
        I_n-
        \frac{\beta}{1+\beta n}vv^\top
    \right).
\]
Thus
\begin{align*}
    \Sigma_v^{-1}\Sigma_u
    &=
    I_n
    +\beta uu^\top
    -\frac{\beta}{1+\beta n}vv^\top
    -\frac{\beta^2}{1+\beta n}vv^\top uu^\top.
\end{align*}
Hence
\begin{align*}
    \operatorname{tr}(\Sigma_v^{-1}\Sigma_u)-n
    &=
    \frac{\beta^2}{1+\beta n}
    \left(n^2-\langle u,v\rangle^2\right).
\end{align*}
Substituting this into \eqref{eq:kl-row-product-expanded}, we obtain the exact
identity
\begin{equation}
\label{eq:kl-exact-label-paths-expanded}
    \mathrm{KL}(P_u^X,P_v^X)
    =
    \frac d2
    \frac{\beta^2}{1+\beta n}
    \left(n^2-\langle u,v\rangle^2\right).
\end{equation}
We now express this in terms of the Hamming distance.  Let
\[
    h(u,v):=\sum_{i=1}^{n} 1_{u_i\neq v_i}.
\]
For each coordinate, \(u_jv_j=1\) if \(u_j=v_j\), and \(u_jv_j=-1\) if
\(u_j\neq v_j\).  Hence
\[
    \langle u,v\rangle
    =
    \sum_{j=1}^n u_jv_j
    =
    (n-h(u,v))-h(u,v)
    =
    n-2h(u,v).
\]
and
\begin{equation*}
    n^2-\langle u,v\rangle^2=
    4h(u,v)(n-h(u,v))\leq 4nh(u, v).
\end{equation*}
Combining this with \eqref{eq:kl-exact-label-paths-expanded} gives
\begin{equation*}
    \mathrm{KL}(P_u^X,P_v^X)
    \leq
    2d\frac{\beta^2 n}{1+\beta n}h(u,v).
\end{equation*}
Finally, we rewrite the coefficient in terms of the effective signal. Set
\[
    a:=\frac{\Delta^2}{\sigma^2}.
\]
We shall use the Gaussian prior with variance
\(
    \alpha^2=\Delta^2/d(1-\nu)
\)
with \(\nu\in(0,1/2)\).
Then \(\beta=a/[d(1-\nu)]\).  Since \(\nu\in(0,1/2)\) is fixed,
\(1/(1-\nu)\) is bounded above and below by universal constants.  Therefore
\[
    d\frac{\beta^2n}{1+\beta n}
    \asymp
    d\frac{(a/d)^2n}{1+an/d}\asymp r_n^2.
\]
We obtain for a universal constant \(C>0\), 
\begin{equation*}
    \mathrm{KL}(P_u^X,P_v^X)
    \leq
    C h(u,v) r_n^2.
\end{equation*}
Conditional on either cylinder \(C_{i+m}\) or \(C_{i-m}\), the outside labels $\eta_{I^c_i}$
have the same conditional law.  Indeed, both local patterns have endpoint values
\(\eta_{i-2k}=+1\) and \(\eta_{i+2k}=-1\), so by the Markov property the law of
\(\eta_{I_i^c}\) depends on the segment $I_i$ only through these endpoints.  Let
this common law be denoted by \(\mu_i\).  For each outside configuration
\(w\), let \(u_w^{(i+m)}\) and \(u_w^{(i-m)}\) be the corresponding full label paths
obtained by concatenating \(w\) with the local patterns of \(C_{i+m}\) and
\(C_{i-m}\), respectively. Then
\[
    h\bigl(u_w^{(i+m)},u_w^{(i-m)}\bigr)=2m
    \qquad\text{for every }w.
\]
Moreover, the two conditional laws of the observed pair
\((X_{1:n},\eta_{I_i^c})\) can be written as:
\[
    P_{i+m}(dx,w)=\mu_i(w)P^X_{u_w^{(i+m)}}(dx),
    \qquad
    P_{i-m}(dx,w)=\mu_i(w)P^X_{u_w^{(i-m)}}(dx).
\]
Thus,
\[
\begin{aligned}
    \mathrm{KL}(P_{i+m},P_{i-m})\leq 
     \sum_{w}
    \mathrm{KL}\bigl(P^X_{u_w^{(i+m)}},P^X_{u_w^{(i-m)}}\bigr)
    \,\mu_i(w)\leq
    \sum_{w} C h\bigl(u_w^{(i+m)},u_w^{(i-m)}\bigr)r_n^2\,\mu_i(w) \leq Cmr_n^2,
\end{aligned}
\]
after changing the universal constant \(C\). For \(m\leq k\),
\(mr_n^2\leq1\), so the Bretagnolle--Huber inequality yields
\[
 1-\|P_{i+m}-P_{i-m}\|_{\mathrm{TV}}
 \geq\frac12\exp\{-\mathrm{KL}(P_{i+m},P_{i-m})\}
 \geq\frac12e^{-C}.
\]
Furthermore, \(\log(1-\delta)\geq-2\delta\) for
\(0<\delta\leq1/2\), and hence
\((1-\delta)^{4k}\geq e^{-8k\delta}\geq e^{-8}\).
Substituting into \eqref{eq:merged-Ai-TV} gives
\begin{equation}
\label{eq:merged-Ai-intermediate-final}
 A_i(I_i)\geq c\delta k,\qquad i\in\mathcal I_k.
\end{equation}
Since \(\|\theta\|^2/\alpha^2\sim\chi_d^2\), the
usual lower-tail inequality gives
\begin{equation}
\label{eq:merged-prior-leakage}
    \pi_\theta(\|\theta\|<\Delta)
    \leq
    \exp\left(-\frac d4\nu^2\right).
\end{equation}
Combining \eqref{eq:merged-common-reduction},
\eqref{eq:merged-Ai-intermediate-final} and \eqref{eq:merged-prior-leakage} and using \(|\mathcal I_k|\geq n/2\),
we obtain
\[
    \Psi_{n,d}(\delta,\Delta,\sigma)
    \geq
    c\delta k-
    \exp\left(-\frac d4\nu^2\right).
\]
When $n r^2_n \geq C_0$, choosing \(\nu\) as a fixed sufficiently small numerical constant gives
\[
    \Psi_{n,d}(\delta,\Delta,\sigma)
    \geq
    c\left(1\wedge\frac{\delta}{r_n^2}\right)-e^{-cd}.
\]
We can separate two regimes:
\begin{itemize}
    \item[-] In the regime \(a\leq d/n\), \(nr_n^2\geq C_0\) and \(r_n^2\leq d/n\) which also implies \(d\geq C_0\). Thus, by choosing \(C_0\) sufficiently large,
the prior leakage \(e^{-cd}\) is at most one half of the positive term
\(c(1\wedge\delta/r_n^2)\) in the preceding lower bound. Consequently,
\[
 \Psi_{n,d}(\delta,\Delta,\sigma)
 \geq c\left(1\wedge\frac{\delta}{r_n^2}\right).
\]
    \item[-] 
In the regime \(a>d/n\), we have \(a/2<r_n^2\leq a\), and hence
\(a<2\) since \(r_n^2\leq1\). Fix \(\theta_0\) with
\(\|\theta_0\|=\Delta\). The proof of
\eqref{eq:merged-common-reduction} uses no Gaussian property of the prior
and therefore also applies to the point-mass prior at \(\theta_0\), for
which the leakage term is zero. Write \(A_i^{(0)}(I_i)\) for
\(A_i(I_i)\) evaluated under this fixed signal. Thus
\[
    \Psi_{n,d}(\delta,\Delta,\sigma)
    \geq \frac1{5n}\sum_{i\in\mathcal I}A_i^{(0)}(I_i).
\]
Choose
\[
    k_0:=1\vee\left\lfloor\frac1\delta\wedge\frac1a\right\rfloor,
    \qquad
    \mathcal I:=\{2k_0+1,\ldots,n-2k_0\},
    \qquad
    I_i:=\{i-2k_0,\ldots,i+2k_0\}.
\]
Then \(k_0\asymp\delta^{-1}\wedge a^{-1}\),
\(k_0\delta\leq1\), and \(k_0a\leq2\). Under
\(n\delta\geq C_0\), with \(C_0\) sufficiently large,
\(|\mathcal I|\geq n/2\) and \(|I_i^c|\geq n/2\).
Use the same pairs of one-change cylinders as above, now with window
radius \(2k_0\). For \(1\leq m\leq k_0\), let
\(P_{i+m}^{(0)}\) and \(P_{i-m}^{(0)}\) be the conditional laws of
\((X_{1:n},\eta_{I_i^c})\) under the fixed signal \(\theta_0\).
The outside labels have the same law under the two cylinders, and the
conditional Gaussian means differ by \(2\theta_0\) at exactly \(2m\)
coordinates. Therefore
\[
    \mathrm{KL}\bigl(P_{i+m}^{(0)},P_{i-m}^{(0)}\bigr)
    =4m\frac{\|\theta_0\|^2}{\sigma^2}
    =4ma\leq8.
\]
The Bretagnolle--Huber inequality gives
\(1-\|P_{i+m}^{(0)}-P_{i-m}^{(0)}\|_{\mathrm{TV}}
\geq e^{-8}/2\). Moreover,
\((1-\delta)^{4k_0-1}\geq c\), since \(k_0\delta\leq1\) and
\(\delta\leq1/2\). The same cylinder-testing argument as in
\eqref{eq:merged-Ai-TV} consequently yields
\(A_i^{(0)}(I_i)\geq c\delta k_0\). Averaging over \(i\in\mathcal I\),
we obtain
\[
    \Psi_{n,d}(\delta,\Delta,\sigma)
    \geq c\delta k_0
    \geq c\left(1\wedge\frac{\delta}{a}\right)
    \geq c\left(1\wedge\frac{\delta}{r_n^2}\right).
\]
\end{itemize}
Finally, suppose \(r_n^2(\Delta)<C_0/n\). Choose \(\Delta'\geq\Delta\)
such that \(r_n^2(\Delta')=C_0/n\); this is possible because
\(r_n^2(\cdot)\) is continuous and strictly increasing from zero to
infinity. Since \(C_0/n\leq\delta\leq1/2\), the weak-signal bound
already established applies at \(\Delta'\). The parameter classes are
nested, so
\[
 \Psi_{n,d}(\delta,\Delta,\sigma)
 \geq\Psi_{n,d}(\delta,\Delta',\sigma)
 \geq c\left(1\wedge\frac{n\delta}{C_0}\right)=c.
\]
\textbf{Strong signal regime.}
\label{step:strong-signal-regime}
Assume \(r_n^2>1\), and recall
\[
 a:=\frac{\Delta^2}{\sigma^2},\qquad
 r_n^2=\frac{a^2}{a+d/n}\leq a.
\]
In particular, \(a>1\). We first record a fixed-signal bound and then
treat the regime in which the unknown direction affects the rate.

Fix \(\theta_0\) with \(\|\theta_0\|=\Delta\), and use the point-mass
prior at \(\theta_0\) in \eqref{eq:merged-common-reduction}, with
\(\mathcal I=\{2,\ldots,n-1\}\) and \(I_i=\{i\}\). Denote the
corresponding local risks by \(A_i^{(0)}\). On
\(\{\eta_{i-1}\neq\eta_{i+1}\}\), the conditional probabilities of
\(\eta_i=1\) and \(\eta_i=-1\), given \(\eta_{-i}\), are equal.
For fixed \(\theta_0\), observations other than \(X_i\) carry no
additional information about \(\eta_i\) given \(\eta_{-i}\).
The conditional Bayes error is therefore \(\Phi^c(\sqrt a)\).
Since this event has probability \(2\delta(1-\delta)\), the reduction,
with no prior leakage, gives
\begin{equation}
\label{eq:fixed-signal-pointmass-lower}
 \Psi_{n,d}(\delta,\Delta,\sigma)
 \geq\frac{n-2}{5n}\,2\delta(1-\delta)\Phi^c(\sqrt a)
 \geq c\delta\Phi^c(\sqrt a).
\end{equation}
Here \(n\geq4\) follows by taking \(C_0\) sufficiently large.
If \(a\geq d\log n/n\), then
\[
 \sqrt a=r_n\sqrt{1+\frac{d}{na}}
 \leq r_n\left(1+\frac1{\log n}\right),
\]
and the desired lower bound follows from
\eqref{eq:fixed-signal-pointmass-lower}.

It remains to consider \(1<a<d\log n/n\). Use the Gaussian prior
\[
 \pi_\theta=\mathcal N(0,\alpha_n^2I_d),\qquad
 \alpha_n^2=\frac{\Delta^2}{d(1-\nu_n)},
\]
where \(\nu_n\) will be specified below. Again take
\(\mathcal I=\{2,\ldots,n-1\}\), \(I_i=\{i\}\), and write
\(A_i=A_i(\{i\})\). Then
\[
 \Psi_{n,d}(\delta,\Delta,\sigma)
 \geq\frac1{5n}\sum_{i=2}^{n-1}A_i
       -\pi_\theta(\|\theta\|<\Delta).
\]
We restrict the local Bayes error to the event
\(E_i:=\{\eta_{i-1}=\eta_i\neq\eta_{i+1}\}\), whose probability is
\(\delta(1-\delta)\). On this event the neighboring states are distinct.
Writing \(S_i:=\sum_{j\neq i}\eta_jX_j\), Gaussian integration shows
that the posterior odds for \(\eta_i=1\) against \(\eta_i=-1\), when
the neighboring states are distinct, are
\[
 \exp\left\{\frac{2\langle X_i,S_i\rangle}
 {n\sigma^2+\sigma^4/\alpha_n^2}\right\}.
\]
Thus the Bayes rule is \(\operatorname{sign}\langle X_i,S_i\rangle\)
on this event, and
\[
 A_i\geq\mathbb P_\pi\bigl(E_i,\ \langle\eta_iX_i,S_i\rangle<0\bigr).
\]
Put \(u=\theta/\sigma\), \(b=\|u\|^2\), and
\[
 Z:=\frac1{\sqrt{n-1}}\sum_{j\neq i}\eta_j\xi_j,
 \qquad W:=u+\frac{Z}{\sqrt{n-1}}.
\]
Conditionally on \((\theta,\eta)\), \(Z\sim\mathcal N(0,I_d)\) is
independent of \(\xi_i\). Conditionally on \((\theta,\eta,Z)\),
\(-\langle\eta_i\xi_i,W/\|W\|\rangle\) is standard normal, with a
law independent of \(Z\). The conditional error probability is therefore
\(\Phi^c(\langle u,W\rangle/\|W\|)\). Its expectation over \(Z\)
depends on \(\theta\) but not on the latent path. Consequently,
\[
 A_i\geq\delta(1-\delta)\,
 \mathbb E_{\theta,Z}\left[
 \Phi^c\left(\frac{\langle u,W\rangle}{\|W\|}\right)\right].
\]
The zero-denominator event has probability zero.

For \(0<\beta_n<1/2\) and \(0<\xi_n<1\), define
\[
 \mathcal A_\theta:=
 \left\{\|Z\|^2\geq d(1-\xi_n),\quad
 \frac{|\langle Z,u\rangle|}{\sqrt{n-1}}\leq\beta_n b\right\}.
\]
On this event, the numerator is positive and
\[
 \frac{\langle u,W\rangle}{\|W\|}
 \leq\frac{(1+\beta_n)b}
 {\sqrt{(1-2\beta_n)b+d(1-\xi_n)/(n-1)}}.
\]
Gaussian and chi-square tail bounds give
\[
 \mathbb P(\mathcal A_\theta\mid\theta)
 \geq1-e^{-d\xi_n^2/4}
          -2e^{-\beta_n^2b(n-1)/2}.
\]
Take \(0<\gamma_n<1/2\), \(\gamma_n\geq4\nu_n\), and define
\(\mathcal B=\{|\|\theta\|^2/\Delta^2-1|\leq\gamma_n\}\).
Since \(\|\theta\|^2/\alpha_n^2\sim\chi_d^2\),
\[
 \pi_\theta(\mathcal B^c)\leq2e^{-c d\gamma_n^2},\qquad
 \pi_\theta(\|\theta\|<\Delta)\leq e^{-d\nu_n^2/4}.
\]
Restricting directly to \(\mathcal B\cap\mathcal A_\theta\), where
\((1-\gamma_n)a\leq b\leq(1+\gamma_n)a\), yields
\[
 A_i\geq c\delta\,p_n\,
 \Phi^c\left(
 \frac{a(1+\beta_n)(1+\gamma_n)}
 {\sqrt{a(1-2\beta_n)(1-\gamma_n)+d(1-\xi_n)/(n-1)}}
 \right),
\]
provided the following factors are nonnegative:
\[
 p_n:=
 \left(1-2e^{-c d\gamma_n^2}\right)
 \left(1-e^{-d\xi_n^2/4}
          -2e^{-\beta_n^2(1-\gamma_n)a(n-1)/2}\right).
\]
Choose explicitly
\[
 \beta_n=(na)^{-1/4},\qquad \xi_n=d^{-1/4},\qquad
 \nu_n=\sqrt{\frac{na}{d\log^2n}},\qquad \gamma_n=4\nu_n.
\]
Because \(a>1\) and \(d>na/\log n\), these choices satisfy, uniformly
in the present regime,
\[
 \beta_n\leq n^{-1/4},\quad
 \xi_n\leq(\log n/n)^{1/4},\quad
 \gamma_n\leq\frac4{\sqrt{\log n}}.
\]
Moreover,
\[
 d\xi_n^2=\sqrt d\longrightarrow\infty,\qquad
 \beta_n^2a(n-1)\asymp\sqrt{na}\longrightarrow\infty,
 \qquad d\gamma_n^2=\frac{16na}{\log^2n}\longrightarrow\infty.
\]
Thus, after increasing the universal constant \(C_0\) if necessary,
all required parameter restrictions hold and \(p_n\geq1/2\).
Since \(d/(n-1)\geq d/n\), the argument of the Gaussian tail is at most
\[
 r_n\{1+C(\beta_n+\gamma_n+\xi_n)\}
 \leq r_n(1+\varepsilon_n),\qquad
 \varepsilon_n:=\frac K{\sqrt{\log n}},
\]
for a sufficiently large universal \(K\). Taking \(C_0\) larger also
ensures \(\varepsilon_n\leq1\). The common reduction now gives
\[
 \Psi_{n,d}(\delta,\Delta,\sigma)
 \geq c\delta\Phi^c\bigl(r_n(1+\varepsilon_n)\bigr)
       -\exp\left\{-\frac{na}{4\log^2n}\right\}.
\]
To absorb the last term, Mills' inequality, \(r_n>1\), and
\(\delta\geq C_0/n\) imply
\[
 \log\frac1{\delta\Phi^c(r_n(1+\varepsilon_n))}
 \leq C(r_n^2+\log n).
\]
On the other hand, \(r_n^2\leq a\) and \(a>1\) give
\[
 \frac{na/(4\log^2n)}{r_n^2+\log n}
 \geq\frac{n}{4\log^2n(1+\log n)}\longrightarrow\infty.
\]
The prior leakage is therefore negligible relative to the positive term,
uniformly in this regime. Increasing \(C_0\) once more absorbs it and yields
\[
 \Psi_{n,d}(\delta,\Delta,\sigma)
 \geq c\delta\Phi^c\bigl(r_n(1+\varepsilon_n)\bigr).
\]
The same choice \(\varepsilon_n=K/\sqrt{\log n}\) also covers the
fixed-signal case \(a\geq d\log n/n\), since
\(\varepsilon_n\geq1/\log n\). Hence the strong-signal lower bound
holds with a deterministic \(\varepsilon_n\to0\) as \(n\to\infty\).

\subsection{Concentration properties of the block Gram matrix}\label{sec:concentration:lemmas}

Throughout this subsection, $n$ is dyadic, as in
Section~\ref{sec:adaptation}. For $\ell\in\mathcal L_n$, let
$k=n/\ell$, and let $I_1,\ldots,I_\ell$ be the consecutive
blocks of size $k$. The block-average matrix satisfies
\[
    \widetilde X^{(\ell)}
    =
    \theta\bigl(\overline\eta^{(\ell)}\bigr)^\top
    +\frac{\sigma}{\sqrt{k}}W,
    \qquad
    \overline\eta_r^{(\ell)}
    =\frac1k\sum_{i\in I_r}\eta_i,
\]
where $W\in\mathbb R^{d\times\ell}$ has independent standard
Gaussian entries and is independent of the latent chain.
For a fixed scale $\ell$, we abbreviate
$\overline\eta^{(\ell)}$ by $\overline\eta$.
We also write
\[
    a=\frac{\|\theta\|^2}{\sigma^2},
    \qquad q=\frac dn,
    \qquad r_n^2=r_n^2(\theta)=\frac{a^2}{a+q}.
\]

\begin{lemma}[Energy of the block labels]
\label{lem:block-label-energy}
Assume that
\(
    \delta\le \frac{\ell}{n}.
\)
Then there exist universal constants $c,C>0$ such that
\[
    \PP\left(
        \frac{\ell}{4}
        \le
        \|\overline\eta\|^2
        \le
        \ell
    \right)
    \ge
    1-C\exp(-c\ell).
\]
\end{lemma}

\begin{proof}
Introduce the transition variables
\[
    \varepsilon_i:=\eta_i\eta_{i+1},
    \qquad 1\le i\le n-1.
\]
Under the stationary symmetric two-state Markov model, the variables
$(\varepsilon_i)_{i=1}^{n-1}$ are independent and satisfy
\(
    \PP(\varepsilon_i=-1)=\delta,\quad
    \PP(\varepsilon_i=1)=1-\delta.
\)

For each block $I_r$, the quantity $\overline\eta_r^2$ is invariant
under a simultaneous change of sign of all labels in the block.
Consequently, $\overline\eta_r^2$ depends only on the transition
variables $\varepsilon_i$ corresponding to edges entirely contained
in $I_r$. Since different blocks involve disjoint sets of transition
variables, the random variables
\(
    \overline\eta_1^2,\ldots,\overline\eta_\ell^2
\)
are independent. Moreover,
\(
    0\le \overline\eta_r^2\le1.
\)

If no transition occurs inside $I_r$, then all labels in this block
are identical and therefore $\overline\eta_r^2=1$. It follows that
\[
    \EE[\overline\eta_r^2]
    \ge
    (1-\delta)^{k-1}.
\]
Since $\delta k\le1$, we have, for $k\ge2$,
\(
    (1-\delta)^{k-1}
    \ge e^{-1}
    \ge\frac13,
\)
while the claim is immediate for $k=1$. Hence
\[
    \EE\|\overline\eta\|^2
    =
    \sum_{r=1}^\ell\EE[\overline\eta_r^2]
    \ge \frac{\ell}{3}.
\]

By Hoeffding's inequality,
\[
    \PP\left(
        \|\overline\eta\|^2
        -
        \EE\|\overline\eta\|^2
        \le -t
    \right)
    \le
    \exp\left(-c\frac{t^2}{\ell}\right).
\]
Taking $t=\ell/12$ gives
\[
    \PP\left(
        \|\overline\eta\|^2<\frac{\ell}{4}
    \right)
    \le
    C\exp(-c\ell).
\]
\end{proof}

We next turn to the corresponding Gram matrix. Define
\[
    G_\ell
    :=
    \frac1\ell
    \bigl(\widetilde X^{(\ell)}\bigr)^\top
    \widetilde X^{(\ell)}
    -
    \frac{d\sigma^2}{n}I_\ell.
\]
\begin{lemma}[Concentration of the block Gram matrix]
\label{lem:block-gram}
There exist universal constants $c,C>0$ such that, for every dyadic
$\ell\le n$, with probability at least
\(
    1-C\exp(-c\ell),
\)
one has
\begin{equation}
\left\|
    G_\ell
    -
    \frac{\|\theta\|^2}{\ell}
    \overline\eta\overline\eta^\top
\right\|_{\op}
\le
    4\sigma\|\theta\|
    \sqrt{\frac{\ell}{n}}
 +
    4\sigma^2
    \frac{\ell+\sqrt{d\ell}}{n}.
    \label{eq:block-gram-perturbation}
\end{equation}
If, in addition, $\delta\le\frac{\ell}{n}$,
then, with probability at least $1-C\exp(-c\ell)$,
\begin{equation}
\left|
    \lambda_{\max}(G_\ell)-\|\theta\|^2
\right|
\le
    \frac34\|\theta\|^2
    +
    4\sigma\|\theta\|
    \sqrt{\frac{\ell}{n}}
 +
    4\sigma^2
    \frac{\ell+\sqrt{d\ell}}{n}.
    \label{eq:block-gram-eigenvalue}
\end{equation}
Consequently, there exists a universal constant
$c_0>0$ such that, whenever
\[
    \delta
    \le
    \frac{\ell}{n}
    \le
    c_0 r_n^2,
\]
then with probability
at least $1-C\exp(-c\ell)$:
\begin{equation}
    \left|
        \lambda_{\max}(G_\ell)-\|\theta\|^2
    \right|
    \le
    \frac9{10}\|\theta\|^2,
\label{eq:block-gram-signal-eigenvalue}
\end{equation}
and, for any unit top eigenvector $\widehat u_\ell$ of $G_\ell$,
\begin{equation}
    \left|
        \left\langle
            \widehat u_\ell,
            \frac{\overline\eta}{\|\overline\eta\|}
        \right\rangle
    \right|^2
    \ge\frac34.
\label{eq:block-gram-alignment}
\end{equation}
Conversely, there exists a universal constant
$C_0>0$ such that, whenever
\(
    \frac{\ell}{n}\ge C_0r_n^2,
\)
\begin{equation}
    \left|\lambda_{\max}(G_\ell)\right|
    \le
    10\sigma^2
    \frac{\ell+\sqrt{d\ell}}{n}
\label{eq:block-gram-noise-regime}
\end{equation}
with probability at least $1-C\exp(-c\ell)$.
\end{lemma}

\begin{proof}
Expanding the Gram matrix gives
\[
    \bigl(\widetilde X^{(\ell)}\bigr)^\top
\widetilde X^{(\ell)}
    =
    \|\theta\|^2
    \overline\eta\overline\eta^\top
    +
    \frac{\sigma}{\sqrt{k}}
    \left(
        \overline\eta\theta^\top W
        +
        W^\top\theta\overline\eta^\top
    \right)
    +
    \frac{\sigma^2}{k}W^\top W.
\]
Since $k\ell=n$, it follows that
\begin{align*}
G_\ell
-
\frac{\|\theta\|^2}{\ell}
\overline\eta\overline\eta^\top
=
\frac{\sigma}{\ell\sqrt{k}}
\left(
    \overline\eta\theta^\top W
    +
    W^\top\theta\overline\eta^\top
\right) +
\frac{\sigma^2}{n}
\left(
    W^\top W-dI_\ell
\right).
\end{align*}
Therefore
\begin{align}
\left\|
G_\ell
-
\frac{\|\theta\|^2}{\ell}
\overline\eta\overline\eta^\top
\right\|_{\op}
\le{}&
\frac{2\sigma}{\ell\sqrt{k}}
\|\overline\eta\|
\|W^\top\theta\|
 +
\frac{\sigma^2}{n}
\|W^\top W-dI_\ell\|_{\op}.
\label{eq:gram-decomposition-bound}
\end{align}

When $\theta\neq0$,
\(
    \frac{W^\top\theta}{\|\theta\|}
    \sim \mathcal{N}(0,I_\ell),
\)
and hence standard Gaussian concentration yields
\(
    \|W^\top\theta\|
    \le
    2\|\theta\|\sqrt{\ell}
\)
with probability at least $1-\exp(-c\ell)$. The corresponding term
vanishes identically when $\theta=0$. Moreover, standard bounds for the extreme singular values of a
Gaussian matrix imply
\[
    \|W^\top W-dI_\ell\|_{\op}
    \le
    4\left(
        \ell+\sqrt{d\ell}
    \right)
\]
with probability at least $1-C\exp(-c\ell)$. Since
\[
    \|\overline\eta\|\le\sqrt{\ell}
\]
deterministically, we obtain
\[
\left\|
G_\ell
-
\frac{\|\theta\|^2}{\ell}
\overline\eta\overline\eta^\top
\right\|_{\op}
\le
4\sigma\|\theta\|\sqrt{\frac{\ell}{n}}
+
4\sigma^2
\frac{\ell+\sqrt{d\ell}}{n}.
\]
This proves \eqref{eq:block-gram-perturbation}. Suppose now that
\(
    \delta\le\frac{\ell}{n}.
\)
By Lemma~\ref{lem:block-label-energy}, outside an event of probability
at most $C\exp(-c\ell)$,
\[
    \frac14
    \le
    \frac{\|\overline\eta\|^2}{\ell}
    \le1.
\]
The rank-one matrix
\[
    S_\ell
    :=
    \frac{\|\theta\|^2}{\ell}
    \overline\eta\overline\eta^\top
\]
has unique non-zero eigenvalue
\(
    \lambda_\star
    =
    \|\theta\|^2
    \|\overline\eta\|^2/\ell.
\)
Consequently,
\(
    \frac14\|\theta\|^2
    \le
    \lambda_\star
    \le
    \|\theta\|^2
\)
and
\(
    |\lambda_\star-\|\theta\|^2|
    \le
    \frac34\|\theta\|^2.
\)
Weyl's inequality together with
\eqref{eq:block-gram-perturbation} therefore gives
\eqref{eq:block-gram-eigenvalue}.

We now consider the regime
\[
    \frac{\ell}{n}\le c_0r_n^2.
\]
Set
\[
    t:=\frac{\ell}{n}.
\]
Recall that
\[
    r_n^2=\frac{a^2}{a+q},
    \qquad
    a=\frac{\|\theta\|^2}{\sigma^2},
    \qquad
    q=\frac dn.
\]
The inequality $t\le c_0r_n^2$ yields
\[
    \frac{t}{a}
    \le
    c_0\frac{a}{a+q}
    \le c_0
\]
and
\[
    \frac{\sqrt{qt}}{a}
    \le
    \sqrt{
        c_0\frac{q}{a+q}
    }
    \le\sqrt{c_0}.
\]
It follows that
\begin{align*}
&
4\sigma\|\theta\|\sqrt{t}
+
4\sigma^2(t+\sqrt{qt})
\\
&\qquad\le
\|\theta\|^2
\left(
    8\sqrt{c_0}+4c_0
\right).
\end{align*}
Choosing $c_0>0$ sufficiently small, the right-hand side is at most
$3\|\theta\|^2/20$. Combining this with
\eqref{eq:block-gram-eigenvalue} gives
\[
    \left|
        \lambda_{\max}(G_\ell)-\|\theta\|^2
    \right|
    \le
    \frac9{10}\|\theta\|^2,
\]
which proves \eqref{eq:block-gram-signal-eigenvalue}.
We next prove the eigenvector statement. Let
\[
    v_\star
    :=
    \frac{\overline\eta}{\|\overline\eta\|}.
\]
Then
\[
    S_\ell
    =
    \lambda_\star v_\star v_\star^\top,
    \qquad
    \lambda_\star
    \ge
    \frac14\|\theta\|^2.
\]
By decreasing $c_0$ further if necessary,
\eqref{eq:block-gram-perturbation} ensures that
\[
    \|G_\ell-S_\ell\|_{\op}
    \le
    \frac14\lambda_\star.
\]
Since $S_\ell$ has spectral gap $\lambda_\star$, the
Davis--Kahan theorem implies
\[
    \sin\angle(\widehat u_\ell,v_\star)
    \le
    \frac{2\|G_\ell-S_\ell\|_{\op}}
         {\lambda_\star}
    \le\frac12.
\]
Hence
\[
    1-
    |\langle\widehat u_\ell,v_\star\rangle|^2
    \le\frac14,
\]
which proves \eqref{eq:block-gram-alignment}.
It remains to consider scales above the effective signal scale.
Assume that
\[
    t=\frac{\ell}{n}\ge C_0r_n^2.
\]
Then
\[
    \frac{a^2}{a+q}
    \le
    \frac{t}{C_0},
\]
and therefore
\[
    a^2-\frac{t}{C_0}a-\frac{tq}{C_0}\le0.
\]
Solving this quadratic inequality gives
\[
    a
    \le
    \frac{t}{C_0}
    +
    \sqrt{\frac{tq}{C_0}}.
\]
If
\[
    h(t):=t+\sqrt{qt},
\]
we thus have, for $C_0\ge1$,
\[
    a\le\frac{h(t)}{\sqrt{C_0}}.
\]
Since $t\le h(t)$,
\[
    \sqrt{at}
    \le
    \frac{h(t)}{C_0^{1/4}}.
\]
On the event of
\eqref{eq:block-gram-perturbation}, using
$\|\overline\eta\|^2/\ell\le1$,
\begin{align*}
|\lambda_{\max}(G_\ell)|
&\le
\|\theta\|^2
+
4\sigma\|\theta\|\sqrt t
+
4\sigma^2h(t)
\\
&=
\sigma^2
\left(
    a+4\sqrt{at}+4h(t)
\right)
\\
&\le
\sigma^2h(t)
\left(
    \frac1{\sqrt{C_0}}
    +
    \frac4{C_0^{1/4}}
    +4
\right).
\end{align*}
Choosing $C_0$ sufficiently large makes the quantity in parentheses
smaller than $10$. Since
\[
    h(t)
    =
    \frac{\ell+\sqrt{d\ell}}{n},
\]
we obtain
\[
    |\lambda_{\max}(G_\ell)|
    \le
    10\sigma^2
    \frac{\ell+\sqrt{d\ell}}{n},
\]
which proves \eqref{eq:block-gram-noise-regime}.
\end{proof}

\subsection{Proof of Proposition~\ref{prop:adaptive:estimation}}

Use the notation of Section~\ref{sec:concentration:lemmas}, and set
\[
    t_\ell=\ell/n,\qquad
    H(t)=t+\sqrt{qt},\qquad
    h_\ell=\sigma^2H(t_\ell).
\]
Let $c_0$ be the constant defined in Lemma \ref{lem:block-gram} and choose $0<c_+<\min\{c_0,1\}$ sufficiently small such that
\begin{equation}
    4\sigma\|\theta\|\sqrt t+4\sigma^2H(t)
    \le \frac1{20}\|\theta\|^2,
    \qquad t\le c_+r_n^2.
    \label{eq:small-scale-perturbation}
\end{equation}
This follows from $r_n^2=a^2/(a+q)$, since the left-hand side
divided by $\|\theta\|^2$ is at most
$8\sqrt{c_+}+4c_+$.
Next choose
\[
    A>c_+^{-1}+4c_+^{-1/2}+4,
\]
and then $0<c_-<c_+$ such that
$A(c_-+\sqrt{c_-})\le1/20$.

Define
\[
    \ell_-=\max\{\ell\in\mathcal L_n:
                      \ell\le n\min(c_-r_n^2,1)\}.
\]
Under $\delta\ge1/(4n)$ and $r_n^2\ge c'\delta$, taking $c'$
sufficiently large ensures that this set is nonempty and
\begin{equation}
    \ell_-\ge\frac n2\min(c_-r_n^2,1)
             \ge c n(r_n^2\wedge1),
    \qquad
    \delta\le t_{\ell_-}\le c_-r_n^2.
    \label{eq:ell-minus-size}
\end{equation}
Indeed, if $c_-r_n^2<1$, dyadic rounding gives
$t_{\ell_-}\ge c_-r_n^2/2\ge\delta$; otherwise $\ell_-=n$.

Since $\ell_-\ge n\delta$, Lemmas~\ref{lem:block-label-energy}
and~\ref{lem:block-gram} apply at every dyadic scale
$\ell\ge\ell_-$: the block-energy and Gram perturbation bounds
hold with failure probability at most $Ce^{-c\ell}$. Taking the intersection of these events over the
scales $\ell\in\mathcal L_n:\,\ell\ge\ell_-$ gives an event $\mathcal E$ on which all these bounds
hold simultaneously, with
\[
    \PP(\mathcal E^c)
    \le C\sum_{\ell\in\mathcal L_n:\,\ell\ge\ell_-}e^{-c\ell}
    \le Ce^{-c\ell_-}
    \le Ce^{-cn(r_n^2\wedge1)}.
\]
We work on $\mathcal E$ below.

\medskip\noindent
\emph{Localization of the selected scale.}
At $\ell_-$, the block-energy bound and
\eqref{eq:small-scale-perturbation}, together with Weyl's
inequality, give
\[
    \lambda_{\max}(G_{\ell_-})
    \ge\left(\frac14-\frac1{20}\right)\|\theta\|^2
    =\frac15\|\theta\|^2.
\]
Moreover, $t_{\ell_-}\le c_-r_n^2$ and the choice of $c_-$ yield $\widehat\ell\ge\ell_-$ because
\[
    Ah_{\ell_-}
    \le A(c_-+\sqrt{c_-})\|\theta\|^2
    \le\frac1{20}\|\theta\|^2\leq \lambda_{\max}(G_{\ell_-}) = \widehat\theta^{\,2}(\ell_-).
\]

For $t\ge c_+r_n^2$, writing $x=a/(a+q)$ and using $c_+ < 1$ gives
\[
    \frac{H(t)}a
    \ge c_+x+\sqrt{c_+(1-x)}
    \ge c_+.
\]
Consequently, $a\le H(t)/c_+$ and since $H(t)\geq t$,
$\sqrt{at}\le H(t)/\sqrt{c_+}$. The general perturbation bound
therefore yields, at every such scale,
\[
    \widehat\theta^{\,2}(\ell)
    =
    |\lambda_{\max}(G_\ell)|\le
    \|\theta\|^2
    +
    4\sigma\|\theta\|\sqrt{t_\ell}
    +
    4\sigma^2H(t_\ell)\le
    \sigma^2H(t_\ell)
    \left(
        \frac1{c_+}
        +
        \frac4{\sqrt{c_+}}
        +4
    \right)<
    Ah_\ell,
\]

These scales cannot satisfy the selection criterion. If
$c_+r_n^2\ge1$, the upper bound below instead holds
deterministically. Hence
\begin{equation}
    n\delta \leq \ell_-\le\widehat\ell,
    \qquad
    \delta\le\frac{\widehat\ell}{n}\le c_+r_n^2.
    \label{eq:selected-scale-range}
\end{equation}

\medskip\noindent
\emph{Signal estimation and spectral alignment.}
At the selected scale, the block-energy bound and
\eqref{eq:small-scale-perturbation} give
\[
    |\lambda_{\max}(G_{\widehat\ell})-\|\theta\|^2|
    \le\left(\frac34+\frac1{20}\right)\|\theta\|^2
    =\frac45\|\theta\|^2.
\]
In particular, $\lambda_{\max}(G_{\widehat\ell})>0$, so
$\widehat\theta^{\,2}=\lambda_{\max}(G_{\widehat\ell})$.
Together with the bound on $\PP(\mathcal E^c)$, this proves
the first assertion.

Write $\overline\eta=\overline\eta^{(\widehat\ell)}$.
Since $c_+<c_0$, the alignment conclusion of
Lemma~\ref{lem:block-gram} also gives
\begin{equation}
    \left|\left\langle
        \widehat u,\frac{\overline\eta}{\|\overline\eta\|}
    \right\rangle\right|^2\ge\frac34.
    \label{eq:selected-bucket-alignment}
\end{equation}

It remains to transfer the result to the original labels.
If $\ell_-=n$, then $\widehat\ell=n$ and
$\|\overline\eta\|^2/\widehat\ell=1$.

Suppose now that $\ell_-<n$. Let
\[
    N_{\rm tr}
    :=
    \sum_{i=1}^{n-1}
    \mathbf 1\{\eta_{i+1}\neq\eta_i\}\sim\operatorname{Bin}(n-1,\delta).
\]
Since $\ell_-<n$, we have
\(
    \ell_-
    \ge
    \frac{c_-}{2}nr_n^2.
\)
Using $r_n^2\ge c'\delta$ and taking $c'$ sufficiently large,
\(
    \EE[N_{\rm tr}]
    \le n\delta
    \le
    \frac{\ell_-}{50}.
\)
A standard Chernoff bound then gives
\begin{equation}
    \PP\left(
        N_{\rm tr}>\frac{\ell_-}{10}
    \right)
    \le
    Ce^{-c\ell_-}
    \le
    Ce^{-cn(r_n^2\wedge1)}.
\label{eq:transition-count-bound}
\end{equation}

Every block without an internal transition has squared label
average equal to one, and at most $N_{\rm tr}$ blocks contain
an internal transition. Therefore, on $\mathcal E$ and outside
the event in \eqref{eq:transition-count-bound},
\begin{equation}
    \frac{\|\overline\eta\|^2}{\widehat\ell}
    \ge 1-\frac{N_{\rm tr}}{\widehat\ell}
    \ge\frac9{10}.
    \label{eq:strong-bucket-energy}
\end{equation}
This lower bound also holds when $\ell_-=n$. By definition of the expanded vector $\widehat u_n$,
\(
    \|\widehat u_n\|=1
\)
and
\begin{align*}
\rho^2
    :=\left|
    \left\langle
        \widehat u_n,\frac{\eta}{\sqrt n}
    \right\rangle
\right|^2
=
\frac1{\widehat\ell}
\left|
    \left\langle
        \widehat u,\overline\eta
    \right\rangle
\right|^2
=
\left|
    \left\langle
        \widehat u,
        \frac{\overline\eta}{\|\overline\eta\|}
    \right\rangle
\right|^2
\frac{\|\overline\eta\|^2}{\widehat\ell}\ge\frac{3}{4}\frac{9}{10}.
\end{align*}

Choose $s\in\{-1,1\}$ such that
$\langle s\widehat u_n,\eta/\sqrt n\rangle=\rho$, and let
$M=\{i:\widehat\eta_i^{(0)}\ne s\eta_i\}$. 
For $i\in M$, $s\eta_i(\widehat u_n)_i\le0$, so
Cauchy--Schwarz yields
\[
    \rho
    \le\frac1{\sqrt n}\sum_{i\notin M}|(\widehat u_n)_i|
    \le\sqrt{1-\frac{|M|}{n}}.
\]
Note that $\rho > 1/2$ and $|M| < n/2$.
Thus $|M|/n\le1-\rho^2$ and 
\[
    \frac1n|\langle\widehat\eta^{(0)},\eta\rangle|
    = \frac{(n - |M|) - |M|}{n}
    \ge2\rho^2-1
    \ge\frac7{20}>\frac13.
\]
The total exceptional probability is at most
$Ce^{-cn(r_n^2\wedge1)}$, proving the second assertion.

\subsection{Proof of Proposition \ref{prop:delta}}
The following lemma highlights a data-driven projection retaining the effective signal.
\begin{lemma}
\label{lem:projected-signal}
Let $m=n/2$, and apply the construction of
Proposition~\ref{prop:adaptive:estimation} to the first half of the
sample. Define
\[
    \widehat a
    :=
    \frac{\widehat\theta^{\,2}}{\sigma^2},
    \qquad
    \widehat r^{\,2}
    :=
    \frac{\widehat a^2}{\widehat a+d/m}.
\]
At the selected scale $\widehat\ell$, let $\widehat u$ denote the
bucket-level leading eigenvector and set
\[
    \widehat v
    :=
    \widetilde X_{S_1}^{(\widehat\ell)}\widehat u,
    \qquad
    \widehat w
    :=
    \frac{\widehat v}{\|\widehat v\|},
\]
with an arbitrary unit vector on the null event
$\{\widehat v=0\}$.

There exist universal constants $c,C>0$ such that, whenever $\delta\geq 1/(2n)$ and
$r_n^2\ge C\delta$,
\[
    \PP(\mathcal E_{\rm proj}^c)
    \le
    C\exp\{-cn(r_n^2\wedge1)\},
\]
where \(\mathcal E_{\rm proj}\) is the event that
\[
    c r_n^2
    \le
    \widehat r^{\,2}
    \le
    C r_n^2,\ \text{and }
    c r_n^2
    \le
    \frac{\langle\theta,\widehat w\rangle^2}{\sigma^2}
    \le
    C r_n^2.
\]
\end{lemma}

\begin{proof}
Since $m=n/2$,
\[
    r_m^2
    :=
    \frac{a^2}{a+d/m}
    \asymp
    r_n^2,
    \qquad
    a=\frac{\|\theta\|^2}{\sigma^2}.
\]
The first conclusion follows immediately from
Proposition~\ref{prop:adaptive:estimation}, since on its good event
\(
    \widehat\theta^{\,2}\asymp\|\theta\|^2.
\)

We prove the second assertion. Write
\(k=m/\widehat\ell,\ \overline\eta=\overline\eta^{(\widehat\ell)}.
\)
Then
\[
    \widehat v
    =
    \theta\langle\overline\eta,\widehat u\rangle
    +
    \frac{\sigma}{\sqrt{k}}W\widehat u.
\]
On the good event constructed in the proof of
Proposition~\ref{prop:adaptive:estimation},
\(|\langle\overline\eta,\widehat u\rangle|
    \ge c\sqrt{\widehat\ell}.\)

For each dyadic scale \(\ell\), let \(W_\ell\) denote its
Gaussian noise matrix. Gaussian concentration and a union
bound give
\[
    \|W_\ell^\top\theta\|
    \le 2\|\theta\|\sqrt{\ell}
\]
simultaneously over \(\ell\ge\ell_-(m)\), outside an event of
probability at most \(C\exp\{-cm(r_m^2\wedge1)\}\), where
\[
    \ell_-(m)
    :=
    \max\left\{
        \ell\in\mathcal L_m:
        \ell\le m\min\{c_-r_m^2,1\}
    \right\}.
\]
The probability bound follows by the same argument of \eqref{eq:ell-minus-size}.
On the localization event, \(\widehat\ell\ge\ell_-(m)\),
so this bound applies at the selected scale.

Hence
\begin{align*}
|\langle\theta,\widehat v\rangle|
\ge
\|\theta\|^2
|\langle\overline\eta,\widehat u\rangle|
-
\frac{\sigma}{\sqrt{k}}
|\langle W^\top\theta,\widehat u\rangle|
\ge
c\|\theta\|^2\sqrt{\widehat\ell}
-
C\sigma\|\theta\|
\frac{\widehat\ell}{\sqrt m}.
\end{align*}
The scale localization proved in
Proposition~\ref{prop:adaptive:estimation} implies that
\(
    \frac{\widehat\ell}{m}
    \le c_+r_m^2
    \le c_+a.
\)
This yields,

\[
    \sigma\|\theta\|
    \frac{\widehat\ell}{\sqrt m}
    \le
    \sigma\|\theta\|\sqrt{c_+a\widehat\ell}
    =
    \sqrt{c_+}\|\theta\|^2\sqrt{\widehat\ell}.
\]

The sufficiently small choice of $c_+$ in the adaptive
estimation proof makes the noise term smaller than half
the signal term. The corresponding upper bound follows
from the triangle inequality and
$|\langle\overline\eta,\widehat u\rangle|
\le\sqrt{\widehat\ell}$.
Hence

Choosing $c_+$ small enough and using the triangle inequality gives
\[
    |\langle\theta,\widehat v\rangle|
    \asymp
    \|\theta\|^2\sqrt{\widehat\ell}.
\tag{31}
\label{eq:theta-vhat}
\]

On the other hand, since $\widehat u$ is also a leading eigenvector
of the uncentered Gram matrix,
\begin{align*}
    \|\widehat v\|^2
    &=
    \widehat u^\top
    \bigl(\widetilde X_{S_1}^{(\widehat\ell)}\bigr)^\top
    \widetilde X_{S_1}^{(\widehat\ell)}
    \widehat u =
    \widehat\ell
    \left[
        \lambda_{\max}(G_{\widehat\ell})
        +
        \frac{d\sigma^2}{m}
    \right].
\end{align*}
The same good event gives
\(
    \lambda_{\max}(G_{\widehat\ell})
    \asymp
    \|\theta\|^2,
\)
and consequently
\[
    \|\widehat v\|^2
    \asymp
    \widehat\ell\,\sigma^2
    \left(
        a+\frac dm
    \right).
\tag{32}
\label{eq:norm-vhat}
\]
Combining \eqref{eq:theta-vhat} and
\eqref{eq:norm-vhat},
\[
    \frac{\langle\theta,\widehat w\rangle^2}{\sigma^2}
    =
    \frac{\langle\theta,\widehat v\rangle^2}
         {\sigma^2\|\widehat v\|^2}
    \asymp
    \frac{a^2}{a+d/m}
    =
    r_m^2
    \asymp
    r_n^2.
\]
\end{proof}

For the next Lemma, we use the notation $\mu = \langle \theta, \widehat w\rangle$ and $\beta = \mu^2 /\sigma^2$.
\begin{lemma}[Sign disagreements at a fixed scale]
\label{lem:delta-fixed-scale}
Use the notation of Section~\ref{sec:adaptation}, and set
\[
    \mathcal F_1=\sigma(X_j:j\in S_1),
    \qquad
    \beta=\frac{\langle\theta,\widehat w\rangle^2}{\sigma^2}.
\]
For each $M\in\mathcal M$, the indicators
\[
    B_{i,M}
    :=\mathbf 1\{S_{2i-1}^{(M)}S_{2i}^{(M)}\le0\},
    \qquad i=1,\ldots,\frac{m}{2M},
\]
are conditionally i.i.d.\ Bernoulli given $\mathcal F_1$.
Their common conditional mean
$p_M:=\EE[B_{i,M}\mid\mathcal F_1]$ satisfies
\begin{equation}
    p_M\le2\delta M+Ce^{-c_1M\beta}
    \label{eq:pM-upper}
\end{equation}
for every $M\in\mathcal M$, and
\begin{equation}
    p_M\ge c_2\delta M
          \bigl(1-Ce^{-c_1M\beta}\bigr)
    \label{eq:pM-lower}
\end{equation}
whenever $M=1$ or $\delta M\le c_0$,
where $c_0,c_1,c_2,C>0$ are universal constants.
\end{lemma}

\begin{proof}

We first justify the conditional independence for a fixed scale $M$.
Write $\varepsilon_j=\eta_j\eta_{j+1}$. The future transition
variables $(\varepsilon_j)_{j\ge m}$ are independent of
$\mathcal F_1$ and are i.i.d., with
$\PP(\varepsilon_j=-1)=\delta$.

Let $a_i$ be the first index of the $i$th pair of adjacent blocks,
$B_{2i-1}^{(M)}\cup B_{2i}^{(M)}$. For every index $t$ in this pair,
\[
    \eta_t=\eta_{a_i}\prod_{j=a_i}^{t-1}\varepsilon_j,
    \qquad
    \zeta_t:=\eta_{a_i}g_t.
\]
Conditionally on $\mathcal F_1$ and the entire latent path,
the multipliers $\eta_{a_i}$ are fixed signs. Gaussian symmetry
therefore gives
\[
    \mathcal L\bigl((\zeta_t)_{t\in S_2}
        \mid\mathcal F_1,\eta_{1:n}\bigr)
    =\mathcal N(0,I_m).
\]
Since this conditional law does not depend on the latent path,
the rotated noises are also independent of that path
conditionally on $\mathcal F_1$.

Moreover,
\[
    \eta_{a_i}Z_t
    =\mu\prod_{j=a_i}^{t-1}\varepsilon_j+\sigma\zeta_t,
\]
with the empty product equal to one. Multiplying both block sums
by $\eta_{a_i}$ leaves their product unchanged. Hence $B_{i,M}$
depends only on the transitions
$\varepsilon_{a_i},\ldots,\varepsilon_{a_i+2M-2}$
internal to its pair and the noises
$\zeta_{a_i},\ldots,\zeta_{a_i+2M-1}$; it no longer depends
on the starting label $\eta_{a_i}$.
Different pairs use disjoint sets of these independent inputs,
and each indicator is the same function of inputs with the
same conditional distribution. Thus the variables $(B_{i,M})_i$
are conditionally i.i.d.\ Bernoulli given $\mathcal F_1$.

For the upper bound, first note that the probability that at least one transition occurs among the
$2M-1$ relevant edges is at most $2\delta M$. On the complementary event, both latent blocks have the same sign,
and each block sum has conditional distribution
\[
    \pm M\mu+\sigma\sqrt M\,G,
    \qquad G\sim \mathcal{N}(0,1).
\]
Therefore the probability that the two observed block sums have
opposite signs is bounded by
\[
    2\Phi^c(\sqrt{M\beta})
    \le
    C e^{-c_1M\beta}.
\]
 Hence
\eqref{eq:pM-upper} follows.

For the lower bound, assume first that $M\ge4$.
Consider the event that the unique transition within the pair
occurs after position $r\in\{3M/4,\ldots,M-1\}$, numbering
positions from the start of the pair. Its probability is
\[
    \frac M4\delta(1-\delta)^{2M-2}
    \ge\frac{\delta M}{8},
    \qquad \delta M\le\frac1{16}.
\]
On this event, the two latent block sums have opposite signs
and magnitudes at least $M/2$ and $M$. Conditionally on any
such latent configuration, a Gaussian tail bound gives
\[
    \PP\bigl(S_{2i-1}^{(M)}S_{2i}^{(M)}\le0
             \mid\mathcal F_1,\eta_{1:n}\bigr)
    \ge 1-2\Phi^c\!\left(\frac{\sqrt{M\beta}}2\right).
\]
For $M=2$, instead require a transition at the block boundary
and no other transition within the pair. Its probability is
$\delta(1-\delta)^2\ge\delta/2=\delta M/4$
when $\delta M\le1/16$.
For $M=1$, the transition event has probability $\delta$,
without any additional restriction on $\delta$.
In both cases, the same conditional Gaussian bound applies.
Consequently, for every admissible dyadic $M$ with either
$M=1$ or $\delta M\le1/16$,
\[
    p_M\ge\frac{\delta M}{8}
    \left[
        1-2\Phi^c\!\left(\frac{\sqrt{M\beta}}2\right)
    \right].
\]
This proves \eqref{eq:pM-lower}, with $c_0=1/16$ and
suitable universal constants. Moreover, choosing a universal
$B_*>0$ such that
$2\Phi^c(\sqrt{B_*}/2)\le1/2$ gives
\begin{equation}
    p_M\ge\frac{\delta M}{16},
    \qquad M\beta\ge B_*,
    \label{eq:pM-numerical-lower}
\end{equation}
under the same restrictions on $M$ and $\delta$.
\end{proof}

Use the notation of Section~\ref{sec:adaptation}, and write
\[
    s=r_n^2(\theta),\qquad
    h=\widehat r^{\,2},\qquad
    \beta=\frac{\langle\theta,\widehat w\rangle^2}{\sigma^2}.
\]
Let $\mathcal F_1=\sigma(X_j:j\in S_1)$. By
Lemma~\ref{lem:projected-signal}, there are universal constants
$a_*,b_*>0$ such that the $\mathcal F_1$-measurable event
\[
    \mathcal E_{\rm proj}
    :=\{a_*s\le h\le b_*s,\ a_*s\le\beta\le b_*s\}
\]
has probability at least $1-Ce^{-cn(s\wedge1)}$.
We work conditionally on $\mathcal F_1$, on this event.

For $M\in\mathcal M$, set $N_M=m/(2M)$.
Lemma~\ref{lem:delta-fixed-scale} implies that $N_MD_M$
is conditionally binomial with parameters $N_M,p_M$.

Choose $\gamma>0$ sufficiently small that
$c_1a_*\ge2\gamma b_*$, where $c_1$ is from
Lemma~\ref{lem:delta-fixed-scale}.
Fix $K=16$ and $\kappa=1/128$, and choose $B$ sufficiently
large that, with $\rho=e^{-\gamma B}$,
\[
    C_1\rho\le\frac18,\qquad
    C_1K\rho\le1,\qquad
    \frac{a_*}{b_*}B\ge B_*,
\]
where $C_1$ is the coefficient in \eqref{eq:pM-upper}. Here $B_*$ is the universal constant in
\eqref{eq:pM-numerical-lower}, established in the proof of
Lemma~\ref{lem:delta-fixed-scale}.
All these constants are fixed before choosing the separation
constant in $s\ge C\delta$.

Since $\delta\ge1/(2n)$,
\[
    (n/4)h\ge a_*C/8\ge B
\]
for sufficiently large $C$. Thus the starting scale $M_0$
and the restricted grid $\mathcal M(\widehat r)$ defined in
the main text are nonempty on $\mathcal E_{\rm proj}$.
For $M\in\mathcal M(\widehat r)$, write
$\tau_M=e^{-\gamma Mh}$. The preceding choices give
\begin{equation}
    e^{-c_1M\beta}\le\tau_M^2,\qquad
    \tau_M\le\rho,\qquad M\beta\ge B_*.
    \label{eq:noise-below-tau-square}
\end{equation}

\medskip\noindent
\emph{Localization of the crossing window.}
Define
\[
    M_-=\max\{M\in\mathcal M(\widehat r):
                         \tau_M\ge K\delta M\},
    \qquad
    M_+=\min\{M\in\mathcal M(\widehat r):
                         \tau_M\le\kappa\delta M\}.
\]
If the first set is empty, set $M_-=M_0/2$.
These scales are deterministic under conditioning on
$\mathcal F_1$.

To verify that \(M_+\) exists, let \(M_{\max}=n/4\) and
\(z=\delta M_{\max}\ge1/8\). Then
\[
    \tau_{M_{\max}}
    \le e^{-\gamma a_*sM_{\max}}
    \le e^{-\gamma a_*Cz}
    \le\kappa z
\]
for a sufficiently large separation constant \(C\).
Thus the defining set for \(M_+\) is nonempty.

Since $\tau_{2M}=\tau_M^2$ and $\tau_M\le\rho<1$,
\begin{equation}
    \frac{\tau_{2M}}{2\delta M}
    =\frac{\tau_M}{\delta M}\frac{\tau_M}{2}
    \le\frac{\rho}{2}\frac{\tau_M}{\delta M}.
    \label{eq:ratio-geometric}
\end{equation}
Thus the ratio $\tau_M/(\delta M)$ decreases by a factor
at most $\rho/2$ at each doubling of $M$.
Let $M_{\mathrm{start}}$ be the first grid point after $M_-$.
By definition of $M_-$, its ratio is less than $K$.
This also holds when the defining set is empty, since then
$M_-=M_0/2$ and $M_{\mathrm{start}}=M_0$.

After $j$ doublings, provided the resulting scale remains
in the grid,
\[
    \frac{\tau_{2^jM_{\mathrm{start}}}}
         {\delta\,2^jM_{\mathrm{start}}}
    <K\left(\frac{\rho}{2}\right)^j.
\]
Since $M_+$ is the first scale where this ratio is at most
$\kappa$, and $M_+$ exists, the window $M_-<M\le M_+$
contains at most
\[
    J:=1+\left\lceil
        \frac{\log(K/\kappa)}{\log(2/\rho)}
    \right\rceil
\]
grid points. This bound is universal because $K$, $\kappa$,
and $\rho$ are fixed universal constants.

Put \(x=s/\delta\) and \(L=1+\log x\).
If \(M_+=M_{0}>1\), minimality of \(M_{0}\) gives
\[
    M_+h<2B,
    \qquad M_+\le\frac{2B}{a_*s}.
\]
If \(M_+>M_{0}\), set \(M'=M_+/2\) and \(y=sM'\).
Minimality of \(M_+\) implies
\[
    e^{-\gamma M'h}>\kappa\delta M',
    \qquad
    \gamma a_*y<\log\frac{x}{\kappa y}.
\]
For \(y\ge1\), the last expression is at most
\(\log(x/\kappa)\); for \(y<1\), \(y\) is already bounded.
Hence \(y\lesssim L\). Combining the two cases gives,
whenever \(M_+>1\),
\begin{equation}
    M_+\le C_*\frac{L}{s},
    \qquad
    \delta M_+\le
    C_*\frac{1+\log x}{x}
    \le\min\{c_0,1/16\},
    \label{eq:Mplus-persistence}
\end{equation}
after increasing the separation constant.
Here \(C_*\) is universal and \(c_0\) is the constant in
Lemma~\ref{lem:delta-fixed-scale}.

\emph{Excluding early crossings.}
For $M\in\mathcal M(\widehat r)$ with $M\le M_-$,
\eqref{eq:pM-upper}, \eqref{eq:noise-below-tau-square},
and the choices of $K$ and $B$ give
\begin{equation}
    p_M
    \le2\delta M+C_1\tau_M^2
    \le(2/K+C_1\rho)\tau_M
    \le\frac{\tau_M}{4}.
    \label{eq:pM-below-tau}
\end{equation}

Since $N_MD_M$ is binomial with parameters $N_M$ and $p_M$,
a standard Chernoff inequality gives
\[
    \PP(D_M\ge\tau_M\mid S_1)
    \le
    \exp(-cN_M\tau_M).
\tag{46}
\label{eq:false-crossing}
\]
If the set defining \(M_-\) is nonempty, then at \(M=M_-\),
\[
    N_{M_-}\tau_{M_-}
    =
    \frac{m}{2M_-}\tau_{M_-}
    \ge
    \frac{Km\delta}{2}
    \ge
    c n\delta.
\tag{47}
\label{eq:early-exponent}
\]
Furthermore, when passing from $M$ to $M/2$,
\[
    N_{M/2}\tau_{M/2}
    =
    2e^{\gamma M\widehat r^{\,2}/2}
    N_M\tau_M
    \ge
    2N_M\tau_M.
\]
Hence the exponents in
\eqref{eq:false-crossing} grow at least geometrically as one moves
backwards along the grid. Summing these bounds over the early scales and using
$n\delta\ge1/2$, we obtain

\[
    \PP\left(
        \exists M\in\mathcal M(\widehat r):
        M\le M_-,
        \ D_M\ge\tau_M
        \,\middle|\,S_1
    \right)
    \le
    C e^{-cn\delta}.
\tag{48}
\label{eq:no-early-crossing}
\]
If the set defining \(M_-\) is empty, there are no
pre-critical scales to control.

\medskip\noindent
\emph{Accuracy at the selected scale.}
For \(M\in\mathcal M(\widehat r)\) with \(M_-<M\le M_+\), we have
\(\tau_M<K\delta M\), \(M\beta\ge B_*\), and either
\(M=1\) or \(\delta M\le1/16\).
Therefore, by \eqref{eq:pM-numerical-lower},
\eqref{eq:pM-upper}, and
\eqref{eq:noise-below-tau-square},
\begin{equation}
    \frac{\delta M}{16}\le p_M
    \le2\delta M+C_1\tau_M^2
    \le(2+C_1K\rho)\delta M
    \le3\delta M.
    \label{eq:pM-transition-regime}
\end{equation}
Since \(N_MD_M\) is conditionally binomial,
\[
    \PP\left(
        |D_M-p_M|>\frac12p_M
        \,\middle|\,\mathcal F_1
    \right)
    \le2e^{-cN_Mp_M}
    \le2e^{-cn\delta},
\]
where \(N_Mp_M\ge n\delta/64\).
There are at most \(J\) scales in this window.
A union bound thus gives, with conditional probability
at least \(1-Ce^{-cn\delta}\),
\begin{equation}
    \frac{\delta M}{32}\le D_M\le\frac92\delta M
    \quad\text{for every }M\in\mathcal M(\widehat r)
    \text{ with }M_-<M\le M_+.
    \label{eq:DM-transition-regime}
\end{equation}

On the event in \eqref{eq:DM-transition-regime},
\[
    D_{M_+}\ge\frac{\delta M_+}{32}
    >\kappa\delta M_+\ge\tau_{M_+}.
\]
Thus the procedure stops by \(M_+\).
Combining this with \eqref{eq:no-early-crossing}, outside
an event of conditional probability at most \(Ce^{-cn\delta}\),
\[
    M_-<\widehat M\le M_+,
    \qquad
    \frac{\delta}{32}
    \le\frac{D_{\widehat M}}{\widehat M}
    \le\frac92\delta.
\]
Since \(1/n\le2\delta\) and \(\delta\le1/2\), truncation
to \([1/n,1/2]\) preserves
\[
    \frac{\delta}{32}\le\widehat\delta\le\frac92\delta.
\]
Both ratios are therefore strictly smaller than \(100\). Finally, since $s\ge C\delta$ and $\delta\le1/2$,
\[
    \PP(\mathcal E_{\rm proj}^c)
    \le Ce^{-cn(s\wedge1)}
    \le Ce^{-cn\delta}.
\]

\subsection{Proof of Theorem \ref{upper:minimax}}

We work under the normalization $\sigma=1$ and write
\[
    a:=\|\theta\|^2,
    \qquad
    r_n^2:=\frac{a^2}{a+d/n}.
\]
For $u\in\{-1,1\}^n$, let
\[
    e(u,\eta)
    :=
    \frac1n\sum_{i=1}^n
    \mathbf 1\{u_i\neq\eta_i\}.
\]
As usual, the global sign of an estimator is chosen so that its
correlation with $\eta$ is nonnegative.

We denote by $\mathcal E_{\rm ad}$ the event on which the adaptive
estimators of Section~\ref{sec:adaptation} satisfy
\begin{equation}
\label{eq:adaptive-event-proof}
    \frac{\delta}{100}\leq\widehat\delta\leq100\delta,
    \qquad
    c a\leq\widehat a_{\rm ad}\leq Ca,
    \qquad
    c r_n^2\leq\widehat r_{\rm ad}^{\,2}\leq Cr_n^2,
\end{equation}
and
\begin{equation}
\label{eq:adaptive-init-correlation}
    \frac1n
    \left|
        \langle
            \widehat\eta_{\rm ad}^{(0)},\eta
        \rangle
    \right|
    \geq c_0
\end{equation}
for some universal $c_0>0$. By the results of
Section~\ref{sec:adaptation},
\begin{equation}
\label{eq:adaptive-event-prob}
    \PP(\mathcal E_{\rm ad}^c)
    \leq
    Ce^{-cn\delta}.
\end{equation}
On the complement of this event we use the trivial loss bound. Since
the normalized Hamming loss is bounded by one, this contributes at most
$Ce^{-cn\delta}$.
We now work on $\mathcal E_{\rm ad}$.

\medskip
\noindent
{\bf Block hollowing and spectral control.}
Let $\mathcal P=(B_1,\ldots,B_N)$ be a partition of $[n]$ and let
$P_B$ denote the coordinate projection onto a block $B$. Define
\[
    H_{\mathcal P}(A)
    :=
    A-\sum_{B\in\mathcal P}P_BAP_B.
\]
Thus $H_{\mathcal P}(A)$ removes the entries whose two indices belong
to the same block.

We shall use repeatedly the elementary inequality
\begin{equation}
\label{eq:block-hollow-op-proof}
    \|H_{\mathcal P}(A)\|_{\op}
    \leq 2\|A\|_{\op}
\end{equation}
for every symmetric matrix $A$. Indeed,
\[
    \left\|
        \sum_{B\in\mathcal P}P_BAP_B
    \right\|_{\op}
    =
    \max_{B\in\mathcal P}
    \|P_BAP_B\|_{\op}
    \leq
    \|A\|_{\op}.
\]
Conditionally on $\eta$,
\[
    X=\theta\eta^\top+W,
\]
 the decomposition used in \cite{Ndaoud18} gives
\begin{equation}
\label{eq:hollowed-iid-decomposition}
    \frac1nH(X^\top X)
    =
    \frac an\eta\eta^\top+Z_n,
\end{equation}
where $Z_n$ is a centered perturbation up to the deterministic
diagonal correction absorbed into its definition. The operator-norm
bounds of \cite{Ndaoud18} imply, on an event
$\mathcal E_{\rm op}$,
\begin{equation}
\label{eq:Zn-bound}
    \|Z_n\|_{\op}
    \leq
    C\frac{a}{r_n}
    +
    o(a).
\end{equation}
The precise negligible term will only be relevant when $r_n\to\infty$.

Applying \eqref{eq:block-hollow-op-proof} to the centered noise
matrix gives the same control, up to a universal factor, after block
hollowing. More precisely, for every partition used by the
procedure,
\begin{equation}
\label{eq:block-Z-bound}
    \|Z_{n,\mathcal P}\|_{\op}
    \leq
    C\frac{a}{r_n}+o(a),
\end{equation}
where
\[
    \frac1nH_{\mathcal P}(X^\top X)
    =
    \frac anH_{\mathcal P}(\eta\eta^\top)
    +
    Z_{n,\mathcal P}.
\]
Only two shifted partitions are used, so the preceding bounds hold
simultaneously on the same event after modifying the universal
constants.

For an index $i$ and the block $B_i$ selected in the construction,
the oracle leave-block-out direction is
\[
    \theta_i^\star
    :=
    \frac1{n-|B_i|}
    \sum_{j\notin B_i}\eta_jX_j.
\]
Conditionally on $\eta$,
\begin{equation}
\label{eq:oracle-leave-block}
    \theta_i^\star
    =
    \theta+
    \frac{Z_i}{\sqrt{n-|B_i|}},
    \qquad
    Z_i\sim \mathcal{N}(0,I_d),
\end{equation}
and $\theta_i^\star$ is independent of
$\{X_j:j\in B_i\}$.

This independence will be used to analyze the local oracle rule.

\medskip
\noindent
{\bf The two initializations and the two signal estimators.}
The constant-factor estimators
$(\widehat a_{\rm ad},\widehat r_{\rm ad})$ are used to distinguish the
bounded-signal regime and to choose the local window length up to
universal constants. In the diverging-signal regime, the sharp estimator
$\widehat r_{\rm sp}$ is used only to decide when it is safe to switch
to the one-step rule $k=1$.

When $\widehat r_{\rm ad}\leq R_0$, the algorithm starts from
$\widehat\eta_{\rm ad}^{(0)}$.

When $\widehat r_{\rm ad}>R_0$, it starts instead from
\[
    \widehat\eta_{\rm sp}^{(0)}
    =
    \sign(\widehat v),
\]
where $\widehat v$ is a leading eigenvector of
$n^{-1}H(X^\top X)$. Since the result of \cite{Ndaoud18} is uniform
over $\eta\in\{-1,1\}^n$, its proof applies conditionally on the
Markovian labels. Hence, after fixing the global sign,
\begin{equation}
\label{eq:spectral-initialization-proof}
    e(\widehat\eta_{\rm sp}^{(0)},\eta)
    \leq
    \frac{C}{r_n^2}+\nu_n,
    \qquad
    \nu_n=o(1),
\end{equation}
in the regime $r_n\to\infty$.

We also recall the definition
\[
    \widehat a_{\rm sp}
    :=
    \left[
        \lambda_{\max}
        \left(
            \frac1nH(X^\top X)
        \right)
    \right]_+.
\]
Since the rank-one matrix
$a\eta\eta^\top/n$ has leading eigenvalue $a$,
Weyl's inequality and \eqref{eq:hollowed-iid-decomposition} yield
\[
    |\widehat a_{\rm sp}-a|
    \leq
    \|Z_n\|_{\op}.
\]
Consequently, whenever $r_n\to\infty$,
\begin{equation}
\label{eq:a-sp-consistency-proof}
    \frac{\widehat a_{\rm sp}}{a}
    =
    1+o_{\PP}(1).
\end{equation}
Defining
\[
    \widehat r_{\rm sp}^{\,2}
    :=
    \frac{\widehat a_{\rm sp}^{\,2}}
         {\widehat a_{\rm sp}+d/n},
\]
we also have
\begin{equation}
\label{eq:r-sp-consistency-proof}
    \frac{\widehat r_{\rm sp}}{r_n}
    =
    1+o_{\PP}(1).
\end{equation}
Indeed, if $\widehat a_{\rm sp}=a(1+u_n)$ with
$u_n=o_{\PP}(1)$, then
\[
    \frac{\widehat r_{\rm sp}^{\,2}}{r_n^2}
    =
    (1+u_n)^2
    \frac{a+d/n}{a(1+u_n)+d/n}
    =
    1+o_{\PP}(1).
\]

The constant-factor estimate of $\delta$ is already sufficient for
all the tuning steps. Indeed, on $\mathcal E_{\rm ad}$,
\begin{equation}
\label{eq:delta-log-sandwich-proof}
    \log\frac1\delta
    \leq
    \log\frac{100}{\widehat\delta}
    \leq
    \log\frac{10000}{\delta}.
\end{equation}
Moreover, with
\[
    L_\delta
    :=
    \log\frac{1-\delta}{\delta},
    \qquad
    \widehat L_\delta
    :=
    \log\frac{1-\widehat\delta}{\widehat\delta},
\]
we have the uniform additive bound
\begin{equation}
\label{eq:Ldelta-additive-proof}
    \bigl|
        \widehat L_\delta-L_\delta
    \bigr|
    \leq
    \log 200.
\end{equation}
Indeed, the ratio $\delta/\widehat\delta$ is between $1/100$ and $100$,
while both $1-\delta$ and $1-\widehat\delta$ belong to $[1/2,1]$.

Combining \eqref{eq:a-sp-consistency-proof},
\eqref{eq:r-sp-consistency-proof}, and
\eqref{eq:Ldelta-additive-proof}, we obtain
\begin{equation}
\label{eq:kappa-consistency-proof}
    \widehat\kappa
    :=
    \frac{\widehat a_{\rm sp}}
         {\widehat r_{\rm sp}^{\,2}}
    \widehat L_\delta
    =
    \frac{a}{r_n^2}
    \left\{
        L_\delta\bigl(1+o_{\PP}(1)\bigr)
        +
        O(1)
    \right\}.
\end{equation}
Thus sharp estimation of $\delta$ is not needed: its constant-factor
uncertainty produces only an additive $O(1)$ perturbation of the
log-likelihood correction.

When $r_n$ stays bounded, we use instead the adaptive quantities
$(\widehat a_{\rm ad},\widehat r_{\rm ad})$. Only constant-factor
control is required in that regime.

\medskip
\noindent
{\bf Deterministic stability in the sharp regime.}
We first assume $r_n\to\infty$. On $\mathcal E_{\rm ad}$, the
data-driven choice of the initialization eventually selects the
spectral initialization, so \eqref{eq:spectral-initialization-proof}
holds.

Fix one of the block partitions and put
\[
    G_{\mathcal P}
    :=
    \frac1nH_{\mathcal P}(X^\top X).
\]
For
\[
    x^{(m)}
    :=
    \widehat\eta^{(m)}-\eta
\]
we have
\begin{equation}
\label{eq:iterative-decomp-proof}
    G_{\mathcal P}\widehat\eta^{(m)}
    =
    G_{\mathcal P}\eta
    +
    Z_{n,\mathcal P}x^{(m)}
    +
    \frac an
    H_{\mathcal P}(\eta\eta^\top)x^{(m)}.
\end{equation}

If $i$ belongs to a block $B$, then
\[
    \left[
        \frac an
        H_{\mathcal P}(\eta\eta^\top)x^{(m)}
    \right]_i
    =
    \frac an\eta_i
    \sum_{j\notin B}
        \eta_j
        \bigl(
            \widehat\eta_j^{(m)}-\eta_j
        \bigr),
\]
so that
\begin{equation}
\label{eq:signal-perturbation-proof}
    \left|
        \left[
        \frac an
        H_{\mathcal P}(\eta\eta^\top)x^{(m)}
        \right]_i
    \right|
    \leq
    2a\,e_m,
    \qquad
    e_m:=e(\widehat\eta^{(m)},\eta).
\end{equation}

The central score and the two pilot scores are either coordinates of
$G_{\mathcal P}\widehat\eta^{(m)}$ or local averages of such
coordinates. Let $A_1,\ldots,A_q$ denote the corresponding linear
operators. Here $q$ is a universal constant and
\[
    \|A_j\|_{\op}\leq C.
\]
By \eqref{eq:block-Z-bound},
\[
\begin{split}
    \|A_jZ_{n,\mathcal P}x^{(m)}\|_2^2
    \leq
    C\|Z_{n,\mathcal P}\|_{\op}^2
    \|x^{(m)}\|_2^2 \leq
    C\frac{a^2}{r_n^2}
    \|x^{(m)}\|_2^2.
\end{split}
\]
Since
\(
    \|x^{(m)}\|_2^2=4ne_m,
\)
we obtain
\begin{equation}
\label{eq:l2-perturb-proof}
    \|A_jZ_{n,\mathcal P}x^{(m)}\|_2^2
    \leq
    C\frac{a^2}{r_n^2}ne_m.
\end{equation}

Let
\[
    \tau_n
    :=
    K\frac{a}{r_n}
    +
    C\frac{a}{r_n^2},
\]
where $K$ is a sufficiently large universal constant. The second term
absorbs the deterministic perturbation
\eqref{eq:signal-perturbation-proof} inside the basin
\[
    e_m
    \leq
    \frac{C_1}{r_n^2}+\nu_n.
\]

For each interior index $i\in\llbracket k+1,n-k\rrbracket$, let
$\mathcal A_i$ denote the event that every oracle score involved in the
local decision at $i$ is separated from its corresponding decision
boundary by at least $\tau_n$, with the perturbation of the Markov
threshold described in \eqref{eq:kappa-consistency-proof}. For the
boundary indices, only the central leave-one-out score is used; we
define the corresponding stability event $\mathcal A_i$ analogously.

On $\mathcal A_i$, a disagreement between the oracle update and the
actual update can therefore occur only if one of the noise
perturbations in \eqref{eq:l2-perturb-proof} exceeds a constant
multiple of $a/r_n$. Hence
\[
\begin{split}
    \sum_{i=1}^n
    \mathbf 1\left\{
        |(A_jZ_{n,\mathcal P}x^{(m)})_i|
        >
        c\frac{a}{r_n}
    \right\}\leq
    C\frac{r_n^2}{a^2}
    \|A_jZ_{n,\mathcal P}x^{(m)}\|_2^2\leq
    \frac1{8q}\,ne_m
\end{split}
\]
after choosing $K$ sufficiently large. Summing over the finitely many
scores gives
\begin{equation}
\label{eq:contraction-proof}
    e_{m+1}
    \leq
    \frac14e_m
    +
    \frac{C}{n}
    \sum_{i=1}^n
    \mathbf 1_{\mathcal A_i^c}.
\end{equation}

Exactly as in the proof of the iterative result of
\cite{Ndaoud18}, the spectral initialization
\eqref{eq:spectral-initialization-proof} and
\eqref{eq:contraction-proof} show by induction that the iterates stay
inside the same basin. Iterating, for $m\geq4\log n$,
\begin{equation}
\label{eq:iteration-final-deterministic}
    e_m
    \leq
    n^{-2}
    +
    \frac{C}{n}
    \sum_{i=1}^n
    \mathbf 1_{\mathcal A_i^c}.
\end{equation}

Let
\[
    I_k:=\llbracket k+1,n-k\rrbracket,
    \qquad
    \partial I_k:=[n]\setminus I_k,
\]
so that $|\partial I_k|\leq2k$. As in \cite{Ndaoud18}, introduce the
event
\[
    \mathcal C
    :=
    \left\{
        \frac1n
        \sum_{i=1}^n
        \mathbf 1_{\mathcal A_i^c}
        \leq
        \frac{c}{r_n^2}
    \right\}.
\]
Markov's inequality gives
\[
    \PP(\mathcal C^c)
    \leq
    Cr_n^2
    \left[
        \sup_{i\in I_k}\PP(\mathcal A_i^c)
        +
        \frac{k}{n}
        \sup_{i\in\partial I_k}\PP(\mathcal A_i^c)
    \right].
\]
The boundary rule is the usual central leave-one-out classifier, so
\eqref{eq:q1-proof} below also applies to
$\sup_{i\in\partial I_k}\PP(\mathcal A_i^c)$. Combining this observation
with \eqref{eq:iteration-final-deterministic} yields the master estimate
\begin{equation}
\label{eq:master-proof-new}
    \EE[e_m]
    \leq
    Cr_n^2
    \sup_{i\in I_k}\PP(\mathcal A_i^c)
    +
    C\frac{k}{n}
    \Phi^c\!\left(
        r_n(1-\varepsilon_n)
    \right)
    +
    Ce^{-cn\delta}.
\end{equation}
Since $\delta\geq1/n$, the boundary term will always be absorbed by the
corresponding local-window or strong-signal rate.

It remains to bound the oracle probability for interior indices.

\medskip
\noindent
{\bf Sharp supervised and pilot probabilities.}
Let $B_i$ be the block used at site $i$ and put $L_i=|B_i|$.
In the sharp regimes,
\[
    k=O(\log n),
    \qquad
    L_i=O(k)=o(n).
\]
From \eqref{eq:oracle-leave-block},
\[
    \theta_i^\star
    =
    \theta+\frac{Z_i}{\sqrt{n-L_i}}.
\]
Therefore the central oracle score is the same supervised score as in
\cite{Ndaoud18}, with $n$ replaced by $n-L_i$. Since
\[
    \frac{a^2}{a+d/(n-L_i)}
    =
    r_n^2(1+o(1)),
\]
the sharp supervised analysis of \cite{Ndaoud18} gives, for some
$\varepsilon_n\to0$,
\begin{equation}
\label{eq:q1-proof}
    r_n^2q_1
    \leq
    C\Phi^c\!\left(
        r_n(1-\varepsilon_n)
    \right),
\end{equation}
where $q_1$ denotes the probability that the central oracle score lies
on the wrong side of its decision boundary or within the margin
$\tau_n$.

If a pilot window of length $k$ contains no transition, then,
after multiplication by the true label,
\[
    \frac1k\sum_{j\in I} \eta_jX_j
    =
    \theta+\frac{Z}{\sqrt{k}},
\]
independently of $\theta_i^\star$. Its effective squared SNR is
\begin{equation}
\label{eq:pilot-snr-proof}
    R_{n,k}^2
    =
    \frac{k a^2}
    {a+(ka+d)/(n-L_i)}
    =
    k r_n^2(1+o(1)).
\end{equation}
Consequently,
\begin{equation}
\label{eq:qk-proof}
    q_k
    \leq
    C\exp(-c k r_n^2),
\end{equation}
where $q_k$ includes the same stability margin.

\medskip
\noindent
{\bf The local-window regime.}
Assume first that $r_n\to\infty$ and that the data-driven rule selects
$k>1$. Since $\widehat r_{\rm ad}>R_0$ eventually in this regime, the
definition of the estimator implies
\[
    \widehat r_{\rm sp}^{\,2}
    <
    2\log\left(
        \frac{100}{\widehat\delta}
    \right).
\]
Together with \eqref{eq:r-sp-consistency-proof} and
\eqref{eq:delta-log-sandwich-proof}, this gives
\begin{equation}
\label{eq:local-window-upper-proof}
    r_n^2
    \leq
    2\log\frac1\delta
    +
    o(r_n^2).
\end{equation}
Indeed, if $r_n\to\infty$ while the local-window rule is selected, then
necessarily $\delta\to0$, and the additive constant
$2\log 10000$ is $o(r_n^2)$.

On $\mathcal E_{\rm ad}$, constant-factor estimation of
$r_n^2$ and $\delta$ yields
\[
    k
    \asymp
    \frac1{r_n^2}
    \log\frac{r_n^2}{\delta}.
\]
By increasing the universal constant in the definition of $k$ if
necessary, \eqref{eq:qk-proof} gives
\begin{equation}
\label{eq:qk-delta-over-r}
    q_k
    \leq
    C\frac{\delta}{r_n^2}.
\end{equation}

Let $N_i$ denote the number of transitions of the Markov chain in the
two pilot windows surrounding an interior index $i$. Since the transition
indicators are i.i.d.\ Bernoulli$(\delta)$,
\[
    \PP(N_i=1)
    \leq Ck\delta,
    \qquad
    \PP(N_i\geq2)
    \leq C(k\delta)^2.
\]

We use the same $0/1/\geq2$ transition decomposition as in the
one-dimensional oracle analysis. On $\{N_i=0\}$, both pilot blocks
have the correct common label. Writing
\[
    q_+
    :=
    \PP\left(
        \text{the Markov-corrected central oracle score is wrong
        or within its stability margin}
    \right),
\]
an error can only occur through the corrected central score, through
two pilot failures, or through a pilot failure together with a central
failure. Thus
\[
    \PP(\mathcal A_i^c,N_i=0)
    \leq
    q_+
    +
    Cq_kq_1
    +
    Cq_k^2.
\]
On $\{N_i=1\}$, at least one pilot window is clean and carries the true
label at $i$, so
\[
    \PP(\mathcal A_i^c,N_i=1)
    \leq
    Ck\delta(q_1+q_k).
\]
Finally,
\[
    \PP(\mathcal A_i^c,N_i\geq2)
    \leq
    C(k\delta)^2.
\]
Therefore
\begin{equation}
\label{eq:intermediate-decomp-proof}
    \PP(\mathcal A_i^c)
    \leq
    q_+
    +
    Cq_kq_1
    +
    Cq_k^2
    +
    Ck\delta(q_1+q_k)
    +
    C(k\delta)^2.
\end{equation}

The Gaussian shift inequality gives
\begin{equation}
\label{eq:qplus-intermediate-proof}
    q_+
    \leq
    \delta\,
    \exp\{o(r_n^2)\}\,q_1.
\end{equation}
Moreover, \eqref{eq:local-window-upper-proof} implies
\begin{equation}
\label{eq:kdelta-exp-proof}
    k\delta
    \leq
    \exp\left(
        -\frac{r_n^2}{2}
        +
        o(r_n^2)
    \right)
\end{equation}
up to polynomial factors in $r_n$. Combining
\eqref{eq:q1-proof},
\eqref{eq:qk-delta-over-r},
\eqref{eq:intermediate-decomp-proof}, and
\eqref{eq:kdelta-exp-proof}, and using Mills' inequalities to absorb
the polynomial and $\exp\{o(r_n^2)\}$ factors into a vanishing
deterioration of the Gaussian argument, we obtain
\begin{equation}
\label{eq:intermediate-local-proof}
    r_n^2
    \sup_{i\in I_k}
    \PP(\mathcal A_i^c)
    \leq
    Ck\delta\,
    \Phi^c\!\left(
        r_n(1-\varepsilon_n)
    \right)
\end{equation}
for some $\varepsilon_n\to0$.

Since $\delta\geq1/n$, the boundary term in
\eqref{eq:master-proof-new} is bounded by the same right-hand side.
Using
\[
    k\delta
    \asymp
    \frac{\delta}{r_n^2}
    \log\frac{r_n^2}{\delta}
\]
in \eqref{eq:master-proof-new} gives
\begin{equation}
\label{eq:intermediate-result-proof}
    \EE[e_m]
    \leq
    Ce^{-cn\delta}
    +
    C
    \frac{\delta\log(r_n^2/\delta)}{r_n^2}
    \Phi^c\!\left(
        r_n(1-\varepsilon_n)
    \right).
\end{equation}

The local-window rule may still be selected when
$r_n^2$ is slightly larger than $2\log(1/\delta)$. In that case,
\eqref{eq:local-window-upper-proof} implies
\[
    r_n^2
    =
    2\log(1/\delta)\bigl(1+o(1)\bigr),
\]
and hence
\[
    \frac1{r_n^2}
    \log\frac{r_n^2}{\delta}
    =
    O(1).
\]
Thus \eqref{eq:intermediate-result-proof} already gives the strong-signal
rate
\[
    \EE[e_m]
    \leq
    Ce^{-cn\delta}
    +
    C\delta
    \Phi^c\!\left(
        r_n(1-\varepsilon_n)
    \right)
\]
throughout this transition region.

\medskip
\noindent
{\bf The one step strong signal regime.}
Still assuming $r_n\to\infty$, suppose that the data-driven rule selects
$k=1$. By construction,
\begin{equation}
\label{eq:k-one-selection-proof}
    \widehat r_{\rm sp}^{\,2}
    \geq
    2\log\left(
        \frac{100}{\widehat\delta}
    \right).
\end{equation}
Together with \eqref{eq:r-sp-consistency-proof} and the first inequality
in \eqref{eq:delta-log-sandwich-proof}, this implies
\begin{equation}
\label{eq:strong-regime-lower-proof}
    r_n^2
    \geq
    2\log\frac1\delta
    -
    o(r_n^2)
    =
    2L_\delta
    -
    o(r_n^2).
\end{equation}
Thus a one-step pilot has error at most
$\delta\exp\{o(r_n^2)\}$, which is the accuracy needed below.

We first consider an interior index $i\in\{2,\ldots,n-1\}$. Let
\[
    U_i
    :=
    \mathbf 1\{\eta_{i-1}\neq\eta_i\},
    \qquad
    V_i
    :=
    \mathbf 1\{\eta_i\neq\eta_{i+1}\}.
\]
The variables $U_i$ and $V_i$ are independent Bernoulli$(\delta)$.

Conditionally on $\theta_i^\star$, put
\[
    \mu_i
    :=
    \frac{
        \langle\theta,\theta_i^\star\rangle
    }{
        \|\theta_i^\star\|
    }.
\]
Since $|B_i|=O(1)$,
\begin{equation}
\label{eq:mu-rn-proof}
    \mu_i
    =
    r_n(1+o_{\PP}(1)).
\end{equation}

The ideal correction in the unnormalized score is
\[
    \kappa_i^\star
    :=
    \frac{\|\theta_i^\star\|^2}
         {\langle\theta,\theta_i^\star\rangle}
    L_\delta.
\]
Using \eqref{eq:oracle-leave-block},
\[
    \kappa_i^\star
    =
    \frac{a}{r_n^2}
    L_\delta(1+o_{\PP}(1)).
\]
The key point is that the estimate of $\delta$ need not be sharp.
Indeed, \eqref{eq:kappa-consistency-proof} and
\eqref{eq:mu-rn-proof} imply
\begin{equation}
\label{eq:kappa-oracle-est-proof}
    \frac{\widehat\kappa}{\|\theta_i^\star\|}
    =
    \frac{L_\delta}{\mu_i}
    \bigl(1+o_{\PP}(1)\bigr)
    +
    O_{\PP}\!\left(
        \frac1{r_n}
    \right).
\end{equation}
The last term is precisely the effect of the constant-factor uncertainty
in $\widehat\delta$ and is negligible in the Gaussian argument.

Set
\[
    q_0:=\Phi^c(\mu_i).
\]
Absorbing the perturbations in
\eqref{eq:kappa-oracle-est-proof} into the Gaussian arguments, put
\[
    q_+
    :=
    \Phi^c\left(
        \mu_i+
        \frac{L_\delta}{\mu_i}(1+o(1))
        +
        O\!\left(\frac1{r_n}\right)
    \right),
    \qquad
    q_-
    :=
    \Phi^c\left(
        \mu_i-
        \frac{L_\delta}{\mu_i}(1+o(1))
        +
        O\!\left(\frac1{r_n}\right)
    \right).
\]

Conditioning on $(U_i,V_i)$ exactly as in the one-dimensional proof,
\begin{equation}
\label{eq:strong-decomp-proof-new}
    \PP(\mathcal A_i^c\mid\theta_i^\star)
    \leq
    C\left(
        q_+
        +
        q_0^2
        +
        \delta q_0
        +
        \delta^2q_-
    \right).
\end{equation}
Indeed, if $U_i=V_i=0$, both neighboring labels coincide with
$\eta_i$; if $U_i\neq V_i$, the neighbors disagree and the central
classifier is used; and if $U_i=V_i=1$, both neighbors are equal to
$-\eta_i$.

By \eqref{eq:strong-regime-lower-proof} and Gaussian Mills inequalities,
\begin{equation}
\label{eq:strong-tail-relations-proof}
    q_+
    \leq
    \delta\,e^{o(r_n^2)}q_0,
    \qquad
    \delta^2q_-
    \leq
    \delta\,e^{o(r_n^2)}q_0,
\end{equation}
and also
\[
    q_0^2
    \leq
    \delta\,e^{o(r_n^2)}q_0.
\]
Taking expectations and using the sharp supervised bound
\eqref{eq:q1-proof}, all factors $e^{o(r_n^2)}$ are absorbed into a
vanishing deterioration of the Gaussian argument. Therefore
\begin{equation}
\label{eq:strong-local-proof}
    r_n^2
    \sup_{2\leq i\leq n-1}
    \PP(\mathcal A_i^c)
    \leq
    C\delta
    \Phi^c\!\left(
        r_n(1-\varepsilon_n)
    \right).
\end{equation}

For the two boundary indices $i\in\{1,n\}$, the estimator uses only the
central leave-one-out rule. By \eqref{eq:q1-proof}, their total
contribution to the normalized Hamming risk is at most
\[
    \frac{C}{n}
    \Phi^c\!\left(
        r_n(1-\varepsilon_n)
    \right)
    \leq
    C\delta
    \Phi^c\!\left(
        r_n(1-\varepsilon_n)
    \right),
\]
because $\delta\geq1/n$. Equivalently, this is exactly the boundary term
in \eqref{eq:master-proof-new} with $k=1$. Hence
\begin{equation}
\label{eq:strong-result-proof}
    \EE[e_m]
    \leq
    Ce^{-cn\delta}
    +
    C\delta
    \Phi^c\!\left(
        r_n(1-\varepsilon_n)
    \right).
\end{equation}

\noindent
{\bf Bounded signal regime.}
It remains to consider sequences for which $r_n$ stays bounded. In
this regime no sharp Gaussian approximation is needed. The procedure
is initialized with the adaptive estimator
$\widehat\eta_{\rm ad}^{(0)}$ constructed in
Section~\ref{sec:adaptation}. On $\mathcal E_{\rm ad}$,
after fixing the global sign,
\[
    \frac1n
    \langle
        \widehat\eta_{\rm ad}^{(0)},\eta
    \rangle
    \geq c_0
\]
for some universal $c_0>0$.

A constant correlation is sufficient here. Indeed, the projected
direction obtained from the labels
$\widehat\eta_{\rm ad}^{(0)}$ has effective squared signal-to-noise
ratio comparable to $r_n^2$. More precisely, the projection result of
Section~\ref{sec:adaptation} yields, on $\mathcal E_{\rm ad}$,
\[
    c r_n^2
    \leq
    \beta
    \leq
    C r_n^2,
\]
where $\beta$ is the squared signal-to-noise ratio of the corresponding
one-dimensional projected model. We may therefore apply the non-sharp one-dimensional local analysis
with signal strength $\beta\asymp r_n^2$. With
\[
    k
    \asymp
    \frac1{r_n^2}
    \log\left(
        1+\frac{r_n^2}{\delta}
    \right),
\]
the probability that a pilot based on a constant block is incorrect
is bounded by
\(
    Ce^{-c k r_n^2},
\)
whereas the probability that at least one transition occurs in a pilot
window is bounded by $Ck\delta$. Consequently,
 \begin{equation}\label{eq:coarse-result-proof}
      \EE[e(\widehat\eta^{(m)},\eta)]
    \leq
    C\left(
        k\delta
        +
        e^{-c k r_n^2}
    \right)
    +
    Ce^{-cn\delta}.  
\end{equation}
By the definition of $k$,
\[
    k\delta
    +
    e^{-c k r_n^2}
    \leq
    C
    \frac{\delta}{r_n^2}
    \left(
        1+\log\frac{r_n^2}{\delta}
    \right)
    \wedge1.
\]
Since $r_n$ stays bounded,
\(
    \Phi^c\!\left(r_n(1-\varepsilon_n)\right)
    \asymp 1
\)
and the preceding bound has the required order.
Finally, in the bounded-signal regime the first and last $k$
coordinates are treated by the one-sided or central rule. Since
$\delta\geq1/n$, their contribution is absorbed into the preceding
bounds. Combining
\eqref{eq:intermediate-result-proof},
\eqref{eq:strong-result-proof}, and
\eqref{eq:coarse-result-proof} proves the result.
\subsection{Proof of Theorem~\ref{upper:minimax:optimal}}
Throughout the proof, we normalize $\sigma=1$ and write
\[
    a:=\|\theta\|^2,
    \qquad
    r_m^2:=\frac{a^2}{a+d/m},
    \qquad m\geq1.
\]
All Hamming losses are understood up to a global sign.

We first introduce the event $\mathcal E_{\rm ad}$ on which the
adaptive estimators of Section~\ref{sec:adaptation} satisfy
\begin{equation}
\label{eq:adaptive-good-final}
    c\delta
    \leq
    \widehat\delta
    \leq
    C\delta,
    \qquad
    c r_n^2
    \leq
    \widehat r_{\rm ad}^{\,2}
    \leq
    C r_n^2,
\end{equation}
and, whenever $r_n^2\gtrsim\delta$, the preliminary label estimator
has a nonvanishing correlation with $\eta$. The results of
Section~\ref{sec:adaptation} imply
\begin{equation}
\label{eq:adaptive-good-proba-final}
    \PP(\mathcal E_{\rm ad}^c)
    \leq
    Ce^{-cn\delta}.
\end{equation}
Since the clustering loss is bounded by one, the contribution of
$\mathcal E_{\rm ad}^c$ is therefore absorbed by the first term in the
statement.

\medskip
\noindent
{\bf Sharp estimation of the signal direction.}
We first record a consequence of the spectral initialization of
\cite{Ndaoud18}. Let $T\subset\llbracket n\rrbracket$ be deterministic
with $m:=|T|$, let
\[
    \widetilde\eta_{\rm sp}^T
    =
    \sign(\widehat v_T),
\]
where $\widehat v_T$ is a leading eigenvector of
\(
    H(X_T^\top X_T),
\)
and, after fixing its global sign, define
\[
    e_T
    :=
    \frac1m
    d_H\left(
        \widetilde\eta_{\rm sp}^T,\eta_T
    \right).
\]
Whenever $r_m\to\infty$, the spectral initialization result of
\cite{Ndaoud18}, which is uniform over the label vector, yields an
event $\mathcal E_{\rm sp}(T)$ such that
\begin{equation}
\label{eq:spectral-labels-final}
    e_T=o(1)
\end{equation}
on $\mathcal E_{\rm sp}(T)$, while
\begin{equation}
\label{eq:spectral-event-prob-final}
    \PP\bigl(\mathcal E_{\rm sp}(T)^c\bigr)
    \leq
    Ce^{-cm}.
\end{equation}

Define
\[
    \widehat\theta_T
    :=
    \frac1m
    \sum_{j\in T}
    \widetilde\eta_{{\rm sp},j}^T X_j,
    \qquad
    \widehat w_{\rm sp}(T)
    :=
    \frac{\widehat\theta_T}
         {\|\widehat\theta_T\|},
\]
and introduce the corresponding oracle center
\[
    \theta_T^\star
    :=
    \frac1m
    \sum_{j\in T}
    \eta_jX_j
    =
    \theta+\frac1mW_T\eta_T.
\]
Writing
\[
    h_T
    :=
    \widetilde\eta_{\rm sp}^T-\eta_T,
\]
we have
\[
    \|h_T\|^2=4me_T
\]
and
\[
    \widehat\theta_T-\theta_T^\star
    =
    \frac{\langle h_T,\eta_T\rangle}{m}\theta
    +
    \frac1mW_Th_T.
\]
Therefore, on the usual Gaussian operator-norm event,
\[
\begin{split}
    \|
        \widehat\theta_T-\theta_T^\star
    \|
    \leq
    2e_T\sqrt a
    +
    \frac{\|W_T\|_{\op}}m
    \|h_T\|\leq
    2e_T\sqrt a
    +
    C\sqrt{e_T}
    \left(
        1+\sqrt{\frac dm}
    \right).
\end{split}
\]
Since
\(
    r_m^2
    =
    \frac{a^2}{a+d/m},
\)
we have
\[
    \sqrt{a+\frac dm}
    =
    \frac a{r_m},
    \qquad
    r_m\leq\sqrt a,
    \qquad
    \sqrt{\frac dm}\leq\frac a{r_m}.
\]
It follows that
\[
\begin{split}
    \frac{
        \|\widehat\theta_T-\theta_T^\star\|
    }{
        \sqrt{a+d/m}
    }
    &\leq
    2e_T\frac{r_m}{\sqrt a}
    +
    C\sqrt{e_T}
    \left(
        \frac{r_m}{a}
        +
        \frac{r_m}{a}\sqrt{\frac dm}
    \right)
\\
    &\leq
    2e_T
    +
    C\sqrt{e_T}
    \left(
        \frac1{r_m}+1
    \right)
    =
    o(1).
\end{split}
\]
Hence
\begin{equation}
\label{eq:center-oracle-equivalence-final}
    \widehat\theta_T
    =
    \theta_T^\star
    +
    o\left(
        \sqrt{a+d/m}
    \right).
\end{equation}

Conditionally on $\eta_T$,
\[
    \theta_T^\star
    =
    \theta+\frac Z{\sqrt m},
    \qquad
    Z\sim N(0,I_d).
\]
Standard Gaussian concentration gives, whenever $r_m\to\infty$,
\(
    \langle\theta,\theta_T^\star\rangle
    =
    a(1+o(1)),
\)
and
\[
    \|\theta_T^\star\|^2
    =
    a+\frac dm
    +
    o\left(
        a+\frac dm
    \right).
\]
Together with
\eqref{eq:center-oracle-equivalence-final}, this yields
\begin{equation}
\label{eq:sharp-direction-final}
    \left|
        \langle
            \theta,\widehat w_{\rm sp}(T)
        \rangle
    \right|
    =
    \frac{a}{\sqrt{a+d/m}}
    (1-o(1))
    =
    r_m(1-o(1)).
\end{equation}
This property will only be used in regimes where $r_m\to\infty$.

\medskip
\noindent
{\bf Alignment.}
We next isolate the deterministic alignment argument. Let $S$ and
$S'$ be two overlapping intervals carrying local reconstructions
$\widehat\eta$ and $\widehat\eta'$, and let
$\nu,\nu'\in\{-1,1\}$ denote their theoretical orientations. Let
$O\subset S\cap S'$ be the interval used for alignment.

The correct relative orientation is
\[
    \rho^\star:=\nu\nu'.
\]
Let
\[
    \widehat\rho
    \in
    \argmin_{\rho\in\{-1,1\}}
    \sum_{i\in O}
    \1\{
        \widehat\eta_i
        \neq
        \rho\widehat\eta'_i
    \}.
\]
If $\widehat\rho\neq\rho^\star$, then the two correctly oriented
reconstructions must disagree on at least half of $O$. Hence
\begin{equation}
\label{eq:alignment-deterministic-final}
    \frac{|O|}{2}
    \1\{
        \widehat\rho\neq\rho^\star
    \}
    \leq
    \sum_{i\in O}
    \1\{
        \nu\widehat\eta_i\neq\eta_i
    \}
    +
    \sum_{i\in O}
    \1\{
        \nu'\widehat\eta'_i\neq\eta_i
    \}.
\end{equation}
Consequently, if the two local estimators have expected pointwise
error bounded by $R$ on $O$ and $|O|$ is comparable with their block
lengths, then
\begin{equation}
\label{eq:alignment-proba-final}
    \PP(
        \widehat\rho\neq\rho^\star
    )
    \leq CR.
\end{equation}
No independence between the two local reconstructions is required.

If $J$ local reconstructions are aligned successively, then, on the
event that all $J-1$ relative alignments are correct, all of them have
one common global orientation. Therefore,
\begin{equation}
\label{eq:alignment-master-final}
    \EE[
        \ell(\widehat\eta,\eta)
    ]
    \leq
    CJR,
\end{equation}
where the constant also absorbs the local classification errors
themselves. For $J=2$, there is no additional loss beyond a universal
constant.

\medskip
\noindent
{\bf Weak signal: $r_n^2\leq1$.}
If
\(
    r_n^2\leq c\delta
\)
for a sufficiently small universal constant $c$, the desired result is
immediate from the trivial bound by one. We therefore assume
\(
    r_n^2\gtrsim\delta.
\)

Let $A,B,C$ be the three consecutive blocks appearing in the
definition of the estimator, each of cardinality $n/3+O(1)$, and put
\[
    S_1=A\cup B,
    \qquad
    S_2=B\cup C.
\]
The direction used on $S_1$ is estimated exclusively from $C$, whereas
the direction used on $S_2$ is estimated exclusively from $A$.
Consequently, conditional on the latent path, each estimated direction depends only on Gaussian noises outside the block to which the corresponding segmentation rule is applied, and is therefore independent of the Gaussian noises within that block.

Let
\[
    \mu_a
    :=
    \left|
        \langle\theta,\widehat w_a\rangle
    \right|,
    \qquad a=1,2.
\]
By the constant-correlation property of the adaptive initializer,
\[
    \mu_a^2
    \geq
    c r_{n/3}^2,
    \qquad a=1,2,
\]
on the corresponding adaptive good event. Moreover,
\[
    r_{n/3}^2
    =
    \frac{a^2}{a+3d/n}
    \geq
    \frac13 r_n^2,
\]
and therefore
\begin{equation}
\label{eq:weak-effective-signal-final}
    \mu_a^2
    \geq
    c' r_n^2,
    \qquad a=1,2.
\end{equation}

Conditionally on $\widehat w_a$ and the latent path $\eta$, the projected observations satisfy
\[
    \langle X_i,\widehat w_a\rangle
    =
    \nu_a\eta_i\mu_a+\xi_i,
    \qquad
    \xi_i\stackrel{\rm iid}{\sim}N(0,1),
\]
where
\[
    \nu_a
    :=
    \sign\langle\theta,\widehat w_a\rangle.
\]

On $\mathcal E_{\rm ad}$, the radius used by the procedure satisfies
\[
    k_{\rm w}
    \asymp
    \frac1{r_n^2}
    \log\left(
        1+\frac{r_n^2}{\delta}
    \right).
\]
Using the proof of
Theorem~\ref{thm:risk_clust_new_procedure_ub}, with
\eqref{eq:weak-effective-signal-final}, gives for every interior point
of $S_a$
\begin{equation}
\label{eq:weak-local-final}
    \PP\left(
        \widehat\eta_i^{(a)}
        \neq
        \nu_a\eta_i
    \right)
    \leq
    C\left(
        \frac{\delta}{r_n^2}\wedge1
    \right).
\end{equation}
Indeed, only a constant-factor lower bound on the one-dimensional
effective signal is needed in this regime.

The overlap $B$ has cardinality of order $n$. After removing
$O(k_{\rm w})$ observations from each end, its interior still has
cardinality comparable to $n$. The alignment inequality
\eqref{eq:alignment-deterministic-final} therefore gives
\[
    \PP(\text{wrong alignment})
    \leq
    C\left(
        \frac{\delta}{r_n^2}\wedge1
    \right).
\]
All observations away from the two sample boundaries can be
assigned to one of the two reconstructions whose complete local
neighborhood contains that observation.

It remains to control the $O(k_{\rm w})$ observations adjacent to the
two endpoints. Their total contribution to the normalized
Hamming risk is at most $Ck_{\rm w}/n$. Since
\(
    \delta\geq\log n /n,
\)
we have
\[
    \frac{k_{\rm w}}n
    \lesssim
    \frac1{nr_n^2}
    \log\frac1\delta
    \lesssim
    \frac{\delta}{r_n^2}.
\]
We conclude that
\begin{equation}
\label{eq:weak-final-proof}
    \EE[
        \ell(\widehat\eta,\eta)
    ]
    \leq
    C\left(
        \frac{\delta}{r_n^2}\wedge1
    \right)
    +
    Ce^{-cn\delta}.
\end{equation}

\medskip
\noindent
{\bf Bounded signal above one.}
Suppose now that
\(
    1<r_n\leq R_0
\)
for some fixed $R_0<\infty$. The same two-fold construction may be
used. The projected signal is bounded below by a positive constant,
and the one-dimensional segmentation bound gives
\[
    \EE[
        \ell(\widehat\eta,\eta)
    ]
    \leq
    C_{R_0}\delta
    +
    Ce^{-cn\delta}.
\]
Since
\(
    \inf_{1\leq x\leq R_0}
    \Phi^c(x)>0,
\)
this is equivalently
\begin{equation}
\label{eq:bounded-final-proof}
    \EE[
        \ell(\widehat\eta,\eta)
    ]
    \leq
    C\delta\Phi^c(r_n)
    +
    Ce^{-cn\delta}.
\end{equation}
Thus no sharp direction estimate is needed while $r_n$ remains
bounded.

\medskip
\noindent
{\bf Diverging intermediate signal.}
We now assume $r_n \geq R_0$ and eventually
\(
    r_n\to\infty
\)
and consider the intermediate range
\begin{equation}
\label{eq:intermediate-range-final}
    r_n^2
    \leq
    c_+\log(1/\delta),
\end{equation}
where $c_+<2$ is a fixed universal constant.

Let
\[
    L_n
    :=
    \left\lceil
        \log(e+\widehat r_{\rm sp})
    \right\rceil,
    \qquad
    M_n
    :=
    \left\lfloor
        \frac n{L_n}
    \right\rfloor.
\]
Since
\(
    \widehat r_{\rm sp}/r_n
    =
    1+o_{\PP}(1)
\)
when $r_n\to\infty$, we have
\begin{equation}
\label{eq:Ln-Mn-final-proof}
    L_n\asymp\log r_n,
    \qquad
    M_n\asymp\frac n{\log r_n}
\end{equation}
outside an event negligible compared with the final rate.

Let
\(
    S_1,\ldots,S_J
\)
be the overlapping intervals used by the procedure. They satisfy
\[
    |S_j|\asymp M_n,
    \qquad
    |S_j\cap S_{j+1}|\asymp M_n,
    \qquad
    J\asymp L_n.
\]
For each $j$, let
\[
    T_j
    :=
    \llbracket n\rrbracket\setminus S_j,
    \qquad
    m_j:=|T_j|.
\]
Then
\[
    m_j
    =
    n\left(
        1-O\left(\frac1{\log r_n}\right)
    \right),
\]
and
\begin{equation}
\label{eq:r-mj-comparison-final}
    r_{m_j}
    =
    r_n\left(
        1-O\left(\frac1{\log r_n}\right)
    \right).
\end{equation}

Fix $j$ and define
\[
    \widehat w_j
    :=
    \widehat w_{\rm sp}(T_j),
    \qquad
    \mu_j
    :=
    \left|
        \langle\theta,\widehat w_j\rangle
    \right|.
\]
Whenever
\(
    r_{m_j}\to\infty,
\)
\eqref{eq:sharp-direction-final} gives an event $\mathcal E_j$ such that
\[
    \PP(\mathcal E_j^c)
    \leq
    Ce^{-cm_j},
\]
and, on $\mathcal E_j$,
\begin{equation}
\label{eq:mu-j-sharp-final}
    \mu_j
    =
    r_{m_j}(1-o(1))
    =
    r_n(1-o(1)).
\end{equation}
Thus there exists a deterministic sequence $\beta_n$ such that
\[
    \mu_j
    \geq
    r_n(1-\beta_n)
\]
on $\mathcal E_j$ and $\beta_n\to0$ when $r_n \to 0$. This argument is applied separately to each $j$;
no simultaneous control over all folds is required.

On $\mathcal E_{\rm ad}$, the segmentation radius satisfies
\begin{equation}
\label{eq:ksharp-final}
    k_{\rm s}
    \asymp
    \frac1{r_n^2}
    \log\frac{r_n^2}{\delta}.
\end{equation}
Under
\eqref{eq:intermediate-range-final} and
$\delta\geq(\log n)/n$,
\[
    k_{\rm s}
    \lesssim
    \log n,
\]
whereas
\[
    M_n
    \asymp
    \frac n{\log r_n}.
\]
Since
\(
    r_n^2
    \lesssim
    \log(1/\delta)
    \lesssim
    \log n,
\)
we have
\(
    \log r_n
    =
    O(\log\log n),
\)
and therefore
\begin{equation}
\label{eq:k-over-M-final-proof}
    \frac{k_{\rm s}}{M_n}
    =
    o(1).
\end{equation}
Hence every consecutive overlap contains an interval
$O_j^\circ$ of size comparable to $M_n$ whose points are at distance
at least $k_{\rm s}$ from all artificial boundaries.

Since $\widehat w_j$ is constructed only from $T_j$, it is independent
of the noise terms indexed by $S_j$. Conditionally on
$\widehat w_j$ and the latent path $\eta$, the projected observations in $S_j$ therefore are independent. Hence the conditional segmentation analysis applies pathwise, and the resulting bound can subsequently be averaged over the latent Markov chain.

Because $c_+<2$, equations
\eqref{eq:intermediate-range-final} and
\eqref{eq:mu-j-sharp-final} imply, for all sufficiently large $n$,
\[
    \mu_j^2
    <
    2\log(1/\delta).
\]
Theorem~\ref{thm:risk_clust_new_procedure_ub} can therefore be
applied. On $\mathcal E_j$, it yields for every interior point of
$S_j$
\begin{equation}
\label{eq:local-sharp-risk-final}
    \PP\left(
        \widehat\eta_i^{(j)}
        \neq
        \nu_j\eta_i
        \,\middle|\,
        \mathcal E_j
    \right)
    \leq
    C r_n\delta
    \Phi^c\!\left(
        r_n(1-\beta_n)
    \right),
\end{equation}
where
\[
    \nu_j
    :=
    \sign\langle\theta,\widehat w_j\rangle.
\]
On $\mathcal E_j^c$ we use the trivial bound one. Since
$m_j=n(1-o(1))$,
\[
    e^{-cm_j}
    \leq
    Ce^{-cn},
\]
and hence, for every fixed $j$,
\begin{equation}
\label{eq:local-sharp-unconditional-final}
    R_n^{\rm loc}
    :=
    \sup_{i\in S_j^\circ}
    \PP\left(
        \widehat\eta_i^{(j)}
        \neq
        \nu_j\eta_i
    \right)
    \leq
    C r_n\delta
    \Phi^c\!\left(
        r_n(1-\beta_n)
    \right)
    +
    Ce^{-cn}.
\end{equation}

We now align consecutive local reconstructions on the intervals
$O_j^\circ$. By
\eqref{eq:alignment-proba-final},
\[
    \PP(\text{alignment $j$ fails})
    \leq
    C R_n^{\rm loc}.
\]
Summing over the
\[
    J\asymp\log r_n
\]
alignments and using
\eqref{eq:alignment-master-final}, we obtain
\begin{equation}
\label{eq:before-mills-final-proof}
    \EE[
        \ell(\widehat\eta,\eta)
    ]
    \leq
    C r_n\log r_n\,
    \delta
    \Phi^c\!\left(
        r_n(1-\beta_n)
    \right)
    +
    Ce^{-cn\delta}.
\end{equation}
Here the contribution
$J e^{-cn}$ is absorbed by $e^{-cn\delta}$.

The only observations not covered by an interior local window are the
$O(k_{\rm s})$ points adjacent to the two genuine sample boundaries.
Using the pointwise projected rule, their contribution is at most
\[
    C\frac{k_{\rm s}}n
    \Phi^c\!\left(
        r_n(1-\beta_n)
    \right).
\]
Since
\[
    \frac{k_{\rm s}}n
    \lesssim
    \frac1{nr_n^2}
    \log\frac1\delta
    \lesssim
    \frac{\delta}{r_n^2},
\]
this is bounded by the right-hand side of
\eqref{eq:before-mills-final-proof}.

It remains to absorb the factor $r_n\log r_n$. By Mills'
inequalities, if $x\to\infty$, $\gamma_x\to0$, and
$\log A_x=o(x^2)$, then there exists another sequence
$\widetilde\gamma_x\to0$ such that
\[
    A_x
    \Phi^c\!\left(
        x(1-\gamma_x)
    \right)
    \leq
    C
    \Phi^c\!\left(
        x(1-\widetilde\gamma_x)
    \right).
\]
Taking
\[
    A_x=x\log x
\]
and $x=r_n$, we obtain a sequence $\varepsilon_n\to0$ for which
\[
    r_n\log r_n\,
    \Phi^c\!\left(
        r_n(1-\beta_n)
    \right)
    \leq
    C
    \Phi^c\!\left(
        r_n(1-\varepsilon_n)
    \right).
\]
Therefore
\begin{equation}
\label{eq:intermediate-final-proof}
    \EE[
        \ell(\widehat\eta,\eta)
    ]
    \leq
    C\delta
    \Phi^c\!\left(
        r_n(1-\varepsilon_n)
    \right)
    +
    Ce^{-cn\delta}.
\end{equation}

\medskip
\noindent
{\bf Logarithmic and strong signal.}
Fix constants
\(
    0<c_-<c_+<2.
\)
When
\(
    r_n^2
    \geq
    c_-\log(1/\delta),
\)
we may use the iterative estimator of
Theorem~\ref{upper:minimax}.

If
\(
    r_n^2
    \leq
    2\log(1/\delta),
\)
that theorem gives
\[
    \EE[
        \ell(\widehat\eta,\eta)
    ]
    \leq
    Ce^{-cn\delta}
    +
    C
    \frac{\delta}{r_n^2}
    \log\frac{r_n^2}{\delta}
    \Phi^c\!\left(
        r_n(1-\varepsilon_n')
    \right).
\]
But
\(
    r_n^2
    \geq
    c_-\log(1/\delta)
\)
implies
\[
    \frac1{r_n^2}
    \log\frac{r_n^2}{\delta}
    \leq C.
\]
Consequently,
\begin{equation}
\label{eq:logarithmic-final-proof}
    \EE[
        \ell(\widehat\eta,\eta)
    ]
    \leq
    Ce^{-cn\delta}
    +
    C\delta
    \Phi^c\!\left(
        r_n(1-\varepsilon_n')
    \right).
\end{equation}
If
\(
    r_n^2
    \geq
    2\log(1/\delta),
\)
the second branch of
Theorem~\ref{upper:minimax} gives
\eqref{eq:logarithmic-final-proof} directly.

Thus the iterative estimator is already sharp, up to the vanishing
perturbation in the Gaussian argument, throughout the whole
logarithmic-signal region.

\noindent
{\bf Adaptive regime selection.}
It remains to justify the data-driven choice between the cross-fitted
and iterative estimators. Along every sequence such that $r_n\to\infty$,
\[
\frac{\widehat r_{\rm sp}}{r_n}=1+o_{\mathbb P}(1).
\]
Moreover, on the adaptive event,
\[
\log(1/\widehat\delta)=\log(1/\delta)+O(1).
\]

If $\log(1/\delta)\to\infty$, this additive error is
$o(\log(1/\delta))$, so the logarithmic threshold is located up to a
vanishing relative error. Since the cross-fitted estimator is valid for
\[
r_n^2\le c_+\log(1/\delta),
\]
while the iterative estimator is already sharp for
\[
r_n^2\ge c_-\log(1/\delta),
\]
with $c_-<c_\star<c_+$, the two validity regions overlap and either
choice yields the required rate in the transition region.

If instead $\log(1/\delta)=O(1)$, the switching threshold remains
bounded. Since $r_n\to\infty$ and
$\widehat r_{\rm sp}/r_n\to1$ in probability, the selector eventually
chooses the iterative estimator with probability tending to one.
Hence the data-driven regime selection is asymptotically harmless in
all cases.

Combining
\eqref{eq:weak-final-proof},
\eqref{eq:bounded-final-proof},
\eqref{eq:intermediate-final-proof},
and
\eqref{eq:logarithmic-final-proof},
we conclude that
\[
    \EE[
        \ell(\widehat\eta,\eta)
    ]
    \leq
    Ce^{-cn\delta}
    +
    C
    \begin{cases}
    \displaystyle
    \frac{\delta}{r_n^2}\wedge1,
    &
    r_n^2\leq1,
    \\[2.5ex]
    \displaystyle
    \delta
    \Phi^c\!\left(
        r_n(1-\varepsilon_n)
    \right),
    &
    r_n^2>1,
    \end{cases}
\]
where $\varepsilon_n\to0$ along every sequence such that $r_n\to\infty$ when $n\to\infty$.

\end{document}